\documentclass[11pt,reqno]{amsart}
\usepackage{amsmath,amsthm,amsfonts,amssymb,mathrsfs,bm,graphicx,stmaryrd}
\usepackage{mathtools}
\usepackage[usenames,dvipsnames]{color}
\usepackage{dsfont}
\usepackage{multicol}
\usepackage[colorlinks=true,linkcolor=blue]{hyperref}
\usepackage{bbm}
\usepackage[normalem]{ulem}

\usepackage[letterpaper,hmargin=1.0in,vmargin=1.0in]{geometry}
\parskip 	\smallskipamount

\newtheorem{theorem}{Theorem}[section]
\newtheorem{lemma}[theorem]{Lemma}
\newtheorem{corollary}[theorem]{Corollary}
\newtheorem{conjecture}[theorem]{Conjecture}
\newtheorem{proposition}[theorem]{Proposition}

\newtheorem{maintheorem}{Theorem}
\newtheorem{maincoro}[maintheorem]{Corollary}

\def\L{\mathbb{L}}
\def\P{\mathbb{P}}
\def\Z{\mathbb{Z}}
\def\R{\mathbb{R}}
\def\E{\mathbb{E}}

\newcommand{\cE}{\mathcal{E}}

\newcommand{\red}{\textcolor{red}}

\begin{document}
\title[Small Ball Probabilities for LPP]{Small deviation estimates and small ball probabilities for geodesics in last passage percolation}
\author{Riddhipratim Basu}
\address{Riddhipratim Basu, International Centre for Theoretical Sciences, Tata Institute of Fundamental Research, Bangalore, India}
\email{rbasu@icts.res.in}
\author{Manan Bhatia}
\address{Manan Bhatia, International Centre for Theoretical Sciences, Tata Institute of Fundamental Research, Bangalore, India}
\email{mananbhatia1701@gmail.com}

\date{}
\maketitle
\begin{abstract} 
For the exactly solvable model of exponential last passage percolation on
$\Z^2$, consider the geodesic $\Gamma_n$ joining $(0,0)$ and $(n,n)$ for large
$n$. It is well known that the transversal fluctuation of $\Gamma_n$ around
the line $x=y$ is $n^{2/3+o(1)}$ with high probability. We obtain the exponent
governing the decay of the small ball probability for $\Gamma_{n}$ and
establish that for small $\delta$, the probability that $\Gamma_{n}$ is
contained in a strip of width $\delta n^{2/3}$ around the diagonal is $\exp
(-\Theta(\delta^{-3/2}))$ uniformly in high $n$. We also obtain optimal small
deviation estimates for the one point distribution of the geodesic showing
that for $\frac{t}{2n}$ bounded away from $0$ and $1$, we have
$\P(|x(t)-y(t)|\leq \delta n^{2/3})=\Theta(\delta)$ uniformly in high $n$,
where $(x(t),y(t))$ is the unique point where $\Gamma_{n}$ intersects the line
$x+y=t$. Our methods are expected to go through for other exactly solvable
models of planar last passage percolation and, upon taking the $n\to \infty$
limit, provide analogous estimates for geodesics in the directed landscape. 
\end{abstract}
\tableofcontents

\section{Introduction and statement of main results}
Small ball probabilities are fundamental objects associated to stochastic
processes where one asks the following question: what is the probability that
a stochastic process remains within a ball of small radius (in an appropriate
norm) of a fixed function? One extensively studied case in the literature is that of the small
ball probabilities of $C[0,1]$ valued Gaussian processes in the sup norm, of
which Brownian motion and Brownian bridge are paradigm examples. If
$\{B_{t}\}_{t\in [0,1]}$ denotes a standard Brownian motion or the standard
Brownian Bridge, it is well known \cite{Chu48} that 
$\log \P(\sup_{t\in [0,1]}|B(t)|\leq \delta) \sim -\frac{\pi^2}{8}\delta^{-2}$ as $\delta\to
0$. Our objective in this paper is to investigate
the corresponding question for the geodesics in planar last passage percolation models in the Kardar-Parisi-Zhang (KPZ) universality class. 

We primarily work with the exactly solvable model of exponential last passage percolation on $\Z^2$. Let $\{\omega_{v}\}_{v\in \Z^2}$
be a configuration of independent and identically distributed rate one
exponentials associated with the vertices of $\Z^2$. For any $u,v\in \Z^2$ such
that $u$ is co-ordinate wise smaller than $v$, and an up-right path $\gamma$
from $u$ to $v$, we define the passage time of the path $\gamma$, denoted
$\ell(\gamma)$, by 
$$\ell(\gamma):=\sum_{w\in \gamma\setminus \{u,v\}} \omega_{w};$$
i.e., the passage time of a path is the sum of the weights on it
excluding the weight of the initial and final vertices.\footnote{Note that we are
excluding both the initial and final vertices in the computation of
$\ell(\gamma)$ contrary to the more standard definition that includes both the endpoints. This is done for certain technical reasons explained later and our main results remain valid under the standard definition. Indeed, one
can note that all paths $\gamma\colon u\rightarrow v$ share the vertices
$u$ and $v$. Hence, the geodesic is not dependent on whether we include the
weights at $u$ and $v$ in the definition of $\ell(\gamma)$. In fact, for our
purposes, we shall also briefly need to consider a variant of the definition
that excludes only the weight of the last vertex.} The last passage time
between $u$ and $v$, denoted $T_{u,v}$, is the maximum of $\ell(\gamma)$ where
the maximum is taken over all up-right paths from $u$ to $v$. The last passage
time from $\mathbf{0}$ to $\mathbf{n}$ (we shall denote the vertex $(r,r)$ by
$\mathbf{r}$ for $r\in \Z$) will be denoted by $T_{n}$. Note that by the
continuity of the exponential distribution, almost surely, between any two
(ordered) vertices $u$ and $v$,  there exists a unique path attaining the last passage time $T_{u,v}$; this path, denoted $\Gamma_{u,v}$, will be called the geodesic between $u$ and $v$. The geodesic between $\mathbf{0}$ and $\mathbf{n}$ will be denoted by $\Gamma_n$.

Observe that $\Gamma_{n}$ defines a stochastic process on $\llbracket 0,2n
\rrbracket$ ($\llbracket a,b \rrbracket$ will denote the discrete interval
$[a,b]\cap \Z$) in the following way: For $t\in \llbracket 0,2n \rrbracket$,
define $\Gamma_n(t):=x(t)-y(t)$, where $(x(t),y(t))$ is the unique point where
$\Gamma_{n}$ intersects the line $\L_{t}:=\{x+y=t\}$; note that the same
definition also allows us to define
$\Gamma_{u,v}(t)$ for any
$u\in \mathbb{L}_0$ and $v\in \mathbb{L}_n$. Clearly,
$\Gamma_{n}(\cdot)$ is an integer valued stochastic process on $\llbracket
0,2n \rrbracket$ pinned to $0$ at either end and having $\pm 1$ increments,
i.e., taking the same realizations as a  simple random walk bridge on the same
interval. The distribution of $\Gamma_n(\cdot)$ however, is very different.
Unlike the $n^{1/2}$ fluctuations in the case of the SRW bridge, $\Gamma_n(\cdot)$ has
fluctuations of the order $n^{2/3}$ \cite{J00,BSS14}, $2/3$ being the
characteristic KPZ scaling exponent for correlation length. Hence it is natural to consider the small ball probability that $\sup |\Gamma_n(\cdot)|$, usually referred to as the \emph{transversal fluctuation} of the geodesic, is upper bounded by $\delta n^{2/3}$ for some small positive $\delta$. Our first main theorem identifies the exponent governing the decay of this probability. 

\begin{maintheorem}
\label{t:sb}
There exists $\delta_0>0$ and positive constants $C_1,C_2,c_1,c_2$ such that
for all $0<\delta<\delta_0$ and for all $n\geq n_0(\delta)$, we have 
\begin{displaymath}
	C_2e^{-c_2\delta^{-3/2}}\leq \P\left (\sup_{t\in \llbracket 0,2n \rrbracket} |\Gamma_n(t)|\leq \delta n^{2/3}\right) \leq C_1e^{-c_1\delta^{-3/2}}.
\end{displaymath}
\end{maintheorem}




Our second main result concerns small deviations of the one point distribution
$\Gamma_n(\cdot)$. As mentioned before, it is known that if $\frac{t}{2n}$ is bounded away from $0$ and $1$, then $\Gamma_n(t)$ has fluctuations of the order $n^{2/3}$ (if $t$ is close to $0$ or $2n$, the fluctuation is of the order $t^{2/3}$ or $(2n-t)^{2/3}$ respectively, see \cite[Theorem 3]{BSS17B}) and the following theorem addresses the small deviation question for such values of $t$. 

\begin{maintheorem}
\label{t:onepoint}
There exists $\delta_0>0$ such that for all $\epsilon\in (0,1)$ and
$0<\delta<\delta_0$, there exist positive constants $C_3,c_3$ depending
on $\epsilon$ such that for
all $n\geq n_0(\delta,\epsilon)$ and $t\in \llbracket \epsilon n,
(2-\epsilon)n \rrbracket$, we have
\begin{displaymath}
	c_3\delta \leq \P\left (|\Gamma_n(t)|\leq \delta n^{2/3}\right) \leq C_3\delta.
\end{displaymath}
\end{maintheorem}

Notice that it is not necessary to take $\delta<\delta_0$ in the statements of
Theorems \ref{t:sb} and \ref{t:onepoint}; indeed, one can simply adjust the
constants so ensure that the estimates hold for all $\delta$. Further,
although we have stated the theorems for a fixed $\delta$ and $n\to \infty$,
one can also get similar results if $\delta\to 0$ sufficiently slowly with
$n$. For Theorem \ref{t:sb}, it suffices to assume $\delta n^{2/3}\to
\infty$;
see Section \ref{s:disf} for a more detailed discussion on this point. It will
be clear from our proofs that Theorem \ref{t:onepoint} holds for any $\delta
\geq 2n^{-2/3}$ for all $n$ sufficiently large (depending on $\epsilon$). The
factor $2$ is needed to handle the case of odd $t$; for even $t$, one gets the same statement for $\delta \geq n^{-2/3}$ (indeed, for $t$ odd, $\P\left (|\Gamma_n(t)|\leq 1\right)=0$).  Furthermore,
for each $L>0$ and any interval $I \subseteq [-Ln^{2/3},Ln^{2/3}]$ with
$|I| \geq 2$ (the lower bound of $2$ is imposed to make sure that the corresponding probability is not $0$ as before), we have that for some positive constants $c_3,C_3$ depending on
$\epsilon$ and $L$, $\P(|\Gamma_n(t)|\in I)\in n^{-2/3}|I|[c_3, C_3]$ for $n$ sufficiently large. See
Corollary \ref{c:lbgen} for a precise statement.



The $n^{2/3}$ fluctuation suggests the following scaling for $\Gamma_n$ akin to the scaling taking simple random walk to Brownian motion: we define a $C[0,1]$ valued stochastic process $\pi_{n}$ by setting
$$\pi_{n}(s):= n^{-2/3}\Gamma_{n}(2ns)$$
 for $s\in [0,1]$ if $2ns\in \Z$ and extending $\pi_{n}$ by linear interpolation to all of
$[0,1]$. One can show using the estimates in \cite{BSS14,BGZ19,BSS17B} that $\pi_n$ is tight in the topology of uniform convergence (see e.g. \cite[Theorem 1.1 (a)]{HS18} for the corresponding result in Poissonian LPP) and it is expected that there exists a $C[0,1]$ valued stochastic process
$\pi$ such that $\pi_{n}\Rightarrow \pi$ weakly in $C[0,1]$, where the limit $\pi$ corresponds to a
geodesic in the universal space-time scaling limit of the last passage percolation landscape. Such a
result has recently been established for the related model of Brownian last
passage percolation \cite{DOV18}, where the limiting object is called `the
directed landscape' and certain geometric properties of the geodesics therein have been established \cite{DOV18,DSV20}. Even without establishing the existence of weak limit, one can read off certain geometric properties of any possible weak limits of $\pi_n$ from uniform pre-limiting estimates on $\Gamma_n$. In particular, we have the following corollary of Theorems \ref{t:sb} and \ref{t:onepoint}.




\begin{maincoro}
\label{c:as}
Let $\pi$ denote any subsequential weak limit $\pi$ of $\pi_n$ in $C[0,1]$
equipped with the topology of uniform convergence. Then we have:
\begin{enumerate}
\item[(i)] There exists $\delta_0>0$ and positive constants $C_1,C_2,c_1,c_2$
	such that for all $0<\delta<\delta_0$, we have 
$$C_2e^{-c_2\delta^{-3/2}}\leq \P\left (\sup_{s\in [0,1]} |\pi(s)|\leq \delta\right) \leq C_1e^{-c_1\delta^{-3/2}}.$$
\item[(ii)] There exists $\delta_1>0$ such that for all $s\in (0,1)$ and $0<\delta <\delta_1$ there exists $C_3(s),c_3(s)$ (bounded away from $0$ and $\infty$ as long as $s$ is bounded away from $0$ and $1$) with
$$c_3\delta \leq \P\left (|\pi(s)|\leq \delta\right) \leq C_3\delta.$$
\end{enumerate}
\end{maincoro}
One expects that Corollary \ref{c:as} also holds when $\pi$ above is replaced by the geodesic in the directed landscape as constructed in \cite{DOV18} since the proofs of Theorem \ref{t:sb} and Theorem \ref{t:onepoint} are expected to go through for Brownian LPP. To maintain the clarity of exposition, we have not attempted to work out the details in this paper; a detailed discussion of how our methods can be adapted to other integrable models of planar last passage percolation together with a discussion on the directed landscape is provided in Section \ref{s:dis}.

\subsection{Background and related results}
\label{s:bg}
Planar last passage percolation models are believed to belong to the KPZ
universality class for a very general class of passage time distributions and
are predicted to exhibit the universal scaling exponents $1/3$ and $2/3$ for
the passage time and the transversal fluctuation of the geodesic
respectively.  Starting with the breakthrough work of Baik-Deift-Johansson
\cite{BDJ99} which established the exponent $1/3$ (and the Tracy-Widom limit)
for Poissonian LPP, this area has seen a great flurry of activity which has
led to a similar analysis of a number of other exactly solvable models of
planar last passage percolation including exponential and geometric LPP
\cite{Jo00}, and Brownian LPP \cite{Bar01}. Using the connection of
exponential LPP with the Totally Asymmetric Simple Exclusion process, the first
order behaviour of $T_{m,n}$ was already identified by Rost in \cite{Ro81} who
established that for $\gamma$ bounded away from $0$ and $\infty$, almost
surely,
$$n^{-1}T_{\mathbf{0}, (n,\gamma n)}\to (1+\sqrt{\gamma})^{2}.$$
 In \cite{Jo00}, it was shown that $T_{m,n}$ has the same law as the largest
eigenvalue of a certain random matrix ensemble (Laguerre Unitary Ensemble
(LUE)) and that for $\gamma$ bounded away from $0$ and $\infty$, one has that
\begin{equation}
	\label{eqn:-4}
	\gamma^{1/6}(1+\sqrt{\gamma})^{-4/3}n^{-1/3}(T_{\mathbf{0}, (n,\gamma n)}-(1+\sqrt{\gamma})^{2}n)
\end{equation}
\noindent
converges weakly to the GUE Tracy-Widom distribution as $n\to \infty$. For our purposes, we
shall need finite $n$ quantitative results, namely, uniform moderate deviation
estimates for $n^{-1/3}(T_{\mathbf{0}, (n,\gamma n)}-(1+\sqrt{\gamma})^{2}n)$. These are provided in \cite{LR10} using the tridiagonal form of LUE (a non-optimal estimate is also available in \cite{BFP12}).

Using the understanding of the fluctuation of the passage times, one can study
the transversal fluctuations of the geodesic. Under some unproven
assumptions, an upper bound on the transversal fluctuation exponent was proved
for first passage percolation by Newman and co-authors (see e.g.\
\cite{New95}) and a rigorous lower bound was proved in \cite{W98,W98+} for the related model of Brownian
motion in a Poissonian potential; a general argument proving both conditional
upper and lower bounds for FPP later appeared in \cite{Cha11}. Using similar
arguments together with the moderate deviation estimates from \cite{BDJ99},
Johansson \cite{J00} first proved the $2/3$ exponent for geodesics in
Poissonian LPP. In particular, he proved that for every $\epsilon>0$, the
probability that either $\sup_{t} |\Gamma_n(t)| \geq n^{2/3+\epsilon}$ or
$\sup_{t} |\Gamma_n(t)| \leq n^{2/3-\epsilon}$ goes to $0$ as $n\to \infty$
(here $|\Gamma_n(\cdot)|$ is defined similarly as before, but for Poissonian
LPP). The same result was proved for Geometric LPP \cite{Jo03} and the same argument would also provide the same result for exponential LPP using the moderate deviation estimates from \cite{LR10}. The transversal fluctuation exponent of $2/3$ for exponential LPP was also established in \cite{BCS06} using a very different approach involving stationary LPP. 

One point of note regarding \cite{J00} (and other similar results e.g.\ \cite{Cha11}) is that the argument for the upper bound for the transversal fluctuation can be made quantitative, while the lower bound cannot. That is, one can tighten the argument in \cite{J00} and write down an upper bound for $\P(|\Gamma_n(t)|>xn^{2/3})$ for some large but fixed $x$ (indeed, one would even get the optimal result if one uses the optimal moderate deviation estimates from \cite{J00}, see \cite{BGZ19}). However, the lower bound typically involves a union bound over a discretization that is polynomially large in $n$, and hence the same argument cannot be used to get a bound for $\P(\sup_{t}|\Gamma_n(t)|<\delta n^{2/3})$ for a small but fixed $\delta$ (a quantitative but non optimal upper bound for $\P(\sup_{t} |\Gamma_n(t)| \leq n^{2/3-\epsilon})$ appeared in \cite{BGH18}). Theorem \ref{t:sb} therefore requires a somewhat different approach.

This paper falls within the general program of understanding the geodesic
geometry in integrable models of last passage percolation using one point
moderate deviation estimates together with percolation techniques. This
program was initiated in \cite{BSS14} and has been followed up in
\cite{BSS17B,
BHS18, HS18,Zha19,BGHH20}. Consequences of understanding the geodesic geometry
have been further explored in \cite{BG18,BGZ19}. In \cite{BSS14}, among other
results, quantitative upper tail bounds for the transversal fluctuation of
geodesics in Poissonian LPP were proved using moderate deviation estimates and a chaining argument. Although \cite{BSS14} used sub-optimal moderate deviation estimates from \cite{BFP12}, the same argument together with the optimal moderate deviation estimates lead to the optimal upper tail bound for  large transversal fluctuations: 
\begin{equation}
\label{e:tfub}
\P(\sup_{t} |\Gamma_n(t)|>xn^{2/3})\leq e^{-cx^{3}}.
\end{equation}
See \cite[Proposition C.9]{BGZ19} for the corresponding result (obtained by the
same argument) written out in the exponential LPP setting. In the set-up of
Poissonian LPP, \cite{HS18} proved a matching lower bound $\P(\sup_{t}
|\Gamma_n(t)|>xn^{2/3})\geq e^{-c'x^{3}}$ establishing that the exponent is
indeed optimal. As far as we are aware, this result does not explicitly appear
in the literature for exponential LPP, but the arguments are robust and are
expected to go through. 

There is a separate line of works relevant for the current paper involving the
related semi-discrete model of Brownian last passage percolation. Using the one point moderate deviation estimates and a special resampling property (Brownian Gibbs property) exhibited by a
line ensemble associated with Brownian LPP, Hammond
\cite{Ham16,Ham17a,Ham17b,Ham17c} developed a deep understanding of geodesic geometry in Brownian LPP. Using similar techniques, \cite{DOV18}
constructed the scaling limit-- the directed landscape, as mentioned before. A more detailed discussion about the results on the geodesic geometry in the directed landscape requires some definitions and is postponed to Section \ref{s:dis}.

\subsection{Outline of the proofs and new contributions in this paper}
\label{s:outline}
As mentioned before, this paper continues the general program of understanding the geodesic geometry in integrable planar last passage percolation models using the one point moderate deviation estimates. As such, while requiring several new ideas and technical ingredients, we borrow ideas and techniques from the existing
literature \cite{BSS14, BSS17B, BHS18, BG18, BGZ19} and also use certain results closely aligned to a few that have already appeared before. We provide a sketch of our
arguments proving Theorems \ref{t:sb} and \ref{t:onepoint} in this subsection
and point out the connections as well as the new contributions of our work. We discuss the upper and lower bounds in each case separately below. In both the cases, the upper bound turns out to be significantly easier than the lower bound. 

%
%
\subsubsection{Theorem \ref{t:sb}, upper bound}
\label{s:sbub}
The basic idea for the upper bound in Theorem \ref{t:sb} is rather simple: we
first obtain upper tail estimates for the weight of best path from
$\mathbf{0}$ to $\mathbf{n}$ constrained to be in the $\delta n^{2/3}$ strip
and show that for small $\delta$, the probability that it is competitive with
$T_{n}$ has the desired upper bound. Namely, if we let $T_{n}^{\delta}$ denote
the weight of the best path from $\mathbf{0}$ to $\mathbf{n}$ that does not
exit the strip $\left\{|x-y|\leq \delta n^{2/3}\right\}$, then we have the following
proposition: 
\begin{proposition}
\label{mod5}
There exists $C,c>0$ such that for all $\delta$ sufficiently small and all $n$
sufficiently large (depending on $\delta$), we have 
$$\P\left(T_{n}^{\delta}\geq 4n-\frac{C}{\delta}n^{1/3}\right)\leq e^{-c\delta^{-3/2}}.$$
\end{proposition}

The idea behind Proposition \ref{mod5} is the following: one can approximate
$T_{n}^{\delta}$ by a sum of $\Theta(\delta^{-3/2})$ many i.i.d.\ variables,
each of
which roughly corresponds to the passage time across the two smaller sides of a
$\delta^{3/2}n\times \delta n^{2/3}$ rectangle. Owing to the negativity of
the mean of the GUE Tracy-Widom distribution, each of these variables have
mean $4\delta^{3/2}n-c'\delta^{1/2}n^{1/3}$ and have sub-exponential tails at
the scale $\delta^{1/2}n^{1/3}$. Once this is established, the proof of
Proposition \ref{mod5} is a simple application of a Bernstein type inequality
for i.i.d.\ sub-exponential variables. With Proposition \ref{mod5} at our disposal, completing the proof of the
 upper bound in Theorem \ref{t:sb} is easy by using the lower tail estimate
 for $T_n$. The argument proving Proposition \ref{mod5} has already been used
 in the literature several times, sometimes with sub-optimal tails (see e.g.\
 \cite{BG18,BGZ19,BHS18}, and also \cite[Proposition 4.2]{BGHH20}, where the
 optimal exponent was obtained in a more general setting), but its consequence
 for the upper bound of small ball probability had not been noted before as far as we are aware.
In the setting of Poissonian LPP, another relevant work is
\cite{DJP18}, where the mean, fluctuation and
 central limit behaviour is studied for an off-scale analogue of
 $T_n^\delta$, where the strip $\left\{|x-y|\leq \delta n^{2/3}\right\}$ in the
 definition of $T_n^\delta$ is replaced by
 the off-scale strip $\left\{|x-y|\leq n^{2/3-\epsilon}\right\}$. 

The proofs of Proposition \ref{mod5} and the upper bound in Theorem \ref{t:sb}
are provided in Section \ref{s:ub}. 
 
\subsubsection{Theorem \ref{t:sb}, lower bound}
This is the most technical part of our arguments and also the heart of new
technical achievements of this paper. Recall that we are trying to show that
on an event of probability at least $e^{-c\delta^{-3/2}}$, the geodesic
$\Gamma_{n}$ does not exit the strip $\left\{|x-y|\leq \delta
n^{2/3}\right\}$. As we only
require to prove this for sufficiently small $\delta$, we shall instead
consider the following reparametrization for notational convenience. We shall
show that there exists an absolute constant $M$ such that with probability at
least $e^{-c\delta^{-3/2}}$, the geodesic does not exit the the strip
$\left\{|x-y|\leq M \delta n^{2/3}\right\}$. 

The main idea is to construct two favourable events. The first one, called
$\mathbf{Inside}$, shall depend on the inside of the strip $\left\{|x-y|\leq
\delta n^{2/3}\right\}$ and shall ensure that:
\begin{enumerate}
\item[$\bullet$] $T_{n}^{\delta}\geq 4n+\frac{C}{\delta}n^{1/3}$.
\item[$\bullet$] The passage time for any points (not necessarily
	well-separated) inside the
	strip $\left\{|x-y|\leq \delta n^{2/3}\right\}$ is not too much smaller compared to its expectation.
\end{enumerate}
Both these conditions can be shown to hold with probability at least $e^{-c\delta^{-3/2}}$, and as they are both increasing events, the FKG inequality ensures that $\P(\mathbf{Inside})$ satisfies a desired probability lower bound.

The second event is a barrier event, called $\mathbf{Bar}$, which ensures that
there is a barrier straddling both of the longer
sides of the rectangle $\{0\leq x+y \leq 2n\} \cap \{|x-y|\leq \delta
n^{2/3}\}$ such that any path that spends a lot of time inside this barrier
region incurs a penalty. Though the event $\mathbf{Bar}$, as defined, will
depend on the entire complement of the above-mentioned rectangle, it
essentially puts constraints only in a region of width $O(\delta n^{2/3})$
around the rectangle. One can show
that the barrier event $\mathbf{Bar}$ holds with probability at least
$e^{-c\delta^{-3/2}}$, and since, by definition, this is independent of
$\mathbf{Inside}$, the intersection of the two favourable events have the desired probability lower bound of $e^{-c\delta^{-3/2}}$. 

The rest of the argument is to show that on these favourable events, one
indeed has that $\Gamma_{n}$ does not exit the strip $\left\{|x-y|\leq M
\delta n^{2/3}\right\}$, which is achieved by ruling out both short and long
excursions outside this strip. Indeed, if the geodesic has a short excursion
(i.e., the starting and ending point of the excursion is separated by
$O(\delta^{3/2} n)$ in the time direction) outside the strip $\left\{|x-y|\leq
\delta n^{2/3}\right\}$ during which it also exits the wider strip
$\left\{|x-y|\leq M \delta n^{2/3}\right\}$ this segment will then have a very
high transversal fluctuation which would make it uncompetitive with the best
path between the excursion endpoints restricted to be within the strip
$\left\{|x-y|\leq \delta n^{2/3}\right\}$. Long excursions are ruled out using the definition of the barrier event and the fact that the best path inside the strip is ensured to be longer than typical. 

A superficially similar scheme was adapted in \cite{BG18, BGZ19} to lower bound correlations between last passage times; however here we are faced with significant new technical challenges. Among other issues, the barrier event has to be suitably defined so that its probability can be appropriately lower bounded as a function of $\delta$, which requires the introduction of a number of new geometric ingredients. This is one of the primary new contributions in this work. 


\subsubsection{Theorem \ref{t:onepoint}, upper bound}
\label{s:onepointub}
This follows quite easily from existing results in the literature, with the
idea going back to \cite{BHS18}. In fact, the special case of $t=n$ and
$\delta n^{2/3}=1$ of Theorem \ref{t:onepoint} was alluded to in \cite[Remark
2.11]{BHS18} in connection with the so-called midpoint problem where it was
remarked that the probability that $\Gamma_{n}$ passes through
$\frac{\mathbf{n}}{2}$ is $O(n^{-2/3})$. An essentially complete sketch for
the upper bound was provided there, and we adapt same argument for our purposes.

Let us fix $t\in\llbracket \epsilon n, (2-\epsilon)n \rrbracket$ and without loss of
generality let us assume $t$ is even. We consider the points $u_{i}=(i\delta
n^{2/3},-i\delta n^{2/3})$ and $v_i=\mathbf{n}+u_{i}$ for $i\in \llbracket
-\frac{\delta^{-1}}{2}, \frac{\delta^{-1}}{2} \rrbracket$. Let $I_{0}$ denote
the line segment on $\L_{t}$ between the points
$(\frac{t}{2}+\frac{\delta}{2} n^{2/3},\frac{t}{2}-\frac{\delta}{2} n^{2/3})$
and $(\frac{t}{2}-\frac{\delta}{2} n^{2/3},\frac{t}{2}+\frac{\delta}{2} n^{2/3})$ and let $I_{i}=I_{0}+u_{i}$. Clearly by translation invariance, for each $i$,
$\P(|\Gamma_n(t)|\leq \delta n^{2/3})=p$ is equal to the probability $p_{i}$
that $\Gamma_{u_i,v_i}$ intersects $I_{i}$. Now clearly, $\sum_{i}
p_{i}=\delta^{-1}p$ is upper bounded by the expected number of distinct points at
which the geodesics $\Gamma_{u_i,v_i}$ can intersect the line $\L_{t}$. We shall show in
Lemma \ref{one-2.1}, following an argument in \cite{BHS18} that the
latter number is upper bounded by a constant independent of $\delta$, and
this would provide the required upper bound for $p$. 
This proof is completed in Section \ref{s:oneub}. 

\subsubsection{Theorem \ref{t:onepoint}, lower bound}
The idea here is similar to the upper bound, but requires several different ingredients. Using the same notations as
above, we need to show that $\sum p_{i}$ is bounded below away from $0$
independently of $\delta$. It suffices to show that with probability bounded
away from $0$ independently of $\delta$, there exists an $i$ such that the
geodesic $\Gamma_{u_i,v_{i}}$ intersects $I_{i}$. By planarity, and the
ordering of geodesics, it is enough to show the following:

\begin{proposition}
	\label{one0}
Let $a_1,a_2$ denote the
points $(-Mn^{2/3},Mn^{2/3})$ and $(Mn^{2/3},-Mn^{2/3})$
respectively and let $b_1=a_1+\mathbf{n}$ and $b_2=a_2+\mathbf{n}$. Given
$\epsilon\in(0,1)$ and $t\in \llbracket\epsilon n,(2-\epsilon)n\rrbracket$, there exists
	a constant $c>0$ and a large positive constant $M$ depending on $\epsilon$ such that for all
	$n>n_0(\epsilon)$, we have
	\begin{displaymath}
		\mathbb{P}\left( \left\{\Gamma_{a_1,b_1}(t)=
		\Gamma_{a_2,b_2}(t)\right\} \cap\left\{ |\Gamma_{a_1,b_1}(t)|\leq
		2Mn^{2/3}\right\} \right)\geq c.
	\end{displaymath}
\end{proposition}

The proof of Proposition \ref{one0} hinges on constructing favourable geometric
events which force the geodesics to coalesce. While the general scheme adapted to establish this is broadly similar to the one employed in \cite[Proposition 3.1]{BSS17B}, since we require a common point of the geodesics to be located in a restricted region in both space and
time (on the line segment $\L_{t}\cap \{|x-y|\leq 2M n^{2/3}\}$), stronger control on the geometry of the geodesics is required and there are new challenges that we need to overcome. We use Proposition \ref{one0} to complete the proof of the lower bound of Theorem \ref{t:onepoint} in Section \ref{s:onelb} and the proof of Proposition
\ref{one0} is provided in Section \ref{s:one0}.

\subsection*{Organization of the paper}
The rest of this paper is organised as follows. In Section \ref{s:prelim}, we
collect the basic inputs we use in this work including the one point moderate
deviation estimates and their consequences that have appeared in the literature.
In Sections \ref{s:ub} and \ref{s:lb}, we complete the proofs of the upper and lower bounds of Theorem \ref{t:sb} respectively. Section \ref{s:onept} contains the proof of Theorem \ref{t:onepoint}. We finish with a discussion of potential extensions in Section \ref{s:dis}. 

\subsection*{Acknowledgements}
RB is partially supported by a Ramanujan Fellowship (SB/S2/RJN-097/2017) from
the Science and Engineering Research Board, an ICTS-Simons Junior Faculty
Fellowship, DAE project no. 12-R\&D-TFR-5.10-1100 via ICTS, and the Infosys Foundation via the Infosys-Chandrasekharan Virtual Centre for Random Geometry of TIFR. MB acknowledges the support from the Long Term Visiting Students Program (LTVSP) at ICTS.

\section{Moderate deviation estimates and consequences}
\label{s:prelim}



In this section, we recall some of the fundamental estimates about exponential LPP and their consequences that have appeared in the literature. These include moderate deviation estimates for the passage times, and estimates on controlling passage times across parallelograms and transversal fluctuations of geodesics. We shall heavily rely on these estimates throughout this paper.

Before starting, we introduce some notation. For a path
$\gamma\colon u\rightarrow v$, where $u\leq v$ (i.e., $u$ is
coordinate-wise smaller than $v$), we use the notation
\begin{gather}
	\label{eqn:-3}
	\ell(\gamma)=\sum_{w\in \gamma \setminus \{u,v\}} \omega_w, \\
	\label{eqn:-2}
	\underline{\ell}(\gamma)=\sum_{w\in \gamma\setminus \{v\}} \omega_w.
\end{gather}
For a point $(x_1,y_1)\in
\mathbb{R}^2$, we will often use the change of co-ordinates
\begin{gather}
	\label{eqn:-1}
	\phi( (x_1,y_1))= x_1+y_1,\\
	\psi( (x_1,y_1)) = x_1-y_1.
\end{gather}
Keeping in line with the literature, $\phi(\cdot)$ and $\psi(\cdot)$ will be called the time coordinate and the space coordinate of a point respectively. For points $u,v\in \mathbb{Z}^2$ with
$u\leq v$, we use $T_{u,v}$ to denote the last passage time from $u$ to $v$,
calculated by using weights given by $\ell$, i.e.,
$$T_{u,v}=\max_{\gamma:u\to v} \ell(\gamma).$$
We shall also have brief occasions to use a variant of the above definition of
last passage time defined by replacing $\ell$ by $\underline{\ell}$ in the
above display: this will be denoted by $\underline{T}_{u,v}$. Clearly,
$\underline{T}_{u,v} \geq T_{u,v}$; in fact, we have that
$\underline{T}_{u,v}=T_{u,v}+\omega_u$, and this implies that
$\mathbb{E}\underline{T}_{u,v}=\mathbb{E}T_{u,v}+1$.
%
We shall use centered passage times; in general, we use a $\widetilde{\cdot}$ symbol over a variable to denote a centered (by its
mean) variable-- e.g.\
\begin{equation}
	\label{eqn:0.1}
	\widetilde{T}_{u,v}=T_{u,v}-\mathbb{E}T_{u,v}.
\end{equation}
We will use $\mathbb{L}_t$ to denote
the line $\{x+y=t\}\subseteq \mathbb{Z}^2$.
%
%


\subsection{One point moderate deviation estimates and passage times across
parallelograms}
As already mentioned, the correspondence between point to point passage times and the largest eigenvalue of LUE was obtained in \cite{Jo00}. The following sharp moderate deviation estimate for the latter has been obtained in \cite{LR10}.\footnote{The correspondence to LUE holds when the last passage time includes the weights of the endpoints. However, for $m,n$ large the contribution of  the endpoints is negligible and Proposition \ref{mod1} holds for our definition of $T_{\mathbf{0},(m,n)}$ (and also $\underline{T}_{\mathbf{0},(m,n)}$).}
%
%

\begin{proposition}[{\cite[Theorem 2]{LR10}}]
\label{mod1}
	For each $\eta>1$, there exist $C,c>0$ depending on $\eta$ such that
	for all $m,n$ sufficiently large with $\eta^{-1}<\frac{m}{n}<\eta$ and
	all $y>0$, we have the following:
	\begin{enumerate}
		\item[(i)] $\mathbb{P}(T_{\mathbf{0},(m,n)}-(\sqrt{m}+\sqrt{n})^2\geq yn^{1/3})\leq
			Ce^{-c\min\{y^{3/2},yn^{1/3}\}}$.
		\item[(ii)] $\mathbb{P}(T_{\mathbf{0},(m,n)}-(\sqrt{m}+\sqrt{n})^2\leq -yn^{1/3})\leq
			Ce^{-cy^3}$.
\end{enumerate}
\end{proposition}

Observe that for $m,n$ as above, Proposition \ref{mod1} implies that
\begin{equation}
\label{e:mean}
|\E T_{\mathbf{0},(m,n)}-(\sqrt{m}+\sqrt{n})^2|\leq Cn^{1/3}
\end{equation}
for some positive constant $C$ (depending only on $\eta$). A similar
statement holds for $\E \underline{T}_{\mathbf{0},(m,n)}$ simply because $\E
\underline{T}_{\mathbf{0},(m,n)}=\E T_{\mathbf{0},(m,n)}+1$.

%
%
Proposition \ref{mod1} can be used to control passage times across an on-scale
parallelogram (i.e., a parallelogram whose dimensions in the time and space directions are $n$ and $n^{2/3}$ respectively). Such estimates were first obtained in \cite{BSS14} in the context of Poissonian LPP, and the details for the exponential LPP was worked out in \cite{BGZ19}; we shall quote the latter source. We need to set up some further notation before stating the results. 
	
%
	We use $U_\delta ^n$ for the
rectangle which is defined by 
\begin{equation}
	\label{eqn:1.0.1}
	U_\delta^n=\left\{ -\delta n^{2/3}\leq \psi(u)\leq \delta n^{2/3} \right\}\cap
\left\{ 0\leq \phi(u)\leq 2n \right\}.
\end{equation}
In general, we suppress the dependence on $n$ and simply write
$U_\delta$ for $U^n_\delta$. We will still use the latter notation in case we
need to use the notation with some parameter other than $n$.  
We shall use notations
$U_{\delta}^{n,\mathtt{L}}$ and
$U_{\delta}^{n,\mathtt{R}}$ to denote the left and right line segments of $
U_{\delta}^n$ respectively.\footnote{The standard convention of rotating the picture by 45 degrees counter-clockwise so that time direction moves vertically upwards will often guide our choice of defining ``left" and ``right".} That is, we define
\begin{gather}
	U_{\delta}^{n,\mathtt{L}}=U_{\delta}^{n}\cap \left\{
	\psi(u)=-\delta n^{2/3},
	\right\}\nonumber\\
	\label{eqn:1.10.1}
	U_{\delta}^{n,\mathtt{R}}=U_{\delta}^{n}\cap \left\{
	\psi(u)=\delta n^{2/3}\right\}.
\end{gather}
To reduce clutter, we usually
abbreviate these to just
$U_{\delta}^{\mathtt{L}}$
and $U_{\delta}^{\mathtt{R}}$. Similarly, 
the two short sides of the
parallelogram $U_\delta^n$ are denoted by $\underline{U}_\delta^n$ and
$\overline{U}_\delta^n$ respectively. That is, we define
\begin{gather}
	\underline{U}_\delta^n= U_\delta^n \cap \mathbb{L}_0,\nonumber\\
	\label{eqn:1.10.2}
	\overline{U}_\delta^n= U_\delta^n \cap \mathbb{L}_{2n}. \nonumber
\end{gather}
These are similarly abbreviated to $\underline{U}_\delta$ and
$\overline{U}_\delta$ respectively. 

We will in general be quoting results
from \cite{BGZ19} for passage times across the parallelograms $U_{\Delta}$
for any fixed $\Delta>0$. These results are originally written for $\Delta=1$ but all the
proofs straightforwardly generalize for any $\Delta$ (see \cite[Lemma
C.3, Lemma C.15]{BGZ19}). Thus, we will
directly quote such results for general $\Delta$ and not comment further. The following result controlling the tails of the maximum and minimum passage
time from $\underline{U}_\Delta$ to $\overline{U}_\Delta$ for any fixed
$\Delta>0$ will be crucial for us. 

	\begin{proposition}[{\cite[Theorem 4.2]{BGZ19}}]
		\label{mod2}
		For any $\Delta>0$, there exist constants $C_1,C_2,c_1,c_2$
		depending on $\Delta$ such that for all $r,n$ large enough,
		we have
		\begin{enumerate}
			\item $\mathbb{P}\left( \sup_{u\in \underline{U}_\Delta,v\in
				\overline{U}_\Delta} \widetilde{T}_{u,v} \geq r
				n^{1/3}\right)\leq C_1e^{-c_1 \min\{r^{3/2},rn^{1/3}\}}$.
			\item $\mathbb{P}\left( \inf_{u\in \underline{U}_\Delta,v\in
				\overline{U}_\Delta} \widetilde{T}_{u,v} \leq -r
				n^{1/3}\right)\leq C_2e^{-c_2 r^{3}}$.
		\end{enumerate}
	\end{proposition}

{We would like to point out that there is another slight
discrepancy between Proposition \ref{mod2} stated as above and the
corresponding statement in \cite{BGZ19}, and the same is true for the other
results below quoted from the same source. Indeed, \cite{BGZ19} proves
Proposition \ref{mod2} with $T$ above replaced by $\underline{T}$. As we have pointed out above, the exclusion of one of the endpoints does not change the estimates. For the sake of completeness, we shall explain, just this once, how to get Proposition \ref{mod2} from the corresponding result in \cite{BGZ19}. We shall ignore this issue for the subsequent results quoted in this section with the understanding that similar minor adaptations can be made to work in each of the cases.} 

{As mentioned earlier, observe first that for any $u,v\in \Z^2$,
$\underline{T}_{u,v}=T_{u,v}+\omega_u$ and hence
$\mathbb{E}\underline{T}_{u,v}=\mathbb{E}T_{u,v}+1$. It therefore follows that $\sup_{u,v}
\widetilde{T}_{u,v} \leq \sup_{u,v} \widetilde{\underline{T}}_{u,v}+1$ and
item (1) of Proposition \ref{mod2} is immediate from the corresponding result
for $\underline{T}$. For item (2), let $U_*$ denote the line segment
$\{u:\phi(u)=1, |\psi(u)|\leq \Delta n^{2/3}+1\}$. By using \eqref{e:mean}, one has
that $\mathbb{E}\underline{T}_{u,v}, \mathbb{E}T_{u,v}\in
(4n-Cn^{1/3},4n+Cn^{1/3})$ for large enough $n$ and all $u\in
U_*\cup \underline{U}_\Delta,v\in \overline{U}_\Delta$. Clearly, this implies the
crude bound
\begin{displaymath}
	\inf_{u\in \underline{U}_\Delta,v\in \overline{U}_\Delta} \widetilde{T}_{u,v} \geq
	\inf_{u\in U_*,v\in \overline{U}_\Delta} \widetilde{\underline{T}}_{u,v}-2Cn^{1/3}
\end{displaymath} and applying \cite[Theorem 4.2]{BGZ19} to the RHS above immediately gives item (2).}

%
	
	We will also require a version of Proposition \ref{mod2} for passage times of highest weight paths restricted to be in some parallelogram. For any $u\leq v$ and a region $G\subseteq \mathbb{Z}^2$ satisfying $u,v\in G\cup \partial G$, we define the constrained passage time
	\begin{equation}
		\label{eqn:1}
		T_{u,v}^{G}=\sup_{\gamma:u\rightarrow v, \gamma \setminus \left\{
		u,v \right\} \subseteq  G}\ell(\gamma).
	\end{equation}
	
	For constrained last passage times, we define the centered version
	$\widetilde{T}_{u,v}^G=T_{u,v}^G-\mathbb{E}T_{u,v}$. Note that this
	notation is a slight deviation from \eqref{eqn:0.1}, the centering here is done with the mean of the unrestricted passage
	time $T_{u,v}$ instead of $T_{u,v}^G$. We similarly define
	$\underline{T}_{u,v}^G$ and
	$\widetilde{\underline{T}}_{u,v}^G=\underline{T}_{u,v}^G-\mathbb{E}\underline{T}_{u,v}$. Since we will be using the terms
	$T_{u,v}^{U_\delta^n}$ and $\widetilde{T}_{u,v}^{U_\delta^n}$ very often,
	to reduce notational clutter, we
	introduce the shorthand notations
	\begin{gather*}
		T_{u,v}^{\delta}= T_{u,v}^{U^n_\delta},\nonumber\\ 
		\widetilde{T}_{u,v}^{\delta}=\widetilde{T}_{u,v}^{U^n_\delta}.
	\end{gather*}
We need the following tail estimates for constrained passage times between well separated points in an $n\times n^{2/3}$ rectangle.

	\begin{proposition}[{\cite[Theorem 4.2]{BGZ19}}]
		\label{mod3}
		For any $\Delta>0$, there exist constants $C_1,C_2,c_1,c_2$
		depending on $\Delta$ such that for any $L>0$ and
		all $r,n$ large enough depending on $L$,
		we have
		\begin{enumerate}
			\item $\mathbb{P}\left( \sup_{u,v\in
				U_\Delta,\phi(v)-\phi(u)\geq \frac{n}{L}}
				\widetilde{T}_{u,v}^{\Delta} \geq r
				n^{1/3}\right)\leq C_1e^{-c_1 \min\{r^{3/2},rn^{1/3}\}}$.
			\item $\mathbb{P}\left( \inf_{u,v\in
				U_\Delta,\phi(v)-\phi(u)\geq \frac{n}{L}}
				\widetilde{T}_{u,v}^{\Delta} \leq -r
				n^{1/3}\right)\leq C_2e^{-c_2 r}$.
		\end{enumerate}
	\end{proposition}

	In Proposition \ref{mod2}, we allowed the two points $u,v$ to vary on the
	shorter sides of the parallelogram $U_\Delta$. We now state an
	analogous result from \cite{BGZ19} where the points $u,v$ vary on the long sides of
	$U_\Delta$ (i.e., $U_\Delta^\mathtt{L}$ or $U_\Delta^\mathtt{R}$), and are thus allowed to be arbitrarily close to each other in
	the time direction.
	\begin{proposition}[{\cite[Lemma C.16]{BGZ19}}]
		\label{mod3.1}
		For any $\Delta>0$, there exist constants $C,c$ depending on
		$\Delta$ such
		that for all $r,n$ large enough, we have
		\begin{displaymath}
			\mathbb{P}\left( \inf_{u,v\in U_\Delta^\mathtt{L},\phi(u)\leq
			\phi(v)} \widetilde{T}_{u,v}^{\Delta} \leq -r
				n^{1/3}\right)\leq Ce^{-c r}.
		\end{displaymath}
		The same holds if $U_\Delta^\mathtt{L}$ is replaced by
		$U_\Delta^\mathtt{R}$.
	\end{proposition}

\subsection{Transversal fluctuation estimates}
For both the small ball and the one point estimates, we will require strong estimates on the upper tail of the transversal
fluctuation of the point-to-point geodesic. The following result from \cite{BGZ19} states that
that paths from $\mathbf{0}$ to $\mathbf{n}$ with a transversal fluctuation larger than
$Mn^{2/3}$ incur a loss of order $M^2 n^{1/3}$ in weight with large probability.

\begin{proposition}[{\cite[Proposition 4.7]{BGZ19}}]
	\label{mod4}
	There exist constants $\xi,c_1>0$ such that for all
	$M$ sufficiently large, and all $n$ sufficiently large, the event
	(denoted by $\mathcal{G}$) that there exists a path $\gamma$ from $\mathbf{0}$ to
	$\mathbf{n}$ satisfying $\gamma\nsubseteq U_M$ and $\ell(\gamma)\geq 4n -
	\xi M^2 n^{1/3}$ satisfies 
	\begin{displaymath}
		\mathbb{P}\left( \mathcal{G} \right)\leq e^{-c_1 M^3}.
	\end{displaymath}
\end{proposition} 

\subsection{Lower Bounds}
To obtain the lower bounds in Theorems \ref{t:sb} and Theorem
\ref{t:onepoint}, one needs to show that the probability of certain unlikely
events are nonetheless uniformly bounded away from $0$. We need two results:
one for the upper tail and one for the lower tail. The first result, which is
a strengthening of \cite[Lemma 4.9]{BGZ19}, shows that with probability bounded
away from $0$, last passage times (and constrained last passage times) can be
arbitrarily larger than typical at the fluctuation scale. Although it is a rather straightforward consequence of \cite[Lemma 4.9]{BGZ19}, we believe that it can be potentially useful in other settings and hence state the following lemma separately. 
%
%
	
%
%

\begin{lemma}
	\label{4.12+}
	For any $\Delta>0$, there exist constants $C,c > 0$ depending on
	$\Delta$ such that for
	every $x > 0$, we have
for all sufficiently large $n$ (depending on $x$)
\begin{displaymath}
	\mathbb{P}\left( \inf_{u\in\underline{U}_\Delta,v\in
			\overline{U}_\Delta}T_{u,v}^{\Delta}\geq 4n+xn^{1/3}
	\right)\geq Ce^{-cx^{3/2}}.
\end{displaymath}
\end{lemma}
\begin{proof}
	We will define three independent events $A_1,A_2,A_3$ such that the
	event in question is a sub-event of $A_1 \cap A_2 \cap A_3$ and
	$\mathbb{P}(A_1\cap A_2\cap A_3)\geq Ce^{-cx^{3/2}}$. Define
	\begin{gather*}
		A_1=\left\{  \inf_{u\in\underline{U}_\Delta}T^{\Delta}_{u,\mathbf{n}/4} \geq 
	n -\frac{x}{2} n^{1/3}\right\}, \\ 
	A_2 = \left\{  T_{\mathbf{n}/4,3\mathbf{n}/4}^{\Delta}\geq 2n+2xn^{1/3}
	\right\}, \\ 
	A_3=\left\{  \inf_{v\in\overline{U}_\Delta}T^{\Delta}_{3\mathbf{n}/4,v} \geq 
	n -\frac{x}{2} n^{1/3}\right\}.
	\end{gather*}
	The independence of $A_1,A_2,A_3$ is clear by definition. Also, by
	\cite[Lemma 4.9]{BGZ19}
	we have $\mathbb{P}(A_2)\geq C_1e^{-c_1x^{3/2}}$ for
	$n$ large enough depending on $x$. Observe now that by \eqref{e:mean}, we
	have that there exists a large enough constant $C_2$ such that
				\begin{displaymath}
					A_1\supseteq \left\{
					\inf_{u\in\underline{U}_\Delta}\widetilde{T}^{\Delta}_{u,\mathbf{n}/4} \geq 
	-\left(\frac{x}{2}-C_2\right) n^{1/3} \right\}
				\end{displaymath} and analogously 
				\begin{displaymath}
					A_3\supseteq \left\{
	\inf_{v\in\overline{U}_\Delta}\widetilde{T}^{\Delta}_{3\mathbf{n}/4,v} \geq 
	-\left(\frac{x}{2}-C_2\right) n^{1/3} \right\}.
				\end{displaymath}
		 Now on using Proposition
	\ref{mod3}, we get high probability lower bounds for $\mathbb{P}(A_1)$ and
	$\mathbb{P}(A_3)$. Combining this with the independence of $A_1,A_2,A_3$
	along with $\mathbb{P}(A_2)\geq Ce^{-cx^{3/2}}$, we get the needed lower
	bound for the probability of $\mathbb{P}(A_1\cap A_2\cap A_3)$. The fact
	that $A_1\cap A_2\cap A_3$ is a sub-event of the event is question is
	straightforward. Indeed, we have that $T_{u,v}^{\Delta}\geq
	T_{u,\mathbf{n}/4}^{\Delta}+ T_{\mathbf{n}/4,3\mathbf{n}/4}^{\Delta}+
	T_{3\mathbf{n}/4,v}^{\Delta}$ deterministically for any $u,v$ in the respective line
	segments due to the exclusion of the endpoints in the definition of
	$\ell$.
\end{proof}

The next result, quoted from \cite{BGZ19}, provides a lower bound for unlikely events in the lower tail. 
\begin{lemma}[{\cite[Lemma 4.10]{BGZ19}}]
	\label{mod3.2}
	For any $\Delta,M>0$, we have that there exists a constant $c>0$
	(depending on $\Delta,M$) such
	that for all $n$ sufficiently large depending on $\Delta, M$, we have
	\begin{displaymath}
		\mathbb{P}\left( \sup_{u\in \underline{U}_{\Delta},v\in
			\overline{{U}_{\Delta}}} T_{u,v} \leq 4n-M
				n^{1/3}\right)\geq c.
	\end{displaymath}
\end{lemma}

\section{Upper bound for the small ball probability}
\label{s:ub}
%
In this section, we provide the proof of Proposition \ref{mod5} and use the same to establish the upper bound in Theorem \ref{t:sb}.
The first step of the proof of Proposition \ref{mod5} is to divide the
rectangle $U_{\delta}$ into sub-rectangles of size $O(\delta^{3/2}n)\times
\delta n^{2/3}$.\footnote{Throughout the paper, we shall assume without loss
of generality that $\delta n^{2/3}$ and $\delta^{3/2} n$ are {even} integers and ignore rounding issues arising from this, and also several other divisibility issues. This is done merely to reduce notational overhead and the reader can verify that the same arguments go through in the general case, with appropriate additions of floor and ceiling signs.} 
%
Define for any $A>0$,
\begin{equation}
	\label{eqn:0.2}
	U_{\delta,A,i}^n= \left\{ -\delta n^{2/3}\leq
\psi(u)\leq \delta n^{2/3} \right\}\cap \left\{ 2iA\delta^{3/2} n\leq \phi(u)\leq
2(i+1)A\delta^{3/2} n \right\}. 
\end{equation}
The abbreviations
$U_{\delta,A,i},\underline{U}_{\delta,A,i},\overline{U}_{\delta,A,i}$
are defined analogously to the abbreviations for the corresponding
quantities of the parallelogram $U_\delta^n$.

The basic idea for the proof of Proposition \ref{mod5} is that for small
$\delta$, $T_n^{\delta}$ can be approximated by sums of passage times across
$U_{\delta,A,i}$. Indeed, for $A$ sufficiently large, to be chosen
appropriately later, we define 
\begin{equation}
		\label{eqn:1.2}
		Y_i=\sup_{u\in \underline{U}_{\delta,A,i}
		,
		v\in \overline{U}_{\delta,A,i}}\underline{T}_{u,v}^{U_{\delta,A,i}}.
	\end{equation}
Clearly, the $Y_i$ are i.i.d.\ across $i$. Let us define i.i.d.\ variables 
$Z_i=\dfrac{Y_{i}-4A\delta^{3/2}n}{A^{1/3}\delta^{1/2}n^{1/3}}$. The next
result gives information about the mean and upper tail of $Z_{i}$. 

\begin{lemma}
\label{l:zi}
For $A$ sufficiently large, there exist positive constants $c_2,c_3,C_3$
(independent of $\delta$) such that for all $n$ sufficiently large
depending on $\delta$, we have 
\begin{itemize}
\item[(i)] $\E Z_{i}\leq -c_2.$
\item[(ii)] $\P(Z_{i}\geq r)\leq C_3e^{-c_3r}$ for each $r>0$. 
\end{itemize}
\end{lemma}

\begin{proof}
Observe first that (ii) is an immediate consequence of Proposition \ref{mod2},
(i) and \eqref{e:mean}. Note that Proposition \ref{mod2}, as stated, is for the variables
$T$ and not the variables $\underline{T}$. However, as remarked in a
footnote earlier, the effect of the endpoints is negligible and Proposition
\ref{mod2} (and also \eqref{e:mean}) also hold for the variables $\underline{T}$.

To prove (i), we show that $(A\delta^{3/2}n)^{-1/3}(\E Y_{i}-\E
T_{A\delta^{3/2}n})$ can be made arbitrarily small by taking $A$ sufficiently
large. Since $\E T_{A\delta^{3/2}n}\leq
4A\delta^{3/2}n-cA^{1/3}\delta^{1/2}n^{1/3}$ (this is a
consequence of the distributional convergence as in \eqref{eqn:-4}  and the fact that GUE Tracy-Widom
distribution has negative mean, see \cite[Lemma A.4]{BGHH20}) for some $c>0$
and $n$ sufficiently large, (i) follows by choosing $A$ appropriately large.
See \cite[Lemma 4.1]{BGHH20} or \cite[Lemma 2.4]{BHS18} for a complete argument.
\end{proof}

We can now give the proof of Proposition \ref{mod5}.

\begin{proof}[Proof of Proposition \ref{mod5}]
Notice first that
	\begin{equation}
		\label{eqn:1.3}
		T_{n}^{\delta}\leq
		\sum_{i=0}^{\delta^{-3/2}/A-1} Y_i.
	\end{equation}
	Indeed, this is the reason we used $\underline{T}$ instead of $T$ in the definition of $Y_{i}$. 
	Thus, it suffices to show that for some constant $c_1$, we have
	\begin{equation}
		\label{eqn:1.4}
		\mathbb{P}\left( \sum_{i=0}^{\delta^{-3/2}/A-1} Y_i \geq 4n-
		\frac{c_1}{\delta}n^{1/3} \right) \leq Ce^{-c\delta^{-3/2}}.
	\end{equation}
	Recalling the definition of $Z_{i}$, \eqref{eqn:1.4} is equivalent to showing 
	\begin{equation}
	\label{e:zitail}
	\P\left (\frac{1}{\delta^{-3/2}/A} \sum Z_{i} \geq -c_1'\right)\leq Ce^{-c\delta^{-3/2}}.
	\end{equation}
	for some $c_1'>0$. Using Lemma \ref{l:zi}, and choosing $c_1'=c_2/2$ where $c_2$ is as in Lemma \ref{l:zi}, \eqref{e:zitail} is an easy consequence of a Bernstein type concentration inequality for sums of i.i.d.\ variables with sub-exponential tails; see e.g.\ \cite[Corollary 2.8.3]{vershynin}. This completes the proof of the proposition. 
\end{proof}

In view of Proposition \ref{mod5}, the proof of the upper bound in Theorem
\ref{t:sb} is almost immediate. 
\begin{proof}[Proof of Theorem \ref{t:sb}, upper bound]
	With $c_1$ as in the statement of Proposition \ref{mod5}, we know that for
	some constants $C,c,C',c'$ (independent of $\delta$), we have
	\begin{align}
		\label{eqn:3.9}
		\mathbb{P}\left( \Gamma_n \subseteq U_\delta \right)&\leq
		\mathbb{P}\left( T_{n}^{\delta} \geq 4n-
		\frac{c_1}{\delta} n^{1/3} \right)+ \mathbb{P}\left(
		T_{n}\leq 4n - \frac{c_1}{\delta} n^{1/3}
		\right) \nonumber\\ 
		&\leq Ce^{-c\delta ^{-3/2}}+C'e^{-c'\delta ^{-3}}\leq
		2Ce^{-c\delta^{-3/2}}
	\end{align}
	where in the second inequality above, we have used Proposition \ref{mod5} along with the second part of Proposition
	\ref{mod1}. This completes the proof of the upper bound in Theorem \ref{t:sb}.
\end{proof}

Before completing this section, let us give a sketch of an alternative proof of the upper bound in Theorem \ref{t:sb}. This argument hinges on having a lower bound of the probability that geodesics have large transversal fluctuation (at the scale $n^{2/3}$). Such a result is known in Poissonian LPP; see \cite[Proposition 1.4]{HS18}. Even though the same argument should work for exponential LPP with minor modifications, we did not find the result in the literature for exponential LPP and hence will not attempt to write down a complete proof. 

For $i\in\left\{
0,1,\cdots,\delta^{-3/2} \right\}$, define the points
\begin{displaymath}
	a_i=
	i\delta^{3/2} \mathbf{n} -\delta n^{2/3} (1,-1).
\end{displaymath} 
Let
$\Gamma_{a_i,a_{i+1}}$ denote the geodesic joining $a_i$ and $a_{i+1}$. By the planar ordering of the geodesics, we know that each $\Gamma_{a_i,a_{i+1}}$ lies to the left of the geodesic $\Gamma_n$, and hence we have  
\begin{equation}
	\label{eqn:3.21}
	\left\{ \Gamma_n\subseteq U_\delta \right\} \subseteq
	\bigcap_{i=0}^{\delta^{-3/2}-1} \left\{
	\sup_{u\in \Gamma_{a_i,a_{i+1}}}\psi(u) \leq 
	\delta n^{2/3}  \right\}= \bigcap_{i=0}^{\delta^{-3/2}-1} \left\{
	\sup_{u\in\Gamma_{a_i,a_{i+1}}}\psi(u)-(-2\delta n^{2/3}) 
	\leq 3\delta n^{2/3} \right\}.
\end{equation}
Note that the events on the right hand side are independent across $i$ and each event has probability bounded away from $1$ by \cite[Proposition 1.4]{HS18} adapted to the exponential case. The upper bound in Theorem \ref{t:sb} follows.

\section{Lower bound for the small ball probability}
\label{s:lb}


In this section, we will obtain the lower bound in Theorem \ref{t:sb}. As
discussed in the introduction, the strategy is to construct a favourable event
with the requisite lower bound on its probability, on which the small ball event holds. We first define the favourable events, and state the probability lower bounds for them. Then we complete the proof of the lower bound in Theorem \ref{t:sb} assuming these. The proofs of the probability bounds are provided at the end of the section.

\subsection{Construction of  favourable events}
Before coming to the construction of the needed events, we introduce  the
following notation:
\begin{gather}
	\mathtt{Left}^n_\delta=\left\{ \psi(u)<-\delta n^{2/3}
\right\},\nonumber\\
\label{eqn:3.24}
\mathtt{Right}^n_\delta= \left\{
\psi(u)>\delta n^{2/3} \right\}. 
\end{gather}
We shall use $\mathtt{Left}_\delta$ and $\mathtt{Right}_\delta$ as
shorthands for the above, but in case we need to make use of the notations for
something other than $n$, we shall use the more general notation.
 	
We
	define three independent events $\mathbf{Bar}^\mathtt{L}$,
	$\mathbf{Inside}$ and $\mathbf{Bar}^\mathtt{R}$ (we
	use $\mathbf{Bar}$ to denote the event
	$\mathbf{Bar}^\mathtt{R}\cap \mathbf{Bar}^\mathtt{L}$) measurable with respect to the
vertex weights in the regions $\mathtt{Left}_\delta$, $U_\delta$ and
$\mathtt{Right}_\delta$
respectively. As already alluded to in the introduction, for some fixed large constant $M$, we shall show that on the event
$\mathcal{E}=\mathbf{Bar}^\mathtt{L}\cap
\mathbf{Inside} \cap \mathbf{Bar}^\mathtt{R}$, we have that $\Gamma_n\subseteq
U_{M\delta}$, and obtain $\mathbb{P}(\mathcal{E})\geq
Ce^{-c\delta^{-3/2}}$ by lower bounding the probabilities of each constituent event separately; see Figure \ref{fig:lb-basic} for an illustration.

\begin{figure}[htbp!]
	\begin{center}
		\includegraphics[width=0.5\linewidth]{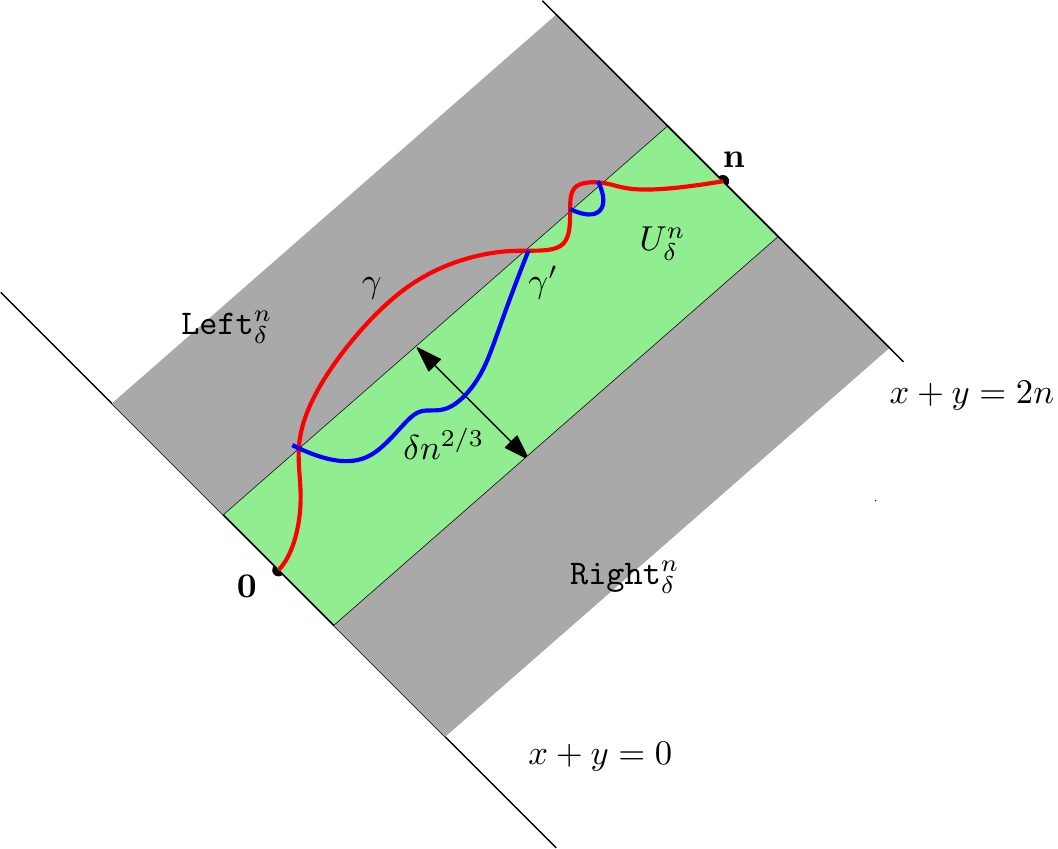}
	\end{center}
	\caption{The basic set-up for the proof of the lower bound in Theorem
	\ref{t:sb}: The region $U_{\delta}$ is shaded in green. The event
	$\mathbf{Inside}$ is a function of the vertex weights in $U_{\delta}$ and
	ensures that any two points $u,v$ in $U^{\delta}$ that differ in the time
	co-ordinate by a large constant times $\delta^{3/2}n$ have large
	$\widetilde{T}^{\delta}_{u,v}$, and even for points $u,v$ that are not
	very well separated in the time coordinate, $\widetilde{T}^{\delta}_{u,v}$
	is not too small. The events $\mathbf{Bar}^\mathtt{L}$ and
	$\mathbf{Bar}^\mathtt{R}$ are functions of the vertex weights on the
	strips to the left and right of $U_{\delta}$ (denoted
	$\mathtt{Left}_{\delta}$ and $\mathtt{Right}_{\delta}$) respectively and
	ensure that for any path $\gamma$ (marked in red) that has either a short
	or a long excursion outside $U_{\delta}$ that exits $U_{M\delta}$ also,
	the excursion can be replaced by a path in the interior of $U_{\delta}$
	(marked in blue) that has higher weight. Thus we ensure that on $\mathbf{Bar}^\mathtt{L}\cap \mathbf{Inside}\cap \mathbf{Bar}^\mathtt{R}$, one has $\Gamma_{n}\subseteq U_{M\delta}$.}
	\label{fig:lb-basic}
\end{figure}

\subsubsection{Choice of parameters}
\label{ss:constants} We shall fix $\delta$ to be sufficiently small
throughout this section. Note that $K,k_0,k_1,\beta,M>0$ will be constants which will appear in the
definitions and will be fixed later in this section; all of them will be independent of $\delta$. 

\noindent We now explicitly point out how the constants are fixed to prevent confusion
later. 
\begin{itemize}
	\item $K$, a large positive constant is obtained by invoking Lemma
		\ref{far1}.
	\item $k_0$, a large positive constant is fixed by invoking Lemma
		\ref{far2}.
	\item $k_1$ is now defined to be $\max\{K,k_0\}+2$.
	\item $\beta$ is an absolute constant not depending on any of the other
		parameters and is obtained
		from the statement of Lemma \ref{short1}. In fact, we have
		$\beta=2^{-2/3}\xi$, where $\xi$ appears in the statement of
		Proposition \ref{mod4}.
	\item The constant $M$ is
		fixed to be large enough compared to all the other parameters fixed so far so that the conclusion of Lemma \ref{short1}
		holds and \eqref{eqn:14.3} in the proof of Lemma \ref{l:close} holds.
\end{itemize}
Note that all the above constants will be
independent of $\delta$, and $n$ will be taken to be sufficiently large
depending on $\delta$ ({and all the other parameters}). Actually, it will be clear from the
proofs that it suffices to take $n\gg \delta^{-3/2}$; we will
comment more about this in Section \ref{s:dis}. For the rest of this section, we shall work with a fixed choice of parameters as described above, $\delta$ sufficiently small and $n$ sufficiently large. 

\subsubsection{The event $\mathbf{Inside}$}
The event $\mathbf{Inside}$ is composed of two parts: (i) $\mathbf{Far}$, which asks that
for any two points $u,v\in U_{\delta}^\mathtt{L}$ (resp.\  $U_{\delta}^{\mathtt{R}}$) which are well separated in the time
direction, $T^{\delta}_{u,v}$ is sufficiently larger than typical; and (ii)
$\mathbf{Close}$, which asks that for any two points $u,v$ in $U_{\delta}^{\mathtt{L}}$ (resp.\  $U_{\delta}^{\mathtt{R}}$) (not necessarily well separated) the constrained centered passage time $\widetilde{T}_{u,v}^{\delta}$ is not much smaller than typical. Let us now move towards defining the events formally. 

\medskip
\noindent
\textbf{Definition of $\mathbf{Far}$:} 
For any positive integer $K$, let $\mathbf{Far}$ denote the event that for any two points $u,v\in
	U_{\delta}^\mathtt{L}$ (or both in $U_{\delta}^\mathtt{R}$) satisfying
	$\phi(v)-\phi(u)\geq K\delta^{3/2}n$, we have
	\begin{equation}
	\label{e:fardef}
		T_{u,v}^{\delta}\geq 2(\phi(v)-\phi(u))+\frac{\phi(v)-\phi(u)}{\delta^{3/2}
	n}(\sqrt{\delta} n^{1/3}).
	\end{equation}
The following lemma provides a lower bound for $\P(\mathbf{Far})$. 		
\begin{lemma}
	\label{far1}
For all $\delta$ sufficiently small, there exists a positive integer $K$
(independent of $\delta$) and positive constants
	$C,c$ (independent of $\delta$) such that
	\begin{displaymath}
		\mathbb{P}(\mathbf{Far})\geq Ce^{-c\delta^{-3/2}}	
	\end{displaymath}
	for all $n$ large enough depending on $\delta$.
\end{lemma}
The proof of Lemma \ref{far1} has been postponed to
Section \ref{ss:far}.



\medskip
\noindent
\textbf{Definition of $\mathbf{Close}$:} 
%
The event $\mathbf{Close}$, as described above will control the $T^{\delta}_{u,v}$ where both $u$ and $v$ are close by points either on $U_{\delta}^{\mathtt{L}}$ or on $U_{\delta}^{\mathtt{L}}$. This will be defined as the intersection of several events indexed by $\mathtt{L}$ or $\mathtt{R}$ depending on which of the boundaries are being considered and also by $i$, which controls the location of the points in the time direction. For $k_1$ sufficiently large (compared to $K$ obtained from Lemma \ref{far1}), a positive absolute constant $\beta$ and 
$i\in\{0,1,\cdots,\frac{1}{k_1 \delta^{3/2}}-2\}$, we define the events $\mathbf{Close}_i^\mathtt{L}$ and $\mathbf{Close}_i^\mathtt{R}$ by setting
\begin{equation}
	\label{eqn:15}
	(\mathbf{Close}_i^\mathtt{L})^c=\left\{	\inf_{u,v\in
	U^{\mathtt{L}}_\delta,2ik_1\delta^{3/2}n\leq \phi(u)\leq \phi(v) \leq
	2(i+2)k_1\delta^{3/2} n }\left\{ T_{u,v}^{1}
	-2(\phi(v)-\phi(u))
		\right\}\leq -\beta M^2 (2k_1 \delta^{3/2} n )^{1/3} \right\},	
\end{equation}

\begin{equation}
	\label{eqn:15.1}
			(\mathbf{Close}_i^\mathtt{R})^c=\left\{	\inf_{u,v\in
			U^\mathtt{R}_\delta,2ik_1\delta^{3/2}n\leq \phi(u)\leq \phi(v) \leq
	2(i+2)k_1\delta^{3/2} n}\left\{ T_{u,v}^{1} -2(\phi(v)-\phi(u))
		\right\}\leq -\beta M^2 (2k_1 \delta^{3/2} n )^{1/3} \right\}.
\end{equation}

We set $\mathbf{Close}_i= \mathbf{Close}_i^\mathtt{L} \cap
\mathbf{Close}_i^\mathtt{R}$ and define 
\begin{equation}
	\label{eqn:15.2}
	\mathbf{Close}=\bigcap_{i=0}^{\frac{1}{k_1 \delta^{3/2}}-2}
	\mathbf{Close}_i.
\end{equation}

We have the following lower bound on $\P(\mathbf{Close})$.

\begin{lemma}
\label{l:close}
There exist positive constants $C,c$ (independent of $\delta$) such that for all $\delta$
sufficiently small, and for the parameters being chosen as in
Section \ref{ss:constants}, we have
$$\P(\mathbf{Close})\geq Ce^{-c\delta^{-3/2}}.$$
\end{lemma}
The proof of Lemma \ref{l:close} has been postponed to the end of Section
\ref{ss:far}.

Finally, we define the event $\mathbf{Inside}$ by
\begin{equation}
	\label{eqn:15.3}
	\mathbf{Inside}= \mathbf{Far} \cap \mathbf{Close}.
\end{equation}
The following lower bound on $\P(\mathbf{Inside})$ easily follows from Lemma \ref{far1} and Lemma \ref{l:close}. 

\begin{lemma}
	\label{inside}
		For all $\delta$ sufficiently small, and for the parameters chosen as in
		Section \ref{ss:constants}, we have that there exist positive constants $C,c$
		(independent of $\delta$) such that for all $n$ large
		enough (depending on $\delta$),
		\begin{displaymath}
			\mathbb{P}\left( \mathbf{Inside} \right)\geq
			Ce^{-c\delta^{-3/2}}.
		\end{displaymath}
\end{lemma}

\begin{proof}
Observe that both $\mathbf{Far}$ and $\mathbf{Close}$ are increasing events
(i.e., for two weight configurations that are point-wise ordered, the event
being satisfied for the smaller weight configuration implies that it is also
satisfied for the larger one), and hence by the FKG inequality,
$$\P(\mathbf{Inside})\geq \P(\mathbf{Far})\P(\mathbf{Close}).$$
The lemma immediately follows from Lemma \ref{far1} and Lemma \ref{l:close}. 
\end{proof}

\subsubsection{The event $\mathbf{Bar}$}
Before proceeding, we first remark that the events $\mathbf{Bar}^\mathtt{L}$
and $\mathbf{Bar}^\mathtt{R}$ will be defined symmetrically about the line
$\left\{ \psi(u)=0 \right\}$. Hence, it suffices to give the details of the
construction of $\mathbf{Bar}^\mathtt{L}$. Our motivation while defining
$\mathbf{Bar}^\mathtt{L}$ is to obtain an environment where paths from
$\mathbf{0}$ to $\mathbf{n}$ which enter the region $\mathtt{Left}_\delta$ incur a loss in weight. To achieve this, $\mathbf{Bar}^\mathtt{L}$ will consist of two
events $\mathbf{Short}^\mathtt{L}$ and $\mathbf{Long}^\mathtt{L}$ which will
give the necessary weight deficits for short and long excursions into the
region $\mathtt{Left}_\delta$ respectively. 

\medskip
\noindent
\textbf{Definition of $\mathbf{Long^{\mathtt{L}}}$:}
For $k_1$ as before, we define
the event $\mathbf{Long}^\mathtt{L}$ as 
\begin{equation}
	\label{eqn:15.5}
	\mathbf{Long}^\mathtt{L}=\left\{\sup_{u,v\in  U_\delta^\mathtt{L}
	:\phi(v)-\phi(u)\geq k_1
	\delta^{3/2} n} \left\{ T_{u,v}^{ \mathtt{Left}_\delta}
	-2(\phi(v)-\phi(u))\right\} <\frac{(\phi(v)-\phi(u))}{\delta^{3/2}
	n}( \sqrt{\delta} n^{1/3})\right\}.
\end{equation}

We have the following lower bound of the probability of the above event. 
\begin{lemma}
	\label{far3}
	 For all
	$\delta$ sufficiently small, and for the parameters chosen as in Section
	\ref{ss:constants}, we have that there exist constants $C,c$ (independent of
	$\delta$) such that for all large enough $n$ (depending on
	$\delta$),  
	\begin{displaymath}
		\mathbb{P}\left( \mathbf{Long}^\mathtt{L} \right)	\ge Ce^{-c\delta^{-3/2}}.
	\end{displaymath}
\end{lemma}
The proof of Lemma \ref{far3} is postponed to Section
\ref{ss:lb1}.

\medskip
\noindent
\textbf{Definition of $\mathbf{Short^{\mathtt{L}}}$:}
Similar to the definition of $\mathbf{Close}^{\mathtt{L}}$,
$\mathbf{Short}^\mathtt{L}$ will also be defined as the intersection of
$\mathbf{Short}^\mathtt{L}_i$, where $i$ shall index the location of the short excursion. 
%
%
For $i\in\{0,1,\cdots,\frac{1}{k_1
\delta^{3/2}}-2\}$, we define $\mathbf{Short}^\mathtt{L}_i$ by setting $(\mathbf{Short}^\mathtt{L}_i)^c$ to be the event
that for some
	$u,v\in
	U_\delta^\mathtt{L}$
	with ${2}ik_1 \delta^{3/2}n\leq \phi(u)\leq \phi(v)\leq {2}(i+2)k_1\delta^{3/2} n$, 
		there exists $\gamma\colon u\rightarrow v$ satisfying
		$\gamma\setminus \left\{ u,v \right\}\subseteq
		\mathtt{Left}_\delta$ and $\gamma \not\subseteq
		(\mathtt{Left}_{M\delta})^c$ such that 
		\begin{equation}
			\label{eqn:15.51}
			\ell(\gamma)>2(\phi(v)-\phi(u))-\beta M^2 (2k_1
		\delta^{3/2}n)^{1/3}.
		\end{equation}
Having defined the events $\mathbf{Short}^\mathtt{L}_i$, we simply
		define
		\begin{equation}
			\label{eqn:15.6}
			\mathbf{Short}^\mathtt{L}=\bigcap_{i=0}^{\frac{1}{k_1
			\delta^{3/2}}-2}\mathbf{Short}^\mathtt{L}_i.
		\end{equation}
We have the following lower bound for $\P(\mathbf{Short}^\mathtt{L})$.
\begin{lemma}
	\label{short2}
		For all $\delta$ sufficiently small, and for the parameters chosen
		as in Section \ref{ss:constants}, we have that there exist positive constants $C,c$
		(independent of $\delta$) such that for $n$ large
		enough depending on $\delta$,
		\begin{displaymath}
			\mathbb{P}\left( \mathbf{Short}^\mathtt{L} \right)\geq
			Ce^{-c\delta^{-3/2}}.
		\end{displaymath}
\end{lemma}
\noindent
The proof of Lemma \ref{short2} is postponed to Section
\ref{ss:lb2}.

\bigskip

Having completed the construction of $\mathbf{Bar}^\mathtt{L}$ by setting 
$$\mathbf{Bar}^\mathtt{L}:=\mathbf{Long}^\mathtt{L}\cap \mathbf{Short}^\mathtt{L}$$
we define
$\mathbf{Bar}^\mathtt{R}$ by symmetry about the line $\{\psi(u)=0\}$. We will not
repeat the details for $\mathbf{Bar}^\mathtt{R}$, but would like to record
that just in the same way as $\mathbf{Bar}^\mathtt{L}$, we also have
\begin{equation}
	\label{eqn:15.7}
	\mathbf{Bar}^\mathtt{R}=  
	\mathbf{Long}^\mathtt{R} \cap \mathbf{Short}^\mathtt{R}
\end{equation}
for analogously defined events 
$\mathbf{Long}^\mathtt{R}$ and $\mathbf{Short}^\mathtt{R}$. 

We have the following lower bound for the probability of $\mathbf{Bar}:=\mathbf{Bar^{\mathtt{L}}}\cap \mathbf{Bar^{\mathtt{R}}}$.

\begin{lemma}
\label{l:bar}
		For all $\delta$ small enough, and for the parameters chosen as in
		Section \ref{ss:constants}, we have that there exist positive constants $C,c$ 
		(independent of $\delta$) such that for all $n$ large
		enough (depending on $\delta$),
		\begin{displaymath}
			\mathbb{P}\left( \mathbf{Bar} \right)\geq
			Ce^{-c\delta^{-3/2}}.
		\end{displaymath}
\end{lemma}
\begin{proof}
	Since $\mathbf{Bar}^\mathtt{L}=\mathbf{Short}^\mathtt{L}\cap
	\mathbf{Long}^\mathtt{L}$, where both the events are decreasing, by
	Lemma \ref{far3}, Lemma \ref{short2} and the FKG inequality, we have
	that for $n$ large enough depending
	on $\delta$,
	\begin{displaymath}
		\mathbb{P}\left( \mathbf{Bar}^\mathtt{L}  \right)\geq
		C_1e^{-c_1\delta^{-3/2}}.
	\end{displaymath}
	By the symmetry about the line $\left\{ \psi(u)=0 \right\}$, we also obtain that
	\begin{displaymath}
		\mathbb{P}\left( \mathbf{Bar}^\mathtt{R}  \right)\geq
		C_1e^{-c_1\delta^{-3/2}}.
	\end{displaymath}
	The lemma now follows immediately by using the independence of
	$\mathbf{Bar}^{\mathtt{L}}$ and $\mathbf{Bar}^{\mathtt{R}}$.
\end{proof}

\subsection{Proofs of Theorem \ref{t:sb}, lower bound and Corollary
\ref{c:as} (i)}

As mentioned before, our main interest is in  the
event $\mathcal{E}$ defined by
\begin{equation}
	\label{eqn:15.8}
	\mathcal{E}=\mathbf{Inside} \cap \mathbf{Bar}.
\end{equation}
We first show that a small ball event is indeed satisfied on the event $\cE$.
\begin{lemma}
	\label{lower1}
	We have the deterministic inclusion $\mathcal{E}\subseteq \left\{
	\Gamma_n\subseteq U_{M\delta}  \right\}$, where the parameters are
	chosen as in Section \ref{ss:constants}. 
\end{lemma}
\begin{proof}
	We prove by contradiction. Clearly if $\Gamma_{n}\subseteq U_{\delta}$, we are done, so let us suppose that for some weight configuration in the event $\mathcal{E}$, we have
	that $\Gamma_n\not\subseteq U_\delta$. First consider the case that
	$\Gamma_n\cap\mathtt{Left}_\delta\neq \emptyset$. In this
	case, there must exist $t_1,t_2\in \mathbb{N}$ satisfying $0\leq
	t_1<t_2\leq 2n$ such that $	\left.\Gamma_n\right|_{[t_1+1,t_2-1]} \subseteq
	\mathtt{Left}_\delta$ and
	$\Gamma_n(t_1)=\Gamma_n(t_2)=-\delta n^{2/3}$. Let us define the points
	$u,v$ by $(\phi(u),\psi(u))=(t_1,\Gamma_n(t_1))$ and
	$(\phi(v),\psi(v))=(t_2,\Gamma_n(t_2))$, that is, $u$ and $v$ are the
	locations of the geodesics at times $t_1$ and $t_2$ respectively. We need to consider two separate cases (refer to Figure \ref{fig:lb-basic}).

\smallskip
\noindent	
	\textbf{Case 1:} $(\phi(v)-\phi(u)) \geq k_1
	\delta^{3/2} n$.
	
	\noindent	
	In this case we reach a contradiction due to the definitions of the
	events
	$\mathbf{Long}^\mathtt{L}$ and $\mathbf{Far}$. Indeed,
	$\mathbf{Long}^\mathtt{L}$ implies that
	\begin{equation}
		\label{eqn:15.9}
		\ell(\left.\Gamma_n\right|_{[t_1,t_2]})< 2(\phi(v)-\phi(u))+\frac{\phi(v)-\phi(u)}{\delta^{3/2}
	n}(\sqrt{\delta} n^{1/3})
	\end{equation}
	while $\mathbf{Far}$ implies (note that $k_1>K$ due to our choice of
	the parameters) that there exists a path
	$\gamma\colon u\rightarrow v$ such that $\gamma \subseteq U_\delta $
	satisfying 
	\begin{equation}
		\label{eqn:15.10}
		\ell(\gamma)\geq 2(\phi(v)-\phi(u))+\frac{\phi(v)-\phi(u)}{\delta^{3/2}
	n}( \sqrt{\delta} n^{1/3}).
	\end{equation}
	It is clear that \eqref{eqn:15.9} and \eqref{eqn:15.10} contradict the
	fact that $\left.\Gamma_n\right|_{[t_1,t_2]}$ is a geodesic from $u$ to
	$v$. 
	
\smallskip
\noindent	
	\textbf{Case 2:} $(\phi(v)-\phi(u)) < k_1
	\delta^{3/2} n$. 

\noindent	
	In this scenario,	
	 there exists an $i_0\in \left\{ 0,1,\cdots,
	\frac{1}{k_1 \delta^{3/2}}-2 \right\}$ satisfying $2i_0 k_1 \delta^{3/2}
	n \leq \phi(u)\leq \phi(v)\leq 2(i_0+2) k_1 \delta^{3/2}n$. Now, the
	definition of the event $\mathbf{Short}^\mathtt{L}_{i_0}$ and
	$\mathbf{Close}$ forces $\left.\Gamma_n\right|_{[t_1,t_2]} \subseteq
	U_{M\delta}$. Indeed, if we had $\left.\Gamma_n\right|_{[t_1,t_2]}\cap
	\mathtt{Left}_{M\delta}\neq \emptyset$, the event $\mathbf{Short}^\mathtt{L}_{i_0}$ would
	imply that
	\begin{equation}
		\label{eqn:15.11}
		\ell(\left.\Gamma_n\right|_{[t_1,t_2]}) \leq 2(\phi(v)-\phi(u))-\beta M^2 (2k_1
		\delta^{3/2}n)^{1/3}.
	\end{equation}
	On the other hand, the event $\mathbf{Close}^\mathtt{L}_{i_0}$ would
	imply that there exists $\gamma\colon u\rightarrow v$ satisfying
	$\gamma\subseteq U_\delta $ along with
	\begin{equation}
		\label{eqn:15.12}
		\ell(\gamma) > 2(\phi(v)-\phi(u))-\beta M^2 (2k_1
		\delta^{3/2}n)^{1/3},
	\end{equation}
	thereby contradicting that $\left.\Gamma_n\right|_{[t_1,t_2]}$ is a
	geodesic between $u$ and $v$. In effect, we have shown that
	$\left.\Gamma_n\right|_{[t_1,t_2]}\subseteq U_{M\delta}$. 	
	
	By an identical
	reasoning and the symmetric definition of the event $\mathbf{Bar}^\mathtt{R}$, we
	can handle the case $\Gamma_n \cap \mathtt{Right}_\delta\neq \emptyset$, and this completes the proof.
\end{proof}

We are now ready to complete the proof of the lower bound in Theorem \ref{t:sb}. 

\begin{proof}[Proof of Theorem \ref{t:sb}, lower bound]
	In view of Lemma \ref{lower1} (we are using that $M$ is a fixed
	constant), we need only show that for all $\delta$
	small enough, we have that there exist positive constants $C,c$
	(independent of $\delta$) such that for all $n$ large enough depending on $\delta$,
	\begin{equation}
		\label{eqn:14.10}
		\mathbb{P}\left( \mathcal{E} \right)\geq Ce^{-c\delta^{-3/2}},
	\end{equation}
		where all the parameters are obtained as described in Section
	\ref{ss:constants}.
	By the definition of the event $\mathcal{E}$
	and the independence of $\mathbf{Bar}$ and
	$\mathbf{Inside}$, we have that
	\begin{equation}
		\label{eqn:14.8}
		\mathbb{P}\left( \mathcal{E} \right)\geq \mathbb{P}\left(
		\mathbf{Bar}\right) \mathbb{P}\left( \mathbf{Inside} \right).
	\end{equation}
		The lower bound for $\mathbb{P}\left( \mathbf{Inside} \right)$ follows
	from Lemma \ref{inside} and the lower bound for $\mathbb{P}\left(
	\mathbf{Bar}
	\right)$ follows from Lemma \ref{l:bar}.
	Piecing these ingredients
	together, we obtain that for $n$ large enough depending on
	$\delta$, \eqref{eqn:14.10} holds.
\end{proof}

Before completing the proofs postponed earlier in this section, we quickly complete the straightforward proof of Corollary \ref{c:as} (i) using Theorem \ref{t:sb}. 
%
%
\begin{proof}[Proof of Corollary \ref{c:as} (i)]
	By using the definition $\pi_{n}(s):= n^{-2/3}\Gamma_{n}(2ns)$ for
	$2ns\in \mathbb{Z}$ along with Theorem \ref{t:sb}, we have 
	\begin{displaymath}
		C_2e^{-c_2\delta^{-3/2}}\leq \P\left (\sup_{s\in [0,1],2ns\in
		\mathbb{Z}}
		|\pi_n(s)|\leq \delta \right) \leq C_1e^{-c_1\delta^{-3/2}}.
	\end{displaymath}
	Since $\pi_n(s)$ is linearly interpolated to all values $s\in [0,1]$ by
	using the values of $\pi_n(s)$ for $s\in [0,1]\cap
	\frac{1}{2n}\mathbb{Z}$, the above immediately implies
	\begin{displaymath}
		C_2e^{-c_2\delta^{-3/2}}\leq \P\left (\sup_{s\in [0,1]}
		|\pi_n(s)|\leq \delta \right) \leq C_1e^{-c_1\delta^{-3/2}},
	\end{displaymath}
	Now, we just need to pass to the limit. To do
	this, first note that the mapping $f\mapsto \sup_{s\in [0,1]}f(s)$ is a
	continuous map from $C[0,1]$ to $\mathbb{R}^+$, where the former is
equipped with the topology of uniform convergence and the latter is
equipped with the Euclidean topology. Consider $\pi$, a
subsequential weak limit $\pi$ of $\pi_n$, that is,
$\pi_{n_i}\Rightarrow\pi$ as $i\rightarrow
\infty$ ($\Rightarrow$ denotes weak convergence) for some subsequence $\left\{ n_i \right\}$. Then by the continuous mapping theorem, we have that
\begin{displaymath}
	\sup_{s\in[0,1]}\pi_{n_i}(s)\Rightarrow \sup_{s\in[0,1]}\pi(s)
\end{displaymath}
as $i\rightarrow \infty$, and we get the result by applying the Portmanteau theorem.
%
\end{proof}
\medskip
The rest of this section is devoted to the proofs of Lemma \ref{far1}, Lemma
\ref{l:close}, Lemma \ref{far3} and Lemma \ref{short2}. Before proceeding
with the proofs, we 
introduce the notation $U_{\delta,i}^n$ for the $i$th rectangle when the strip
$U_\delta^n$ into $\delta^{-3/2}$ many smaller rectangles.
That is, let 
\begin{equation}
	\label{eqn:0}
	U_{\delta,i}^n= \left\{ -\delta n^{2/3}\leq
\psi(u)\leq \delta n^{2/3} \right\}\cap \left\{ 2i\delta^{3/2} n\leq \phi(u)\leq
2(i+1)\delta^{3/2} n \right\}. 
\end{equation}
For convenience, we make the above definition for all
$i\in \mathbb{Z}$ instead of just $i\in\left\{ 0,\dots,\delta^{-3/2}-1
\right\}$, though only the values in $\left\{ 0,\dots,\delta^{-3/2}-1
\right\}$ correspond to subrectangles in $U_\delta^n$. Again, we will simply write $U_{\delta,i}^n$ as $U_{\delta,i}$.
Similar to the definitions \eqref{eqn:1.10.1} for the left and right sides of
$U_\delta^n$, we denote the left and right sides of $U_{\delta,i}^n$ by
$U_{\delta,i}^{n,\mathtt{L}}$ and $U_{\delta,i}^{n,\mathtt{R}}$
respectively. These will be abbreviated to $U_{\delta,i}^{\mathtt{L}}$
and $U_{\delta,i}^{\mathtt{R}}$.

\subsection{Lower bounds for events inside $U_{\delta}$}
\label{ss:far}
This subsection is devoted to the proofs of Lemma \ref{far1} and Lemma \ref{l:close}, i.e., we prove the lower bounds for the probabilities of the events $\mathbf{Far}$ and $\mathbf{Close}$. The first one is more involved and will take up most of this subsection. 

\begin{figure}[t]
	\begin{center}
		\includegraphics[width=0.4\linewidth]{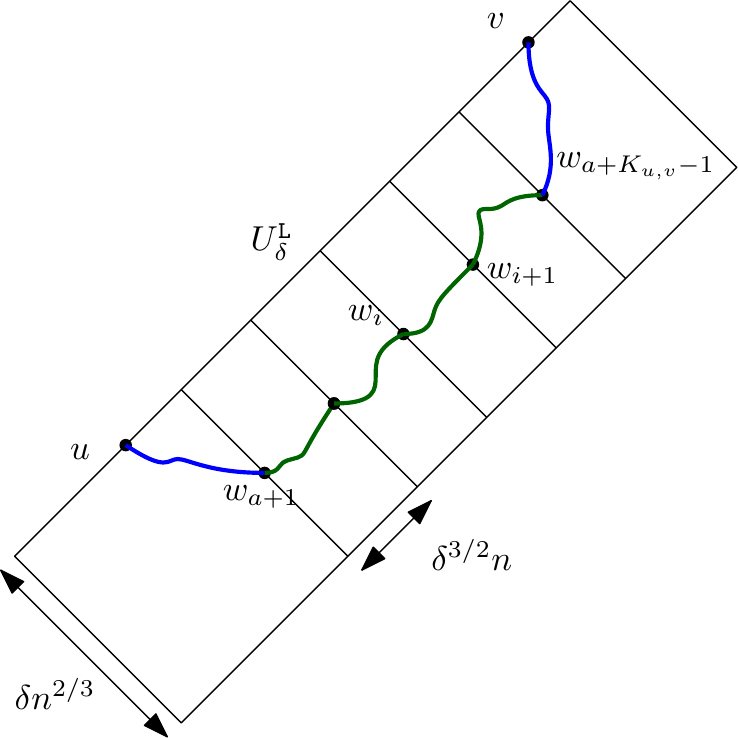}
	\end{center}
	\caption{Proof of Lemma \ref{far1}: for $u,v\in U_{\delta}^{\mathtt{L}}$
	with $\phi(v)-\phi(u)\ge K\delta^{3/2} n$, we lower bound
	$T_{u,v}^{\delta}$ by the weight of the concatenated path showed in the
	figure. The initial and final segments (marked in blue) denote the paths
	attaining weights $T^{\delta}_{u,w_{a+1}}$ and $T^{\delta}_{w_{a+K_{u,v}-1},v}$ respectively; both these paths are ensured to be not too small compared to typical. The intermediate green segments denote the paths attaining $T^{\delta}_{w_{i},w_{i+1}}$ and these are ensured to be larger than typical. For $K_{u,v}$ large, these two conditions ensure that $T_{u,v}^{\delta}$ is larger than typical, as required. }
	\label{fig:far}
\end{figure}

\begin{proof}[Proof of Lemma \ref{far1}]
	We locally define $\mathbf{Far}_\alpha$ to be the event that for any two points $u,v\in
	U_{\delta}^\mathtt{L}$ (or both in $U_{\delta}^\mathtt{R}$) satisfying
	$\phi(v)-\phi(u)\geq K\delta^{3/2}n$, we have
	\begin{displaymath}
		T_{u,v}^{\delta}\geq 2(\phi(v)-\phi(u))+\frac{\phi(v)-\phi(u)}{\delta^{3/2}
		n}(\alpha\sqrt{\delta} n^{1/3}).
	\end{displaymath}
	It suffices to prove that there exists a positive integer $K$,
	positive constants $C,c$ and some $\alpha>1$ (all independent of $\delta$)
	such that
	\begin{displaymath}
		\mathbb{P}\left( \mathbf{Far}_\alpha \right)\geq
		Ce^{-c\delta^{-3/2}}.
	\end{displaymath}
	Indeed, this is because $\mathbb{P}\left( \mathbf{Far}_\alpha \right)$ is
	decreasing in $\alpha$, and $\mathbb{P}(\mathbf{Far})$
	corresponds to $\alpha=1$. We shall
	construct an event $\mathcal{C}$ with the requisite probability lower
	bound such that $\mathcal{C}$ implies $\mathbf{Far}_{\alpha}$ for some
	large $\alpha$ whose value will be chosen later. 
	
	To illustrate our strategy, let us consider two points $u,v$ on $U_{\delta}^\mathtt{L}$ which
	satisfy $\phi(v)-\phi(u)\geq K\delta^{3/2} n$, where $K$ will be
	chosen large later (the case $u,v\in
	U_{\delta}^\mathtt{R}$ can be handled by an identical argument). It is
	immediate that we have
	\begin{equation}
		\label{eqn:1.11}
		\left\lfloor\frac{\phi(v)-\phi(u)}{2\delta^{3/2}n}\right\rfloor\geq
		K/2-1,
	\end{equation}
Thus, there exists a positive integer $K_{u,v}$ such that 
\begin{equation}
	\label{eqn:1.111}
	K_{u,v}\geq K/2-3
\end{equation}
and a positive
	integer $a$ depending on $u,v$ satisfying
	\begin{equation}
		\label{eqn:1.12}
		2(a-1)\leq\frac{\phi(u)}{\delta^{3/2} n} \leq 2a\leq 2(a+1)\leq \cdots\leq
		2(a+ K_{u,v})\leq \frac{\phi(v)}{\delta^{3/2} n}\leq
		2(a+K_{u,v}+1).
	\end{equation}
	 Let $w_i$ denote the point
	 $i\delta^{3/2}\mathbf{n}$. This gives that
	\begin{equation}
		\label{eqn:3}
		T_{u,v}^{\delta}\geq T^{\delta}_{u,
		w_{a+1}}+\sum_{j=2}^{K_{u,v}-1}
		T^{\delta}_{w_{a+j-1},w_{a+j}}+  T^{\delta}_{w_{a+K_{u,v}-1},
		v}.
	\end{equation}
	Our aim now is to construct an event having probability at least
	$Ce^{-c\delta^{-3/2}}$, on which we have that the r.h.s\ in \eqref{eqn:3}
	is larger than typical for each pair $u,v$ satisfying the conditions in the statement of the lemma. We will do this by ensuring that the terms
	$T^{\delta}_{w_{a+j-1},w_{a+j}}$ in \eqref{eqn:3} are larger
	than typical for $j=2$ to $K_{u,v}-3$ while the terms $T^{\delta}_{u,
		w_{a+1}}$ and $T^{\delta}_{w_{a+K_{u,v}-3},
		v}$ are not too small compared to their typical value; see Figure \ref{fig:far}.

		For each $i\in\left\{ 0,\dots,\delta^{-3/2}-2 \right\}$, define the events
	\begin{equation}
		\label{eqn:3.1}
		\mathcal{A}_i=\left\{ T^{\delta}_{w_i,w_{i+1}} \geq
		2(2\delta^{3/2} n) + 4(\alpha \sqrt{\delta} n^{1/3})\right\}.
	\end{equation}
	Note that the $\mathcal{A}_i$ satisfy
	\begin{equation}
		\label{eqn:3.1.1}
			\mathbb{P}(\mathcal{A}_i)\geq Ce^{-c\alpha^{3/2}}>c_3
		\end{equation}
for some $C,c$ (independent of $\delta$) coming from Lemma \ref{4.12+}.
	We now define the events $\mathcal{B}_i$ as follows:
	\begin{equation}
		\mathcal{B}_i=\left\{ \inf_{u,v\in U_{\delta,i}\cup
		U_{\delta,i+1}, \phi(v)-\phi(u)\geq 2\delta^{3/2}n} \widetilde{T}^{\delta}_{u,v} \geq
		-\alpha \sqrt{\delta} n^{1/3}\right\}\label{eqn:3.2}
	\end{equation}
 Note
		that by an application of Proposition \ref{mod3}, we have that for
		some constants $C_2,c_2$,
		\begin{equation}
			\label{eqn:3.2.1}
			\mathbb{P}(\mathcal{B}_i)\geq 1 - C_2 e^{-c_2\alpha}>\frac{1}{2}
		\end{equation}
		 for all $\alpha$ sufficiently large. Finally, we define
		the event $\mathcal{C}$ by
		\begin{equation}
			\label{eqn:3.2.2}
			\mathcal{C}=\bigcap_{i=0}^{\delta^{-3/2}-2} \left(
			\mathcal{A}_i\cap \mathcal{B}_i\right).
		\end{equation}
		Notice now that the events $\mathcal{A}_i$ and
		$\mathcal{B}_i$ are all increasing events measurable with respect to
		the vertex weights in $U_\delta$. By using the FKG inequality
		along with \eqref{eqn:3.1.1} and \eqref{eqn:3.2.1}, we
		have
		\begin{align}
			\label{eqn:3.8}
			\mathbb{P}\left( \mathcal{C} \right)&\geq \left(\prod_{i}
			\mathbb{P}\left( \mathcal{A}_i \right) \prod_{i}\mathbb{P}\left(
			\mathcal{B}_i \right) 
			\right)\nonumber \\ 
			&\geq (c_3/2)^{\delta^{-3/2}-2}\geq e^{-c_5\delta^{-3/2}}
		\end{align}		
		for some $c_5>0$.
		
It remains to prove that $\mathcal{C}\subseteq \mathbf{Far}_{\alpha}$ for some $\alpha$ sufficiently large. We shall only consider the case of $u,v\in U_{\delta}^{\mathtt{L}}$, the case of $u,v\in U_{\delta}^{\mathtt{R}}$ can be handled by an identical argument. On the event $\mathcal{C}$, we have that in \eqref{eqn:3}, for any
		$u,v$ satisfying the conditions in the statement of the lemma,
		\begin{equation}
			\label{eqn:3.3}
			T^{\delta}_{w_{a+j-1},w_{a+j}} \geq  2(2\delta^{3/2} n)
			+ 4(\alpha \sqrt{\delta} n^{1/3})
		\end{equation}
		for all $j\in \left\{ 2,\cdots,K_{u,v}-1 \right\}$. This is a
		straightforward consequence of the definition of the events
		$\mathcal{A}_i$. By using the definition of $\mathcal{B}_i$, we have
		that on $\mathcal{C}$,
		\begin{equation}
			\label{eqn:3.4}
			T^{\delta}_{u,
			w_{a+1}}\geq \mathbb{E}T_{u,
			w_{a+1}}-\alpha \sqrt{\delta} n^{1/3}.
		\end{equation}
	By using \eqref{e:mean}, we have that for all $\alpha$ large enough,
	on $\mathcal{C}$,
	\begin{equation}
		\label{eqn:3.5}
		T_{u,w_{a+1}}^\delta \geq
			2(\phi(w_{a+1})-\phi(u))-\frac{3\alpha}{2}\sqrt{\delta}
			n^{1/3}.
	\end{equation}
	By an analogous argument, we obtain that on
		$\mathcal{C}$, we have for all $\alpha$ large enough,
		\begin{equation}
			\label{eqn:3.6}
			T^{\delta}_{w_{a+K_{u,v}-1},v}\geq
			2(\phi(v)-\phi(w_{a+K_{u,v}-1}))-\frac{3\alpha}{2}\sqrt{\delta}
			n^{1/3}.
		\end{equation}
		On combining \eqref{eqn:3.3}, \eqref{eqn:3.5} and \eqref{eqn:3.6} with
		\eqref{eqn:3}, we deduce that on $\mathcal{C}$, for all $\alpha$
		large enough
		\begin{align}
			\label{eqn:3.7}
			T_{u,v}^{\delta}- 2(\phi(v)-\phi(u)) &\geq\left( 4K_{u,v}-11
			\right)\alpha\sqrt{\delta} n^{1/3} \nonumber\\
			&\geq 3K_{u,v}\alpha \sqrt{\delta} n^{1/3} \nonumber\\
			&\geq \frac{\phi(v)-\phi(u)}{\delta^{3/2}
	n}(\alpha \sqrt{\delta} n^{1/3}),
		\end{align}
		and we now fix $\alpha$ to be one such value which in addition satisfies
		$\alpha>1$.
		Note that we have used $\eqref{eqn:1.111}$ along with the fact that
		$K$ can be fixed to be large to obtain the last
		two inequalities. Thus, we have established that $\mathcal{C}\subseteq
		\mathbf{Far}_{\alpha}$, which together with \eqref{eqn:3.8} completes the proof. 		
\end{proof}
\noindent
We shall now prove Lemma \ref{l:close}.

\begin{proof}[Proof of Lemma \ref{l:close}]
Using Proposition \ref{mod3.1} along with the fact
that $M$ is fixed to be much larger than all the other parameters in Section
\ref{ss:constants}, it is clear that for $n$ large enough
	depending on $\delta$, we have
	\begin{gather}
		\mathbb{P}(\mathbf{Close}_i^\mathtt{L})\geq 1- C_1 e^{-c_1
		M^2}\geq 1/2,\nonumber\\
		\label{eqn:14.3}
		\mathbb{P}(\mathbf{Close}_i^\mathtt{R})\geq 1- C_1
		e^{-c_1 M^2}\geq 1/2.
	\end{gather}
	Since $\mathbf{Close}_i^\mathtt{L}$ and
	$\mathbf{Close}_i^\mathtt{R}$ are increasing events, we have by the FKG
	inequality,
	\begin{equation}
		\label{eqn:14.5}
		\mathbb{P}\left( \mathbf{Close}_i \right)\geq 1/4.
	\end{equation}
	Again, the $\mathbf{Close}_i$ are all increasing events, and by the FKG
	inequality, we have
	\begin{equation}
		\label{eqn:14.6}
		\mathbb{P}\left( \mathbf{Close} \right)\geq \left(1/4
		\right)^{\frac{\delta^{-3/2}}{k_1}-1}\geq C_2 e^{-c_2\delta^{-3/2}},
	\end{equation}
	completing the proof of the lemma.
\end{proof}

\subsection{Lower bounds for the barrier events}
This subsection is devoted to the proofs of Lemma \ref{far3} and Lemma \ref{short2}, i.e., we prove the lower bounds for the probabilities of the events $\mathbf{Long}^{\mathtt{L}}$ and $\mathbf{Short}^{\mathtt{L}}$. 

\subsubsection{Lower bound for the event $\mathbf{Long}^\mathtt{L}$}
\label{ss:lb1}

We need the following result to prove Lemma \ref{far3}. Recall the notation
$U_{\delta,i}$ and  $U_{\delta,i}^\mathtt{L}$ for all $i\in \mathbb{Z}$ from \eqref{eqn:0}.
\begin{lemma}
	\label{far2}
	For all $\delta$ sufficiently small, there exist constants
	$C,c$ (independent of $\delta$) and a positive integer $k_0$ (independent of
	$\delta$) such that for any integers $i,j$
	with $j-i=k\geq k_0$, for all $n$ large enough depending on
	$\delta$, we have that
	\begin{displaymath}
		\mathbb{P}\left( \sup_{u\in U_{\delta,i}^\mathtt{L},v\in
		U_{\delta,j}^\mathtt{L} } \left\{ T^{ \mathtt{Left}_\delta}_{u,v}-2(\phi(v)-\phi(u))
		\right\}\geq  \frac{(\phi(v)-\phi(u))}{\delta^{3/2}
		n}( \sqrt{\delta} n^{1/3}) \right) \leq Ce^{-ck}<1.
	\end{displaymath}
\end{lemma}

\begin{proof}
	Observe that the event in question is measurable with respect to the vertex weights in the region $\mathtt{Left}_{\delta}$ (indeed, this is one of the reasons we chose the definition of $T$ to exclude the weights of the endpoints). Condition on the occurrence of the event in question and refer to it
	locally in this proof as $\mathcal{A}$. Thus on $\mathcal{A}$, there exist
	$u\in U_{\delta,i}^\mathtt{L},v\in
	U_{\delta,j}^\mathtt{L} $ such that  $T^{ \mathtt{Left}_\delta}_{u,v}-2(\phi(v)-\phi(u))
		\geq  \frac{(\phi(v)-\phi(u))}{\delta^{3/2}
		n}( \sqrt{\delta} n^{1/3})$.  Define the points
		$w_1=( (i-1)\delta^{3/2}n,(i-1)\delta^{3/2} n)$ and
	$w_2=( (j+2)\delta^{3/2}n,(j+2)\delta^{3/2} n)$. 
	 Let $\mathcal{A}_1$ be defined
	by
	\begin{displaymath}
		\mathcal{A}_1=\left\{ T_{w_1,u}^{(\mathtt{Left}_\delta)^c}-2(\phi(u)-\phi(w_1)) \geq - \frac{1}{4} \frac{(\phi(v)-\phi(u))}{\delta^{3/2}n}( \sqrt{\delta} n^{1/3})\right\} 
	\end{displaymath}
		and define $\mathcal{A}_2$ by
		\begin{displaymath}
				\mathcal{A}_2=\left\{ T_{v,w_2}^{(\mathtt{Left}_\delta)^c}-
	2(\phi(w_2)-\phi(v)) \geq - \frac{1}{4} \frac{(\phi(v)-\phi(u))}{\delta^{3/2}
		n}(\sqrt{\delta} n^{1/3})\right\}.
		\end{displaymath}
	 Note that $\mathcal{A}_1$ and $\mathcal{A}_2$
		are independent of each other and the vertex weights in the region
		$\mathtt{Left}_\delta$. Also, $\mathcal{A}_1$ and $\mathcal{A}_2$
		together with the conditioning imply the event $\mathcal{A}_3$ which is defined
		by
		\begin{equation}
			\label{eqn:4}
			\mathcal{A}_3=\left\{T_{w_1,w_2}- 2(\phi(w_2)-\phi(w_1))\geq\frac{1}{2} \frac{(\phi(v)-\phi(u))}{\delta^{3/2}
			n}(\sqrt{\delta} n^{1/3})\right\}.
		\end{equation}
			Indeed, this is simply a consequence of $T_{w_1,w_2}\geq
		T_{w_1,u}^{(\mathtt{Left}_\delta)^c}+T^{
		\mathtt{Left}_\delta}_{u,v}+T_{v,w_2}^{(\mathtt{Left}_\delta)^c}$,
		which holds because both the endpoints are excluded in the
		definition of $\ell$.
		Note that $\phi(w_2)-\phi(w_1) = 2(k+3)\delta^{3/2} n $. From this
		discussion, it is clear that 
\begin{equation}
	\label{eqn:5}
	\min_{u\in
	U_{\delta,i}^\mathtt{L}}\mathbb{P}(\mathcal{A}_1)\min_{v\in
	U_{\delta,j}^\mathtt{L}}\mathbb{P}(\mathcal{A}_2)\mathbb{P}(\mathcal{A})\leq
	\mathbb{P}(\mathcal{A}_3)\leq Ce^{-c
	\left(\frac{1}{(k+3)^{1/3}}\frac{\phi(v)-\phi(u)}{\delta^{3/2} n}\right)^{3/2}}\leq
	C_1 e^{-c_1 k}.
\end{equation}
The second inequality follows by using Proposition \ref{mod1} and the
third inequality follows by taking $k_0$ large, observing that
$\phi(v)-\phi(u)\geq
2(k-2) \delta^{3/2} n$.

It remains to establish good lower bounds for $
\min_{u}\mathbb{P}(\mathcal{A}_1)$ and $\min_{v}\mathbb{P}(\mathcal{A}_2)$; by symmetry, we only
deal with the former. Note that by translation invariance and Proposition
\ref{mod3}, we have that for $n$
large enough depending on $\delta$ and for some constants $C_2,c_2$
(independent of $u$ and $\delta$), 
\begin{equation}
	\label{eqn:6}
	\mathbb{P}(\mathcal{A}_1^c)\leq C_2e^{-c_2\frac{\phi(v)-\phi(u)}{\delta^{3/2}
	n}\frac{ \sqrt{\delta}n^{1/3}}{(\phi(u)-\phi(w_1))^{1/3}}}\leq
	C_3e^{-c_3 k}\leq \frac{1}{2}
\end{equation}
by taking $k_0$ to be large enough.
Indeed, \eqref{eqn:6} is obtained by using Proposition \ref{mod3} on the
parallelogram $U_{\delta,i-1}\cup U_{\delta,i}$, using \eqref{e:mean} to
change the centering from $\mathbb{E}T_{w_1,u}$ to $2(\phi(u)-\phi(w_1))$ and noting that
\begin{equation}
	\label{eqn:6.1}
	T_{w_1,u}^{(\mathtt{Left}_\delta)^c}\geq
T_{w_1,u}^{U_{\delta,i-1}\cup U_{\delta,i}}.
\end{equation}
The second inequality in \eqref{eqn:6} follows by using that $(\phi(v)-\phi(u))\geq
2(k-2)\delta^{3/2} n $ along with $2\delta^{3/2} n\leq \phi(u)-\phi(w_1)\leq 4
\delta^{3/2} n$.
Thus, \eqref{eqn:5}, \eqref{eqn:6} and an analogous lower bound on $\min_{v}\mathbb{P}(\mathcal{A}_2)$  immediately yields
\begin{equation}
	\label{eqn:7}
	\mathbb{P}(\mathcal{A})\leq 4C_1e^{-c_1k}
\end{equation}
if we take $k_0$ to be large enough.
\end{proof}
We are now ready to prove Lemma \ref{far3}. 

\begin{proof}[Proof of Lemma \ref{far3}]
	Denote by $\mathcal{A}_{i,j}$ the event whose probability is considered in
	Lemma \ref{far2}. By the choice of $k_1$ and $k_0$ in Section \ref{ss:constants}, we
	have that for $u,v$ as in the definition of $\mathbf{Long}^{\mathtt{L}}$
	we have $\phi(v)-\phi(u)\geq k_1 \delta^{3/2} n \geq
	(k_0+2)\delta^{3/2} n$, and this
		implies that $u\in U_{\delta,i}^\mathtt{L},v\in
		U_{\delta,j}^\mathtt{L}$ for some $i,j$ satisfying $0\leq i\leq j\leq
		\delta^{-3/2}-1$ and $j-i\geq k_0$. Thus, we
		have that
		\begin{equation}
			\label{eqn:8}
			\mathbb{P}\left( \sup_{u,v\in  U_\delta^\mathtt{L}
		:\phi(v)-\phi(u)\geq k_1
		\delta^{3/2} n} \left\{ T_{u,v}^{\mathtt{Left}_\delta}
	-2(\phi(v)-\phi(u))\right\} <\frac{(\phi(v)-\phi(u))}{\delta^{3/2}
	n}( \sqrt{\delta} n^{1/3})\right)\geq \mathbb{P}\left(
	\bigcap_{i,j:j-i\geq k_0} \mathcal{A}_{i,j}^c
	\right).
	\end{equation}
	If we let $k$ be the variable for $j-i$,
		observe that for a fixed value of $k$, there are $\delta^{-3/2} -k$
		pairs of $(i,j)$ satisfying $j-i=k$ and $0\leq i\leq j\leq
		\delta^{-3/2}-1$. Also, by translation invariance,
		we have that $\mathbb{P}(\mathcal{A}_{i,j})$ depends only on $k=j-i$.
		By the FKG inequality for the decreasing events $\mathcal{A}_{i,j}^c$, we have 
		\begin{equation}
			\label{eqn:9}
			 \mathbb{P}\left(
	\bigcap_{i,j:j-i\geq k_0} \mathcal{A}_{i,j}^c
	\right)\geq  \prod_{i,j:j-i\geq k_0} \mathbb{P}\left(
	 \mathcal{A}_{i,j}^c
	 \right)\geq \prod_{k_0\leq k\leq
	 \delta^{-3/2}-1}(1-Ce^{-ck})^{\delta^{-3/2} - k}\geq \prod_{k\geq
	 k_0}(1-Ce^{-ck})^{\delta^{-3/2}}.
		\end{equation}
		Noting that $\prod_{{k\geq
		k_0}}(1-Ce^{-ck})>0$ because $\sum_{k\geq 1} e^{-ck}<\infty$, we have
		the needed result on using \eqref{eqn:8}.
\end{proof}

%
\subsubsection{Lower bound for the event $\mathbf{Short}^\mathtt{L}$}
\label{ss:lb2}
In this subsection, we will obtain the required lower bound for
$\mathbb{P}\left( \mathbf{Short}^\mathtt{L} \right)$ in Lemma \ref{short2}. We first need the following result.

%

\begin{figure}[t]
	\begin{center}
		\includegraphics[width=0.5\linewidth]{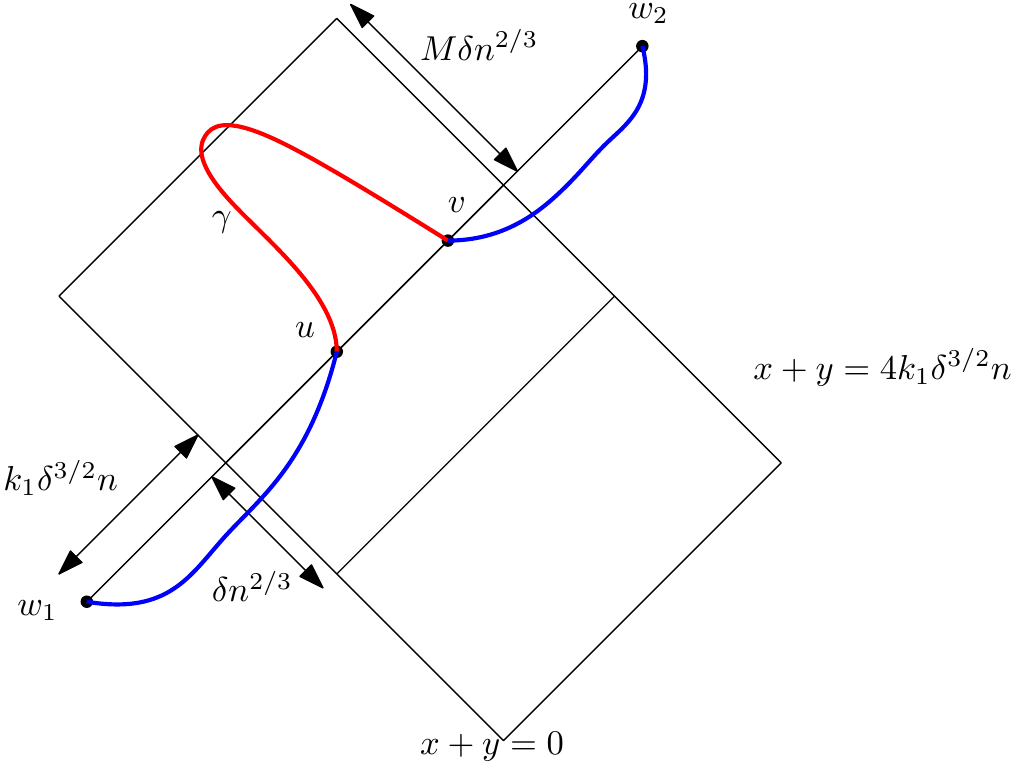}
	\end{center}
	\caption{Proof of Lemma \ref{short1}: we show that it is unlikely for
	there to be a short excursion $\gamma$ (marked in red) between $u$ and $v$
	outside $U_{\delta}$ that exits $U_{M\delta}$ and has weight not too
	small. To show this, on the event of existence of such a $\gamma$, we
	construct a path $\gamma_1$ (the concatenation of the paths marked in blue
	with $\gamma$) such that $\gamma_1$ is a path between two fixed points,
well separated in the time direction, such that $\ell(\gamma_1)$ is not too small. Notice that $\gamma_1$ has large transversal fluctuation and hence the latter event is shown to be unlikely using Proposition \ref{mod4}.}
	\label{fig:short}
\end{figure}


\begin{lemma}
	\label{short1}
	For any $k_1>0$ fixed, all $\delta$ sufficiently small,
	and $i\in\{0,1,\cdots,\frac{1}{k_1 \delta^{3/2}}-2\}$,  we have that there
	exists an absolute constant $\beta$ and positive constants $C,c$
		(independent of $\delta$) such that for $M$ large enough (independent
		of $\delta$) and for $n$ large
		enough depending on $\delta$, 
	\begin{displaymath}
		\mathbb{P}
		\left( (\mathbf{Short}^\mathtt{L}_i)^c \right)\leq Ce^{-cM^3}<1/2.
	\end{displaymath}
\end{lemma}
\begin{proof}
	By translation invariance, we can restrict to $i=0$. Define
	$w_1=(-\frac{1}{2}\delta^{3/2} n,\frac{1}{2}\delta^{3/2}
	n)-k_1\delta^{3/2}\mathbf{n}$ and $w_2=2k_1 \delta^{3/2} \mathbf{n} + (-\frac{1}{2}\delta^{3/2} n,\frac{1}{2}\delta^{3/2}
	n)+k_1 \delta^{3/2} \mathbf{n}$; see Figure \ref{fig:short}. Now, condition on the occurrence of the event
	$(\mathbf{Short}^\mathtt{L}_0)^c $. Thus, there exist $u,v\in
	U_\delta^\mathtt{L}$
	along with a path $\gamma\colon u\rightarrow v$ satisfying
	\eqref{eqn:15.51} along with $0\leq \phi(u)\leq \phi(v)\leq
	4k_1\delta^{3/2}n$.
	Analogous to $\mathcal{A}_1$ and $\mathcal{A}_2$ the proof of Lemma \ref{far2}, consider the
	events \begin{displaymath}
		\mathcal{B}_1=\left\{ T_{w_1,u}^{(\mathtt{Left}_\delta)^c}-
	2(\phi(u)-\phi(w_1)) \geq - \frac{M^2}{2} \beta(2k_1 \delta^{3/2} n)^{1/3}\right\}
	\end{displaymath} and 
	\begin{displaymath}
		\mathcal{B}_2= \left\{ T_{v,w_2}^{(\mathtt{Left}_\delta)^c}-
	2(\phi(w_2)-\phi(v)) \geq - \frac{M^2}{2} \beta  (2k_1 \delta^{3/2}
	n)^{1/3}\right\}.
	\end{displaymath}
	Note that $\mathcal{B}_1$ and
	$\mathcal{B}_2$ are independent of each other and the vertex weights in
	$\mathtt{Left}_\delta$. Also, $\mathcal{B}_1$ and $\mathcal{B}_2$ together with the
	conditioning imply the event $\mathcal{B}_3$ which is defined by
	\begin{equation}
		\label{eqn:11}
		\mathcal{B}_3=\left\{ \exists \gamma_1\colon w_1\rightarrow w_2
		\text{ with } \gamma_1 \cap \mathtt{Left}_{M\delta}\neq \phi \text{
		and } \ell(\gamma_1) -2(\phi(w_2)-\phi(w_1))\geq -2\beta M^2
		(2k_1 \delta^{3/2} n)^{1/3} \right\}.
	\end{equation}
	Indeed, if $\chi_1\colon w_1\rightarrow u$ is the path attaining $
	T_{w_1,u}^{(\mathtt{Left}_\delta)^c}$ and $\chi_2\colon v\rightarrow
	w_2$ is the path attaining $T_{v,w_2}^{(\mathtt{Left}_\delta)^c}$, then
	we can define $\gamma_1\colon w_1\rightarrow w_2$ as the concatenation of
	$\chi_1,\gamma,\chi_2$ and use $\ell(\gamma_1)\geq
	\ell(\chi_1)+\ell(\gamma)+\ell(\chi_2)$ which holds because both the
	endpoints were excluded in the definition of $\ell$.
	
	Since $k_1$ is some fixed constant (independent of $\delta$), we can
	invoke Proposition \ref{mod4} to say that there exists an absolute constant
	$\beta$ such that we have 
	\begin{equation}
		\label{eqn:12}
		\mathbb{P}\left( \mathcal{B}_3 \right)\leq Ce^{-cM^3}
	\end{equation}
	for large enough $n$. Indeed, we can just define $\beta=2^{-2/3} \xi$,
	where $\xi$ is in the statement of Lemma \ref{mod4}. Using this, we
	have that
	\begin{equation}
		\label{eqn:13}
		\min_u\mathbb{P}\left( \mathcal{B}_1 \right)\min_v\mathbb{P}\left( \mathcal{B}_2
		\right)\mathbb{P}\left( (\mathbf{Short}^\mathtt{L}_0)^c  \right)\leq
		\mathbb{P}\left( \mathcal{B}_3 \right)\leq Ce^{-cM^3},
	\end{equation}
	where both the minimum's in the above equation are taken over the set $\left\{
	z\in U_\delta^\mathtt{L}:\phi(z)\in [0,4k_1\delta^{3/2}n] \right\}$.
	By the symmetry of the setting, we only prove a lower bound for $\min_u \mathbb{P}\left( \mathcal{B}_1
\right)$ and omit the proof of the corresponding lower bound for $\min_v\mathbb{P}\left( \mathcal{B}_2
		\right)$ . Note that for all possible points $u$, we have that
{$2k_1\delta^{3/2} n\leq \phi(u)-\phi(w_1)\leq 6k_1 \delta^{3/2} n$}, and thus
$\frac{(\phi(u)-\phi(w_1))^{1/3}}{2\beta
		(2k_1 \delta^{3/2} n)^{1/3}}$ is bounded away from $0$ and
		$\infty$. Hence, on noting that $T_{w_1,u}^{(\mathtt{Left}_\delta)^c}\geq
	T_{w_1,u}^{V}$ for the rectangle $V$ defined by 
	\begin{displaymath}
		V=\left\{ -\delta n^{2/3}\leq \psi(u)\leq \delta n^{2/3}
		\right\}\cap \left\{ -2k_1\delta^{3/2}n\leq \phi(u)\leq
		6k_1\delta^{3/2} n \right\}
	\end{displaymath}
	and by using Proposition \ref{mod3} for $V$, along with
	\eqref{e:mean} to change the centering from $\mathbb{E}T_{w_1,u}$ to
	$2(\phi(u)-\phi(w_1))$, we get that for
	$M$ large enough (independent of $\delta$),
		\begin{equation}
			\label{eqn:14}
			\min_u\mathbb{P}\left( \mathcal{B}_1 \right)\geq 1-C_2e^{-c_2 M^2}\geq
			\frac{1}{2}.
		\end{equation}
		On combining this and the analogous lower bound of $\min_v\mathbb{P}\left( \mathcal{B}_2
		\right)$ with \eqref{eqn:13}, we get that for large enough
		$M$ (independent of $\delta$), we have
		\begin{equation}
			\label{eqn:14.1}
			\mathbb{P}( (\mathbf{Short}^\mathtt{L}_0)^c )\leq 4C e^{-c M^3}
		\end{equation}
		provided that $n$ is large enough depending on $\delta$.
\end{proof}

Instead of the conditioning argument presented above, another slightly
different way to prove
Lemma \ref{short1} would be to define $\mathcal{B}$ by
\begin{displaymath}
	\mathcal{B}=\left\{ \inf_{u',v'\in V^\mathtt{L},
	\phi(v')-\phi(u')\geq 2 k_1\delta^{3/2}n}
	\left\{T^V_{u',v'}-2(\phi(v')-\phi(u'))\right\}\geq
	- \frac{M^2}{2} \beta  (2k_1 \delta^{3/2}
	n)^{1/3}\right\},
\end{displaymath}
where $V^\mathtt{L}$ refers to the left long side of the rectangle $V$. One
can now observe that $(\mathbf{Short}^\mathtt{L}_0)^c\cap \mathcal{B}
\subseteq \mathcal{B}_3$, where the events on the left hand side are independent. $\mathbb{P}\left( \mathcal{B} \right)$ can be lower bounded by Proposition
\ref{mod3} and \eqref{e:mean}, and the rest of the proof follows
similarly. 

We end this section by proving Lemma \ref{short2}.

\begin{proof}[Proof of Lemma \ref{short2}]
	Note that the $\mathbf{Short}^\mathtt{L}_i$ are all decreasing events.
	Using the FKG inequality along with Lemma \ref{short1} immediately gives
	that for $n$ large enough depending on $\delta$,
	\begin{equation}
		\label{eqn:14.2}
		\mathbb{P}\left( \mathbf{Short}^\mathtt{L} \right) \geq
	(1-C_1e^{-c_1M^{3}})^{\frac{\delta^{-3/2}}{k_1}-1}\geq
	(1/2)^{c_2\delta^{-3/2}}=
	Ce^{-c\delta^{-3/2}}.
	\end{equation}
	Note that as mentioned in Section
	\ref{ss:constants}, $M$ is fixed to be large enough
	compared to the other parameters so as to
	satisfy the conclusion of Lemma \ref{short1}.
\end{proof}

\section{One point small deviation estimates for the geodesic}
\label{s:onept}

In this section, we will provide the proof of Theorem \ref{t:onepoint} and
the proof of Corollary \ref{c:as} (ii). We start with the upper bound in Theorem \ref{t:onepoint}, which is easier. 

\subsection{Proof of Theorem \ref{t:onepoint}, upper bound}
\label{s:oneub}
As already mentioned in the introduction, we shall complete the proof of the
upper bound using the idea outlined in \cite[Remark 2.11]{BHS18}. The main
ingredient that we require, an estimate of the number of disjoint geodesics between two
intervals of size $n^{2/3}$ on $\mathbb{L}_0$ and $\mathbb{L}_{2n}$, is quoted
from \cite{BHS18}. Recall that we use $\mathbb{L}_t$ to denote the line
$\left\{ x+y=t \right\}\subseteq \mathbb{Z}^2$. Also, recall from \eqref{eqn:1.0.1} that for any $\Delta>0$, we use the notation $U_\Delta$ for the rectangle
\begin{displaymath}
	\left\{ -\Delta n^{2/3}\leq \psi(u)\leq \Delta n^{2/3} \right\}\cap
	\left\{ 0\leq \phi(u)\leq 2n \right\}.
\end{displaymath}
%
%
%
%

\begin{proposition}[{\cite[Theorem 2]{BHS18}}]
	\label{one-2}
	For any $\Delta>0$, let $\mathcal{M}_l$ denote the event that there exist points
	$\left\{f_1,f_2,\cdots,f_l\right\}$ on $
	\underline{U}_\Delta$ and points
	$\left\{g_1,g_2,\cdots,g_l\right\}$ on $\overline{U}_\Delta$
	satisfying $\psi(f_1)>\psi(f_2)>\cdots> \psi(f_l)$ and
	$\psi(g_1)>\psi(g_2)>\cdots> \psi(g_l)$ such that the geodesics
	$\Gamma_{f_i,g_i}$ are all pairwise disjoint. Then there exist
	positive constants $n_0,l_0$ such that for all $n>n_0$ and all
	$l_0\leq l\leq n^{0.01}$, we have
	\begin{displaymath}
		\mathbb{P}\left( \mathcal{M}_l \right)\leq e^{-c_1l^{1/4}}
	\end{displaymath}
	for some positive constant $c_1$ depending on $\Delta$.
\end{proposition}
Note that  \cite[Theorem 2]{BHS18} is stated for the special case
$\Delta=1$, but one can check that the same argument gives the result for every $\Delta>0$.

Now fix an $\epsilon>0$ and let $t\in[\![\epsilon n , (2-\epsilon) n]\!]$ as in the
setting of Theorem \ref{t:onepoint}. Let $\mathcal{L}_t$ denote the maximum number of points $\left\{
h_1,h_2,\cdots,h_l
\right\}$ on $\mathbb{L}_{t}\cap U_1$ strictly decreasing in $\psi(\cdot)$ such that there exist points 	$\left\{f_1,f_2,\cdots,f_l\right\}$ on $\underline{U}_1$ and points
	$\left\{g_1,g_2,\cdots,g_l\right\}$ on $\overline{U}_1$
	strictly decreasing in $\psi(\cdot)$ such that $h_i\in \Gamma_{f_i,g_i}$. We now use
	Proposition \ref{one-2} to show that $\mathbb{E}\mathcal{L}_t$ is upper bounded uniformly in $n$.

	\begin{lemma}
		\label{one-2.1}
		Fix an $\epsilon \in (0,1)$. There exists a constant $C>0$ such that for all $n$ large and all $t\in\llbracket\epsilon n , (2-\epsilon) n\rrbracket$, we have
		\begin{displaymath}
			\mathbb{E}\mathcal{L}_t\leq C.
		\end{displaymath}
	\end{lemma}
	\begin{proof}
		For a collection $\left\{ f_1,f_2,\cdots,f_l \right\}$, $\left\{
		g_1,g_2,\cdots,g_l \right\}$ and $\left\{ h_1,h_2,\cdots,h_l \right\}$
		satisfying the conditions in the definition of $\mathcal{L}_t$,
		we have that any two geodesics $\Gamma_{f_i,g_i}$ and
		$\Gamma_{f_j,g_j}$ must be disjoint when restricted to at least one of the
		time intervals $[0,\epsilon n]$ and $[(2-\epsilon) n, 2n]$. Indeed, this follows by the ordering (and uniqueness) of
		geodesics along with the fact that $h_i\neq h_j$. Let $F$ be the
		size of the largest subset of $\left\{ 1,\cdots, l \right\}$ such
		that all the geodesics $\left\{ \Gamma_{f_i,g_i} \right\}_{i\in I_1}$
		are pairwise disjoint when restricted to the time interval
		$[0,\epsilon n]$. Let $G$ be defined analogously for the
		time interval
		$[(2-\epsilon) n, 2n]$. It is easy to see that
		$\max\left\{ F,G \right\}\geq l/2$. Hence, by a union bound along
		with Proposition \ref{one-2}, we have
		\begin{equation}
			\label{eqn:14.10.1}
			\mathbb{P}\left( \mathcal{L}_t\geq l \right)\leq
			e^{-c_2l^{1/4}}
		\end{equation}
		for all $l\leq n^{0.01}$ and some constant $c_2$ depending on
		$\epsilon$. The far end of the tail of $\mathcal{L}_t$
		can be bounded by using that 
		\begin{equation}
			\label{eqn:14.10.2}
			\mathcal{L}_t\leq |\mathbb{L}_t\cap
		U_1^n|= n^{2/3} +1
		\end{equation}
		deterministically, and on using \eqref{eqn:14.10.1} with
		$l=n^{0.01}$, this yields
		\begin{equation}
			\label{eqn:14.10.3}
			\mathbb{E}[\mathbbm{1}_{\mathcal{L}_t\geq n^{0.01}}
			\mathcal{L}_t]\leq e^{-c_2n^{0.01/4}}(n^{2/3}+1)<C_2
		\end{equation}
		for some positive constant $C_2$ for all $n$. The proof is completed using \eqref{eqn:14.10.1} and
		\eqref{eqn:14.10.3}.
	\end{proof}

	We now complete the proof of the upper bound in Theorem \ref{t:onepoint}
	by using Lemma \ref{one-2.1}.

\begin{proof}[Proof of Theorem \ref{t:onepoint}, upper bound]
We consider the points $u_{i}=(i\delta
n^{2/3},-i\delta n^{2/3})$ and $v_i=\mathbf{n}+u_{i}$ for $i\in \llbracket
-\frac{\delta^{-1}}{2}, \frac{\delta^{-1}}{2} \rrbracket$. By translation invariance, we have that the probability 
	\begin{equation}
		\label{eqn:14.11}
		\mathbb{P}\left( |\Gamma_{u_i,v_i}(t)- \psi(u_i)|< \delta n^{2/3} \right)
	\end{equation}
	is independent of $i$. Also, by looking at the definition of
	$\mathcal{L}_t$ in Lemma \ref{one-2.1}, we can deduce that
	\begin{equation}
		\label{eqn:14.12}
		\sum_{i=-\delta^{-1}/4}^{\delta^{-1}/4}\mathbb{P}\left(|\Gamma_{u_i,v_i}(t)-
		\psi(u_i)|< \delta n^{2/3}
		\right)=\mathbb{E}\left[\sum_{i=-\delta^{-1}/4}^{\delta^{-1}/4}
		\mathbbm{1}\left(|\Gamma_{u_i,v_i}(t)- \psi(u_i)|<\delta
		n^{2/3}\right)
		\right]\leq \mathbb{E}\left[ \mathcal{L}_t \right]\leq C,
	\end{equation}
	where $C_2$ is obtained from Lemma \ref{one-2.1}. \eqref{eqn:14.11} and
	\eqref{eqn:14.12} together imply 
	\begin{equation}
		\label{eqn:14.13}
		\mathbb{P}\left( |\Gamma_{u_0,v_0}(t)- \psi(u_0)|< \delta n^{2/3}
		\right)=\mathbb{P}\left( |\Gamma_n(t)|< \delta n^{2/3}
		\right)\leq 2C \delta,
	\end{equation}
	and the result follows immediately.
\end{proof}

\subsection{Proof of the lower bound in Theorem \ref{t:onepoint}}
\label{s:onelb}
We shall prove the lower bound in Theorem \ref{t:onepoint} in this subsection. We first complete the proof of the lower bound using Proposition \ref{one0} and prove the latter result, which is the heart of the technical content in this section, in the next subsection.


We recall the notation from the statement of Proposition \ref{one0}:
for some $M>0$ which will be taken to be large later, we define the points
$a_1,a_2$ by
\begin{align}
	&a_1=(-Mn^{2/3},Mn^{2/3}),\nonumber\\
	\label{e*:1}
	&a_2=(Mn^{2/3},-Mn^{2/3}),
\end{align}
and let $b_1=a_1+\mathbf{n}$ and $b_2=a_2+\mathbf{n}$.
Note that the parameter $M$ used in this section is completely
independent of the parameter $M$ used in Section \ref{s:lb}. Proposition \ref{one0} roughly
states that the geodesics $\Gamma_{a_1,b_1}$ and
$\Gamma_{a_2,b_2}$ coalesce with positive probability, while
not straying more than $Mn^{2/3}$ distance away in the transversal direction.

\begin{figure}[htbp!]
\begin{center}
	\includegraphics[width=0.5\linewidth]{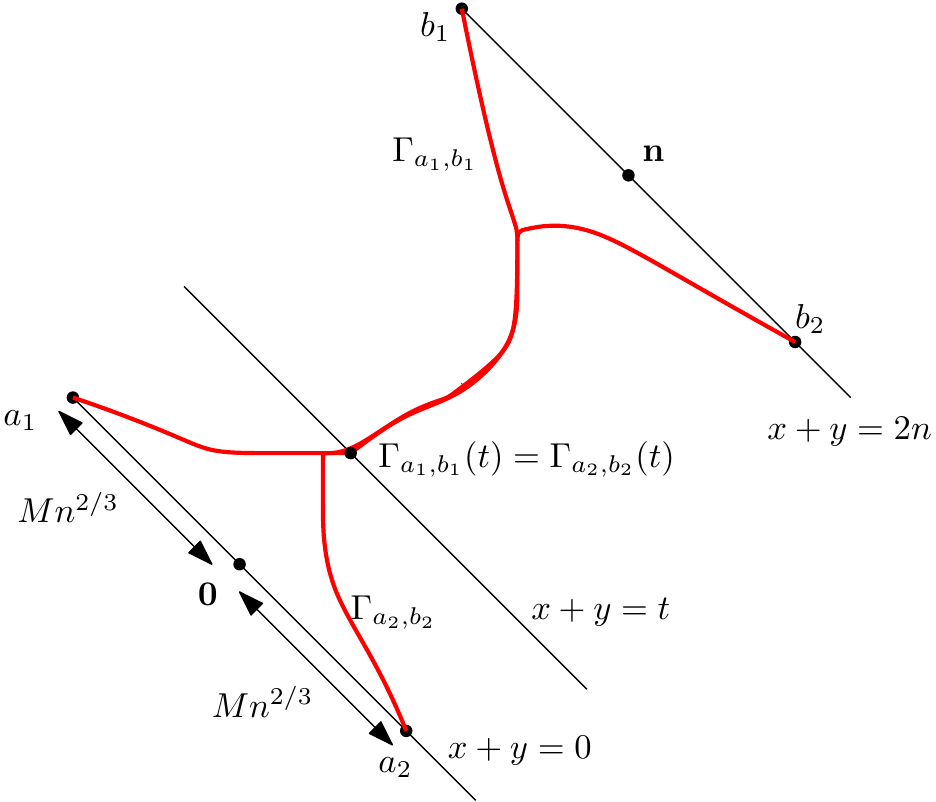}
	\end{center}
	\caption{Proof of the lower bound in Theorem \ref{t:onepoint}: Proposition
	\ref{one0} implies that with probability bounded away from $0$, the
	geodesics $\Gamma_{a_1,b_1}$ and $\Gamma_{a_2,b_2}$ intersect the line
	$x+y=t$ at the same point with the point's $\psi$-coordinate bounded by
	$M\delta n^{2/3}$ in the absolute value. By the planar ordering of
	geodesics, this implies that for every point $u$ between $a_1$ and $a_2$ and for every point $v$ between $b_1$ and $b_2$, $\Gamma_{u,v}$ intersects the line $x+y=t$ at the same point. In particular, this implies that there exists $u_i=(i\delta n^{2/3},-i\delta n^{2/3})$ and $v_i=u_{i}+\mathbf{n}$ such that $|\Gamma_{u_i,v_i}-\psi(u_i)|\leq \delta n^{2/3}$, and hence the expected number of such $i$ is bounded away from $0$. The proof of the lower bound in Theorem \ref{t:onepoint} follows by a translation invariance argument.}
	\label{fig:lowerbound2}
\end{figure}


\begin{proof}[Proof of the lower bound in Theorem \ref{t:onepoint}]
The basic strategy is similar to the proof of the upper bound.	 We consider points $u_{i}=(i\delta
n^{2/3},-i\delta n^{2/3})$ and $v_i=\mathbf{n}+u_{i}$ for $i\in \llbracket
-M\delta^{-1}, M\delta^{-1} \rrbracket$. By
	the same reasoning as in \eqref{eqn:14.11}, we have that
	\begin{equation}
		\label{eqn:14.14}
			\mathbb{P}\left( |\Gamma_{u_i,v_i}(t)- \psi(u_i)|\leq \delta n^{2/3} \right)
	\end{equation}
	is independent of $i$. Similarly, we also have that 
	\begin{equation}
		\label{eqn:14.15}
		\sum_{i=-M/\delta}^{M/\delta}\mathbb{P}\left(|\Gamma_{u_i,v_i}(t)-
		\psi(u_i)|\leq \delta n^{2/3}
		\right)=\mathbb{E}\left[\sum_{i=-M/\delta}^{M/\delta}
		\mathbbm{1}\left(|\Gamma_{u_i,v_i}(t)- \psi(u_i)|\leq \delta
		n^{2/3}\right)
		\right].
	\end{equation}
	We locally refer to the event defined in the statement of Proposition
	\ref{one0} as $\mathcal{A}$. Note that by the ordering of geodesics, we have that
	$\Gamma_{u_i,v_i}(t)$ is the same for all $i$ on the event
	$\mathcal{A}$. This in turn
	implies that on the event $\mathcal{A}$, the expression
	\begin{displaymath}
			\mathbbm{1}\left(|\Gamma_{u_i,v_i}(t)- \psi(u_i)|\leq \delta
			n^{2/3}\right)
	\end{displaymath}
	must be $1$ for at least one $i$. Thus, we have
	\begin{align}
		\label{eqn:14.16}
		\mathbb{E}\left[\sum_{i=-M/\delta}^{M/\delta}
		\mathbbm{1}\left(|\Gamma_{u_i,v_i}(t)- \psi(u_i)|\leq \delta
		n^{2/3}\right)
		\right]&\geq \mathbb{E}\left[\sum_{i=-M/\delta}^{M/\delta}
		\mathbbm{1}\left(|\Gamma_{u_i,v_i}(t)- \psi(u_i)|\leq \delta
		n^{2/3}\right)
		; \mathcal{A}\right] \nonumber \\
		&\geq \mathbb{P}\left( \mathcal{A} \right)\geq c,
	\end{align}
	where $c$ comes from Proposition \ref{one0}. Finally, by combining
	\eqref{eqn:14.14}, \eqref{eqn:14.15} and \eqref{eqn:14.16}, we obtain
	\begin{equation}
		\label{eqn:14.17}
		\mathbb{P}\left( |\Gamma_n|\leq \delta n^{2/3} \right)=
		\mathbb{P}\left(|\Gamma_{u_0,v_0}(t)-
		\psi(u_0)|\leq \delta n^{2/3} \right)\geq \frac{c}{2M}\delta,
	\end{equation}
	which finishes the proof since $\frac{c}{2M}$ is a constant.
\end{proof}

Before moving onto the proof of Proposition \ref{one0}, we quickly finish the proof of
Corollary \ref{c:as} (ii) by using Theorem \ref{t:onepoint}.

\begin{proof}[Proof of Corollary \ref{c:as} (ii)]

	We use Theorem 
	\ref{t:onepoint} along with the definition $\pi_{n}(s)= n^{-2/3}\Gamma_{n}(2ns)$ for
	$2ns\in \mathbb{Z}$. Indeed, for $n$ large enough, we have that for any
	fixed $s\in (0,1$), $\frac{\lceil 2ns\rceil}{2n}$ is bounded away from $0$
	and $1$ and now Theorem \ref{t:onepoint} implies that
	for some constants $c,C$, we have
	\begin{displaymath}
		c\delta\leq \mathbb{P}\left( |\Gamma_{n}(\lceil 2ns\rceil )|\leq
		\delta n^{2/3}\right)\leq C\delta.
	\end{displaymath}
	Using that $\Gamma_n(\cdot)$ has $\pm 1$ increments, and that
	$\pi_n(s)$ is defined for all $s\in[0,1]$ by interpolating the values on
	$[0,1]\cap \frac{1}{2n}\mathbb{Z}$, we obtain that
	\begin{displaymath}
		\frac{c}{2}\delta\leq \mathbb{P}\left( |\pi_n(s)|\leq
		\delta \right)\leq 2C\delta.
	\end{displaymath}
	for all $n$ large enough. The mapping $f\mapsto |f(s)|$ is continuous as a
	map from $C[0,1]$ to $\mathbb{R}^+$ with the uniform convergence
	topology and Euclidean topology respectively. Thus by the continuous
	mapping theorem, if $\pi_{n_i}\Rightarrow \pi$
	as $i\rightarrow\infty$ for a subsequence $\left\{ n_i \right\}$, we have
	that $|\pi_{n_i}(s)|\Rightarrow|\pi(s)|$. We can
	now use the Portmanteau theorem in the same way as in the Proof of
	Corollary \ref{c:as} (i) to obtain
	\begin{displaymath}
			\frac{c}{4}\delta\leq \mathbb{P}\left( |\pi(s)|\leq
		\delta \right)\leq 4C\delta
	\end{displaymath}
	for any fixed $s\in (0,1)$. The uniformity of $c,C$ as long as $s$ is
	bounded away from $0$ and $1$ is evident from the proof.
\end{proof}

\subsection{Proof of Proposition \ref{one0}}
\label{s:one0}
%

To prove Proposition \ref{one0}, we shall again construct a host of favourable
events whose intersection holds with positive probability and on which the
geodesics are forced to coalesce at the required location. These events
will be defined with a parameter $M$ (note that this $M$ has nothing to do
with the parameter used in the proof of the lower bound in Theorem \ref{t:sb}
from the previous sections). We start by outlining how the choice of $M$ is
fixed, and then define the relevant events.

\subsubsection{Choice of the parameter $M$}
\label{ss:constants1}
All the constructions will be made with a parameter $M$ which will be fixed to be a large constant
at the end, and this fixed constant is the one appearing in the statement
of Proposition \ref{one0}. The
specific choice of $M$ is fixed large enough to satisfy the
conclusions of Lemma \ref{one0.2}, Lemma \ref{one1} and
Proposition \ref{one5}. Using this value of $M$ in Lemma \ref{one3.1}
provides us with a corresponding value of $\beta$ and the final probability
lower bound in \eqref{eqn:31} is in terms of $\beta$. 

\noindent
Observe that \eqref{e:mean} along with
$\mathbb{E}\underline{T}_{u,v}=\mathbb{E}T_{u,v}+1$ implies that for all $\Delta$
large, and for all $n$ large depending on $\Delta$, we have that for all $u\in
\underline{U}_\Delta$ and $v\in \overline{U}_\Delta$,
\begin{align}
	\label{eqn:14.17.0}
	&-C'n^{1/3}\leq \mathbb{E}T_{u,v}-(4n-C\Delta^2n^{1/3})\leq
	C'n^{1/3},\nonumber\\
	&-C'n^{1/3}\leq \mathbb{E}\underline{T}_{u,v}-(4n-C\Delta^2n^{1/3})\leq
	C'n^{1/3}
\end{align}
for some constants $C,C'$ not depending on $\Delta$. We will repeatedly use
\eqref{eqn:14.17.0} with $\Delta$ being equal to
$M^{3/4},M$ and $M^2$ to change the centering in the estimates.  

\subsubsection{Events favourable for coalescence}

Recall from \eqref{eqn:1.0.1} that for any $\Delta>0$, we use the notation $U_\Delta$ for the rectangle
\begin{displaymath}
	\left\{ -\Delta n^{2/3}\leq \psi(u)\leq \Delta n^{2/3} \right\}\cap
	\left\{ 0\leq \phi(u)\leq 2n \right\}.
\end{displaymath}
Let the rectangles $R_1$ and $R_2$ be defined by
\begin{gather}
	\label{eqn:14.17.1} R_1=U_{M^2}\cap\left\{\frac{\epsilon}{4} n\leq
	\phi(u)\leq \frac{3\epsilon}{4}n \right\},\\
\label{eqn:14.17.2}
R_2=U_{M^2} \cap \left\{2n-\frac{3\epsilon}{4}n\leq \phi(u)\leq
2n-\frac{\epsilon}{4}n \right\}.
\end{gather} Let $\underline{R}_1$ and $\overline{R}_1$ denote the
short sides of $R_1$ lying on the lines $\mathbb{L}_{\frac{\epsilon n}{4}}$ and
$\mathbb{L}_{\frac{3\epsilon n}{4}}$ respectively. Similarly, let  $\underline{R}_2$ and
$\overline{R}_2$ denote the
short sides of $R_2$ lying on the lines $\mathbb{L}_{2n-\frac{3\epsilon n}{4}}$
and $\mathbb{L}_{2n-\frac{\epsilon n}{4}}$ respectively; see Figure \ref{fig:barrier}. We now define some high probability events which will
be used later. For a point
$u\in \mathbb{L}_0$, and any $x>0$, define the event $\mathbf{Restr}_{u,x}$ by
\begin{equation}
	\label{eqn:14.17.3}
	\mathbf{Restr}_{u,x}=\left\{\text{All } \gamma\colon u\rightarrow u+\mathbf{n}
	\text{ satisfying } \gamma \cap (U_x)^c\neq \emptyset \text{ have }
	\ell(\gamma) \leq 4n-\frac{\xi}{2} x^2n^{1/3}\right\},
\end{equation}
where in this section, $\xi$ always denotes the constant from Proposition \ref{mod4}.
We define the event $\mathbf{Restr}$ by
\begin{equation}
	\label{eqn:14.18}
	\mathbf{Restr}= \mathbf{Restr}_{\mathbf{0},2M}\cap
	\mathbf{Restr}_{a_1,M^2}\cap \mathbf{Restr}_{a_2,M^2},
\end{equation}
where  $a_1=(-Mn^{2/3}, Mn^{2/3})$ and $a_2=(Mn^{2/3},-Mn^{2/3})$
as defined in \eqref{e*:1}. Also, recall that we have defined
$b_1=a_1+\mathbf{n}$ and $b_2=a_2+\mathbf{n}$.

We now define some more events which will be used in our construction.
\begin{itemize}
	\item $\mathcal{W}_1$: the event that for any $u\in \mathbb{L}_0\cap
		U_{M^2}$ and $v\in \underline{R}_1$, we have $|\widetilde{T}_{u,v}|\leq
		M n^{1/3}$.
	\item $\mathcal{W}_2$: the event that for any $u\in \overline{R}_1$ and $v\in
		\underline{R}_2$, we have $|\widetilde{\underline{T}}_{u,v}|\leq
		M n^{1/3}$.
	\item $\mathcal{W}_3$: the event that for any $u\in \overline{R}_2$ and $v\in
		\mathbb{L}_{2n}\cap
		U_{M^2}$, we have $|\widetilde{\underline{T}}_{u,v}|\leq
		M n^{1/3}$.
\end{itemize}

\begin{itemize}
	\item $\mathcal{S}_1$:  the event that for any $u\in \underline{R}_1$ and $v\in
		\overline{R}_1$, we have $|\widetilde{\underline{T}}^{R_1}_{u,v}|\leq
		M n^{1/3}$.
	\item $\mathcal{S}_2$:  the event that for any $u\in \underline{R}_2$ and $v\in
		\overline{R}_2$, we have $|\widetilde{\underline{T}}^{R_2}_{u,v}|\leq
		M n^{1/3}$.
\end{itemize}
\begin{itemize}
	\item $\mathcal{S}_1'$:  the event that for any $u\in \underline{R}_1$ and $v\in
		\overline{R}_1$, we have $\widetilde{\underline{T}}^{R_1}_{u,v}\leq
		M n^{1/3}$.
	\item $\mathcal{S}_2'$:  the event that for any $u\in \underline{R}_2$ and $v\in
		\overline{R}_2$, we have $\widetilde{\underline{T}}^{R_2}_{u,v}\leq
		M n^{1/3}$.
\end{itemize}
Note that both $T$ and $\underline{T}$ are used to define the above events,
and the minor technical reason for doing this will become clear in the proof of
Proposition \ref{one5}. 

We will be using the events
\begin{gather}
	\label{eqn:14.19}
	\mathcal{W}=\mathcal{W}_1\cap \mathcal{W}_2 \cap \mathcal{W}_3,\\
	\label{eqn:14.20}
	\mathcal{S}=\mathcal{S}_1\cap \mathcal{S}_2,\\
	\label{eqn:14.21}
	\mathcal{S}'=\mathcal{S}_1'\cap \mathcal{S}_2'.
\end{gather}
Note that the reason for defining $\mathcal{S}'$ and $\mathcal{S}$ separately is that $\mathcal{S}'$ is a
decreasing event, and this will allow us to apply the FKG inequality in
Proposition \ref{one4}. We now show that all the above events occur with high probability.

\begin{figure}[t]
\begin{center}
	\includegraphics[width=0.5\linewidth]{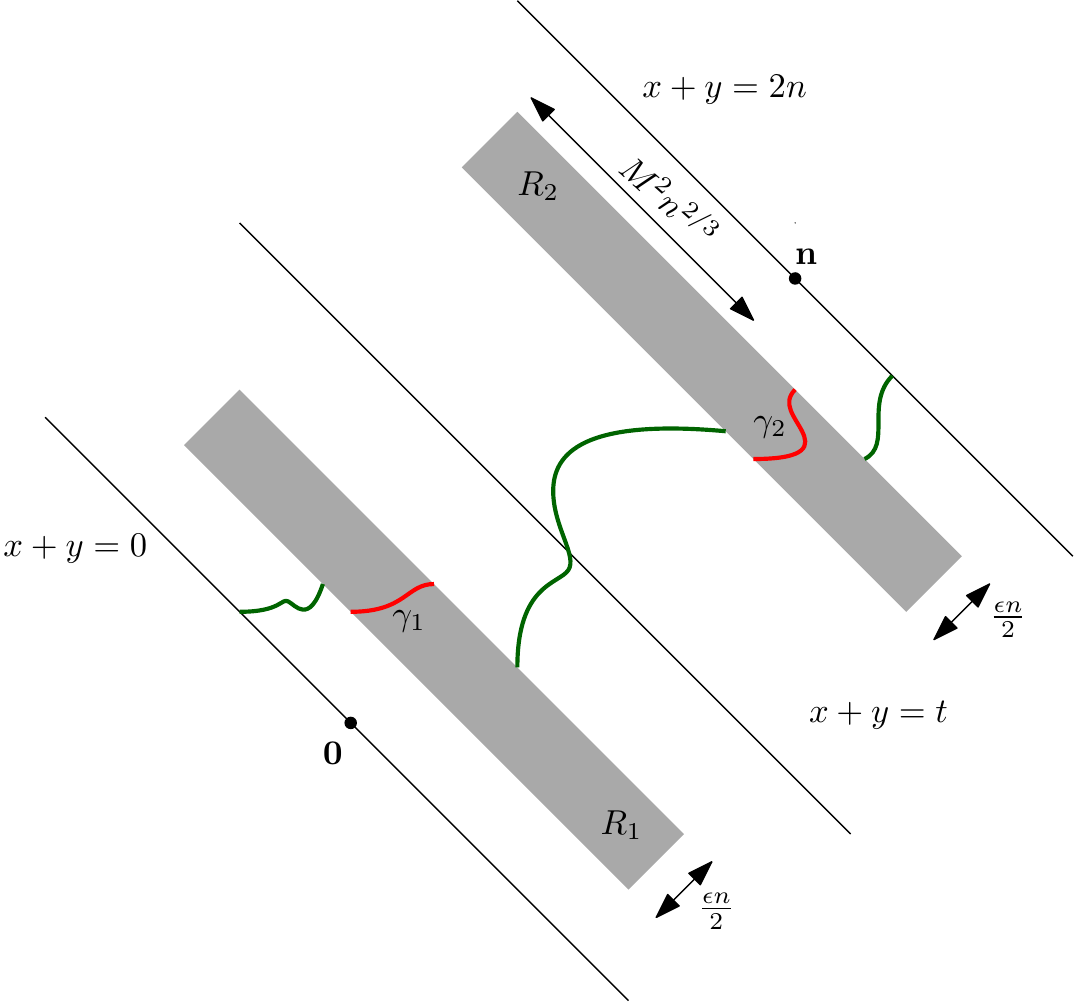}
	\end{center}
	\caption{The basic elements for constructing favourable events to force coalescence of geodesics in the proof of Proposition \ref{one0}: The rectangles $R_1$ and $R_2$, marked in grey, act as barriers, on either side of the lines $x+y=t$. All paths across $R_1$ and $R_2$ (except for those that intersect two specially designated paths $\gamma_1$ and $\gamma_2$, marked in red) are forced to be much shorter than typical. This is the content of the event $\mathcal{P}^{\gamma_1,\gamma_2}$. The other events $\mathcal{S}$, $\mathcal{S}'$ and $\mathcal{W}$ are typical; $\mathcal{W}$ ensures that the paths through the regions below $R_1$, between $R_1$ and $R_2$, and above $R_2$ (some instances of such path are marked in green) have typical weights, whereas $\mathcal{S}'$ ensures that any path through $R_1$ and $R_2$ does not have atypically large weight. The event $\mathbf{Restr}$, not shown in the figure, ensures that paths have typical transversal fluctuation.}
	\label{fig:barrier}
\end{figure}

\begin{lemma}
	\label{one0.2}
	For all $M$ large, and all $n$ large depending on $M$, we have $\mathbb{P}(\mathcal{W})\geq
	0.99$, $\mathbb{P}\left( \mathcal{S} \right)\geq 0.99, \mathbb{P}\left( \mathcal{S}'
	\right)\geq 0.99$ and $\mathbb{P}\left( \mathbf{Restr} \right)\geq 0.99$.
	Thus, we have
	\begin{displaymath}
		\mathbb{P}\left( \mathcal{S}'\cap \mathbf{Restr}\cap \mathcal{W} \right)\geq 0.97.
	\end{displaymath}
\end{lemma}
\begin{proof}
	The lower bounds for $\mathbb{P}(\mathcal{W}),
	\mathbb{P}(\mathcal{S})$ and $\mathbb{P}(\mathcal{S}')$ are very simple consequences of a union bound, Proposition
	\ref{mod2} and Proposition \ref{mod3}. We just illustrate the proof by proving the bound for
	$\mathbb{P}(\mathcal{S})$. By symmetry and a union bound, we only need to
	lower bound
	$\mathbb{P}(\mathcal{S}_1)$. Now, divide $\underline{R}_1$ into $M^2$ disjoint
	segments of length $1$, and call these segments $(\underline{R}_1)_i$
	where $i\in \left\{ 1,\cdots,M^2 \right\}$; do the same for
	$\overline{R}_1$ and call the segments $(\overline{R}_1)_j$. Let
	$P_{i,j}$ denote the parallelogram with $(\underline{R}_1)_i$ and
	$(\overline{R}_1)_j$ as the short sides. By using a slightly generalized
	version of
	Proposition \ref{mod3} for parallelograms whose long sides have slopes
	other than $1$ (see \cite[Theorem 4.2 (ii)]{BGZ19}), we have 
	\begin{equation}
		\label{eqn:14.22}
		\mathbb{P}\left( \inf_{u\in (\underline{R}_1)_i, v\in
		(\overline{R}_1)_j} \widetilde{T}_{u,v}^{R_1}\geq -Mn^{1/3} \right)\geq \mathbb{P}\left( \inf_{u\in (\underline{R}_1)_i, v\in
		(\overline{R}_1)_j} \widetilde{T}_{u,v}^{P_{i,j}}\geq -Mn^{1/3}
		\right)\geq 1-Ce^{-cM}.
	\end{equation}
	By Proposition \ref{mod2}, we have
	\begin{equation}
		\label{eqn:14.23}
		\mathbb{P}\left( \sup_{u\in (\underline{R}_1)_i, v\in
		(\overline{R}_1)_j} \widetilde{T}_{u,v}^{R_1}\leq Mn^{1/3} \right) \geq \mathbb{P}\left( \sup_{u\in (\underline{R}_1)_i, v\in
		(\overline{R}_1)_j} \widetilde{T}_{u,v}\leq Mn^{1/3} \right)\geq 1-C_1
		e^{-c_1 M^{3/2}}.
	\end{equation}
	{To complete the proof, one just needs to take a union bound over the polynomially many (in $M$) choices of $i,j$, and we are done by choosing $M$ sufficiently large because of the stretched exponential decay in the probability.}

	The lower bound for $\mathbb{P}\left( \mathbf{Restr} \right)$ is an
	immediate application of Proposition \ref{mod4} and a union bound.
Finally, we obtain
	\begin{equation}
	\label{eqn:14.23.1}
	\mathbb{P}\left( \mathcal{S}'\cap \mathbf{Restr}\cap \mathcal{W} \right)\geq 0.97.
	\end{equation}
	by another union bound.
\end{proof}

We now use the events that we have defined to show that with high
probability, the geodesic
$\Gamma_n$ is constrained inside $U_{M^{3/4}}$ while
ensuring that its different segments have weights close to typical.
\begin{lemma}
	\label{one1}
	On the event $\mathbf{Restr}_{\mathbf{0},M^{3/4}} \cap \mathcal{W} \cap \mathcal{S}$, we have that the geodesic $\Gamma_n\subseteq
U_{M^{3/4}}$ for all $M$ large, and $n$ large depending on $M$. We also have
\begin{displaymath}
	\mathbb{P}\left( \mathbf{Restr}_{\mathbf{0},M^{3/4}} \cap \mathcal{W} \cap \mathcal{S}
	\right)\geq 0.97.
\end{displaymath}
\end{lemma}
\begin{proof}
	On $\mathcal{W}\cap \mathcal{S}$, we have that for some constant $C$,
	\begin{equation}
		\label{eqn:14.23.2}
		T_{\mathbf{0},\mathbf{n}}\geq T_{\mathbf{0},\frac{\epsilon}{8}\mathbf{n}}+T^{R_1}_{\frac{\epsilon}{8}\mathbf{n},\frac{3\epsilon}{8}\mathbf{n}}+T_{\frac{3\epsilon}{8}\mathbf{n},(1-\frac{3\epsilon}{8})\mathbf{n}}+
		T^{R_2}_{(1-\frac{3\epsilon}{8})\mathbf{n},(1-\frac{\epsilon}{8})\mathbf{n}}+
		T_{(1-\frac{\epsilon}{8})\mathbf{n},\mathbf{n}}\geq 4n
		-(5M+5{C})n^{1/3},
	\end{equation}
	where the constant $C$ comes from switching the centering by using
	\eqref{e:mean}.
	Using the definition of $\mathbf{Restr}_{\mathbf{0},M^{3/4}}$ and the fact
	that $\xi M^{3/2}/2>M+C$ for $M$ large, it is clear from \eqref{eqn:14.23.2}
	that $\Gamma_n \subseteq U_{M^{3/4}}$.
	
	The probability lower bound is an
	easy consequence of Proposition \ref{mod4} and Lemma \ref{one0.2}. Indeed, we just need to
	individually lower bound $\mathbb{P}\left(
	\mathbf{Rest}_{\mathbf{0},M^{3/4}} \right)$, $\mathbb{P}\left( \mathcal{W}
	\right)$ and $\mathbb{P}(\mathcal{S})$ by $0.99$ and use a union bound.
\end{proof}
At this point, we introduce some notation. For a given path
$\gamma\colon u\rightarrow v$, we define
\begin{gather}
	\label{eqn:14.24.1}
	\widetilde{\ell}(\gamma)=\ell(\gamma)-\mathbb{E}T_{u,v},\\	
	\widetilde{\underline{\ell}}(\gamma)=
	\underline{\ell}(\gamma)-\mathbb{E}\underline{T}_{u,v}\nonumber.
\end{gather}
From now on, for a fixed $M$, we say that a path  $\gamma$ is {$\theta$-typical} if it
satisfies $\gamma \subseteq U_{\theta}$ and
$|\widetilde{\underline{\ell}}(\gamma)|\leq Mn^{1/3}$. 

\begin{lemma}
	\label{one2}
	Let $\mathcal{T}$ be the event that
	$\left.\Gamma_n\right|_{R_1}$ and $\left.\Gamma_n\right|_{R_2}$ are
	$M^{3/4}$-typical. Then $\mathbb{P}(\mathcal{T})\geq 0.97$ for all $M$
	large, and $n$ large depending on $M$ as in the statement of Lemma
	\ref{one1}.
\end{lemma}
\begin{proof}
	By
Lemma \ref{one1}, on the event $\mathbf{Restr}_{\mathbf{0},M^{3/4}} \cap \mathcal{W} \cap \mathcal{S}$, we have that the geodesic $\Gamma_n\subseteq
U_{M^{3/4}}$, which in particular implies that
$\left.\Gamma_n\right|_{R_1},\left.\Gamma_n\right|_{R_2}\subseteq
U_{M^{3/4}}$. Also, the definition of the event $\mathcal{S}$ ensures that on
$\mathbf{Restr}_{\mathbf{0},M^{3/4}} \cap \mathcal{W} \cap \mathcal{S}$, we have
$|\widetilde{\underline{\ell}}(\left.\Gamma_n\right|_{R_1})|\leq Mn^{1/3}$ and
$|\widetilde{\underline{\ell}}(\left.\Gamma_n\right|_{R_2})|\leq Mn^{1/3}$. Thus, we have that
\begin{displaymath}
	\mathbb{P}(\mathcal{T})\geq \mathbb{P}\left( \mathbf{Restr}_{\mathbf{0},M^{3/4}} \cap \mathcal{W} \cap
	\mathcal{S} \right)\geq 0.97
\end{displaymath}
by Lemma \ref{one1}.
\end{proof}

Let $\mathscr{I}_1$ and
$\mathscr{I}_2$ denote the set of all $M^{3/4}$-typical paths from
$\underline{R}_1$ to $\overline{R}_1$ and $\underline{R}_2$ to
$\overline{R}_2$ respectively. By Lemma \ref{one2}, it is clear that
\begin{equation}
	\label{eqn:14.25}
	\mathbb{P}\left( \left.\Gamma_n\right|_{R_1}\in \mathscr{I}_1,
\left.\Gamma_n\right|_{R_2} \in \mathscr{I}_2\right)\geq 0.97.
\end{equation}
Let $\mathscr{I}$ denote the set of pairs of paths $(\gamma_1,\gamma_2)$ such
that $\gamma_1\in \mathscr{I}_1$, $\gamma_2\in \mathscr{I}_2$ and 
\begin{equation}
	\label{eqn:14.26}
\mathbb{P}\left( \mathcal{S}'\cap {\mathbf{Restr}}\cap \mathcal{W}|
\left.\Gamma_n\right|_{R_1}=\gamma_1, \left.\Gamma_n\right|_{R_2}=\gamma_2
\right)\geq \frac{1}{2}.	
\end{equation}
In view of Lemma \ref{one0.2}, we have the following lower bound.
\begin{lemma}
	\label{one3}
	$\mathbb{P}\left( ( \left.\Gamma_n\right|_{R_1},
	\left.\Gamma_n\right|_{R_2})\in \mathscr{I} \right)\geq 0.9$ for all $M$
	large and $n$ large depending on $M$ as in the statement of Lemma
	\ref{one1}.
\end{lemma}

\begin{proof}
Let $\mathscr{P}$ be the set of pairs of paths $(\gamma_1,\gamma_2)$ such that
$\gamma_1\in \mathscr{I}_1$ and $\gamma_2\in \mathscr{I}_2$. Note that we have 
	\begin{align*}
		&\mathbb{P}(\mathcal{S}'\cap {\mathbf{Restr}}\cap
		\mathcal{W})\\
		&=1-\mathbb{P}( (\mathcal{S}'\cap {\mathbf{Restr}}\cap
		\mathcal{W})^c )\\
		&\leq 1-\sum_{(\gamma_1,\gamma_2)\in \mathscr{P}\setminus
		\mathscr{I}}
		\mathbb{P}\left( (\mathcal{S}'\cap {\mathbf{Restr}}\cap \mathcal{W})^c|
\left.\Gamma_n\right|_{R_1}=\gamma_1, \left.\Gamma_n\right|_{R_2}=\gamma_2
\right)\mathbb{P}\left( \left.\Gamma_n\right|_{R_1}=\gamma_1, \left.\Gamma_n\right|_{R_2}=\gamma_2 \right)\\
&\leq 1-\frac{1}{2}\mathbb{P}( ( \left.\Gamma_n\right|_{R_1},
\left.\Gamma_n\right|_{R_2})\in \mathscr{P}\setminus \mathscr{I})
\end{align*}
	 by using the definition of $\mathscr{I}$
	as in \eqref{eqn:14.26}.
	Using \eqref{eqn:14.23.1} along with the above, we obtain
	\begin{equation}
		\label{eqn:14.27}
		\mathbb{P}\left( ( \left.\Gamma_n\right|_{R_1},
		\left.\Gamma_n\right|_{R_2})\in \mathscr{P}\setminus \mathscr{I} \right)\leq
		0.06.
	\end{equation} 
	Finally, we have that
	\begin{displaymath}
		\mathbb{P}\left( ( \left.\Gamma_n\right|_{R_1},
		\left.\Gamma_n\right|_{R_2})\in \mathscr{I} \right)=	\mathbb{P}\left( ( \left.\Gamma_n\right|_{R_1},
		\left.\Gamma_n\right|_{R_2})\in \mathscr{P} \right)-\mathbb{P}\left( ( \left.\Gamma_n\right|_{R_1},
		\left.\Gamma_n\right|_{R_2})\in \mathscr{P}\setminus \mathscr{I} \right)\geq
		0.97-0.06=0.91
	\end{displaymath}
	by using \eqref{eqn:14.25} and \eqref{eqn:14.27}
	and this completes the
	proof.
\end{proof}

We will finally want to condition on the event $\left\{\left( \left.\Gamma_n\right|_{R_1},
\left.\Gamma_n\right|_{R_2}\right)=(\gamma_1,\gamma_2)\right\}$ for some
$\gamma\in \mathscr{I}$, and then show that in this conditional environment, we
have that both the geodesics {$\Gamma_{a_1,b_1}$} and
{$\Gamma_{a_2,b_2}$} meet
$\gamma_1$ and $\gamma_2$ with positive probability. {We will ensure
this by ``decreasing'' the background vertex weights in the region
$R_1\setminus \gamma_1$ and $R_2 \setminus \gamma_2$}, and we define some
more events which will help us in achieving this.

For any path $\gamma_1\subseteq R_1$ from $\underline{R}_1$ to
$\overline{R}_1$, and any path $\gamma_2\subseteq R_2$ from
$\underline{R}_2$ to
$\overline{R}_2$, we will consider the following events:
\begin{itemize}
	\item $\mathcal{P}^{\gamma_1}$: the event that any path $\gamma\subseteq R_1$ from $\underline{R}_1$ to
		$\overline{R}_1$ disjoint from $\gamma_1$ satisfies
		$\widetilde{\underline{\ell}}(\gamma)\leq -M^4 n^{1/3}$.
	\item $\mathcal{P}^{\gamma_2}$: the event that any path $\gamma\subseteq R_2$ from
		$\underline{R}_2$ to $\overline{R}_2$ disjoint from $\gamma_2$ satisfies
		$\widetilde{\underline{\ell}}(\gamma)\leq -M^4 n^{1/3}$.
\end{itemize}
In the above setting, we define the event
$\mathcal{P}^{\gamma_1,\gamma_2}$ by
\begin{displaymath}
	\mathcal{P}^{\gamma_1,\gamma_2}=\mathcal{P}^{\gamma_1}\cap\mathcal{P}^{\gamma_2}.
\end{displaymath}
 Note that the event $\mathcal{P}^{\gamma_1,\gamma_2}$ depends only on the vertex weights
	inside $(R_1\setminus 
	(\gamma_1 \cup \overline{R}_1) )\cup (R_2\setminus  (\gamma_2\cup
	\overline{R}_2)
	)$. 

Note that by using \eqref{eqn:14.17.0} along with Lemma \ref{mod3.2}, we have that once we fix some large
constant $M$, then for some
constant $c$ not depending on $\gamma_1$ or $\gamma_2$, we have that the
probabilities of each of the above events are lower bounded by some
constant $c$ for all $n$ large enough. Using the
independence between the vertex weights in $R_1$ and $R_2$, we have that $\mathcal{P}^{\gamma_1}$ and $\mathcal{P}^{\gamma_2}$ are
independent and this
immediately implies the following lemma:

\begin{lemma}
	\label{one3.1}
	For any path $\gamma_1\subseteq R_1$ from $\underline{R}_1$ to
	$\overline{R}_1$ and any path $\gamma_2\subseteq R_2$ from
	$\underline{R}_2$ to $\overline{R}_2$, we
	have that for any $M>0$, there exists a
	positive constant $\beta$  (depending on the choice of $M$, independent of
	the choice of $\gamma_1,\gamma_2$) such that for all $n$ large
	enough,
	\begin{equation}
	\label{eqn:23}
	\mathbb{P}\left( \mathcal{P}^{\gamma_1,\gamma_2} \right)\geq \beta.
\end{equation}
\end{lemma}
\begin{proof}
	The lemma follows by using \eqref{eqn:14.17.0}, Lemma \ref{mod3.2}, and the
	independence of $\mathcal{P}^{\gamma_1}$ and
	$\mathcal{P}^{\gamma_2}$ as explained above. Indeed, by
	\eqref{eqn:14.17.0}, we have that for any path $\gamma\subset R_1$
	from $\underline{R}_1$ to $\overline{R}_1$, and for some
	positive constants $C_1,C_2$ not depending on $M$,
	\begin{displaymath}
		\left\{ \widetilde{\underline{\ell}}(\gamma)\leq -M^4 n^{1/3}
		\right\}\supseteq \left\{ \underline{\ell}(\gamma)-\epsilon n\leq (-C_1-1)M^4
		n^{1/3}-C_2n^{1/3} \right\},
	\end{displaymath}
	and the $M$-dependent lower bound for $\mathbb{P}(\mathcal{P}^{\gamma_1})$ now follows
	by using Lemma \ref{mod3.2} for the rectangle $R_1$. Similar
	considerations for $R_2$ yield the lower
	bound for $\mathbb{P}(\mathcal{P}^{\gamma_2})$.
\end{proof}

Again, note that the $\beta$ and $M$ in this section have no relation to the
ones in Section \ref{ss:constants}.

\begin{proposition}
	\label{one4}
	For any $M>0$ fixed, let $\beta=\beta(M)$ be the constant obtained from Lemma \ref{one3.1}. For any path $\gamma_1\subseteq R_1$ from $\underline{R}_1$ to
	$\overline{R}_1$ and any path $\gamma_2\subseteq R_2$ from
	$\underline{R}_2$ to $\overline{R}_2$, we have
	\begin{displaymath}
		\mathbb{P}\left( \mathcal{P}^{\gamma_1,\gamma_2}| \mathcal{S}'\cap \mathbf{Restr}\cap \mathcal{W}\cap  \left\{
		\left.\Gamma_n\right|_{R_1}=\gamma_1 \right\}\cap \left\{
		\left.\Gamma_n\right|_{R_2}=\gamma_2
		\right\} \right)\geq \beta
	\end{displaymath}
	for all $n$ large enough.
\end{proposition}

\begin{proof}
	Note that after conditioning on the configuration in $\left((R_1\setminus 
	(\gamma_1 \cup \overline{R}_1) )\cup (R_2\setminus  (\gamma_2\cup
	\overline{R}_2)
	)\right)^c$, we have that
	$\mathcal{S}',\mathbf{Restr},\mathcal{P}^{\gamma_1,\gamma_2}$ and $\left\{
		\left.\Gamma_n\right|_{R_1}=\gamma_1 \right\}\cap \left\{
		\left.\Gamma_n\right|_{R_2}=\gamma_2
		\right\}$ are decreasing events on $(R_1\setminus 
	(\gamma_1 \cup \overline{R}_1) )\cup (R_2\setminus  (\gamma_2\cup
	\overline{R}_2)
	)$.  Let $\mathscr{F}_{\gamma_1,\gamma_2}$ denote the
	$\sigma$-algebra generated by the vertex weights in $\left((R_1\setminus 
	(\gamma_1 \cup \overline{R}_1) )\cup (R_2\setminus  (\gamma_2\cup
	\overline{R}_2)
	)\right)^c$. Using the FKG inequality and the dependence of
	$\mathcal{P}^{\gamma_1,\gamma_2}$ only on the vertex weights in $(R_1\setminus 
	(\gamma_1 \cup \overline{R}_1) )\cup (R_2\setminus  (\gamma_2\cup
	\overline{R}_2)
	)$
	(thus,
	$\mathcal{P}^{\gamma_1,\gamma_2}$ is independent of
	$\mathscr{F}_{\gamma_1,\gamma_2}$), we have that
	\begin{equation}
		\label{eqn:24}
		\begin{split}
			\mathbb{P}\left( \mathcal{P}^{\gamma_1,\gamma_2} \cap \mathcal{S}' \cap \mathbf{Restr} \cap  \left\{
		\left.\Gamma_n\right|_{R_1}=\gamma_1 \right\}\cap \left\{
		\left.\Gamma_n\right|_{R_2}=\gamma_2
		\right\} | \mathscr{F}_{\gamma_1,\gamma_2}\right) \geq\\ \mathbb{P}\left( \mathcal{P}^{\gamma_1,\gamma_2}
	\right) 	\mathbb{P}\left( \mathcal{S}'\cap \mathbf{Restr} \cap  \left\{
		\left.\Gamma_n\right|_{R_1}=\gamma_1 \right\}\cap \left\{
		\left.\Gamma_n\right|_{R_2}=\gamma_2
		\right\} | \mathscr{F}_{\gamma_1,\gamma_2} \right).
	\end{split}
	\end{equation}
	Noting that the event $\mathcal{W}$ is measurable with respect to
	$\mathscr{F}_{\gamma_1,\gamma_2}$, we have
	\begin{align}
		\label{eqn:25}
			&\mathbb{P}\left( \mathcal{P}^{\gamma_1,\gamma_2} \cap \mathcal{S}'\cap \mathbf{Restr} \cap  \left\{
		\left.\Gamma_n\right|_{R_1}=\gamma_1 \right\}\cap \left\{
		\left.\Gamma_n\right|_{R_2}=\gamma_2
		\right\} | \mathcal{W}\right) \nonumber \\ &\geq  \mathbb{P}\left( \mathcal{P}^{\gamma_1,\gamma_2}
	\right) 	\mathbb{P}\left( \mathcal{S}'\cap \mathbf{Restr} \cap  \left\{
		\left.\Gamma_n\right|_{R_1}=\gamma_1 \right\}\cap \left\{
		\left.\Gamma_n\right|_{R_2}=\gamma_2
		\right\} | \mathcal{W} \right)
	\end{align}
	and this in turn yields
	\begin{equation}
		\label{eqn:26}
		\mathbb{P}\left( \mathcal{P}^{\gamma_1,\gamma_2}| \mathcal{S}'\cap \mathbf{Restr}\cap \mathcal{W} \cap  \left\{
		\left.\Gamma_n\right|_{R_1}=\gamma_1 \right\}\cap \left\{
		\left.\Gamma_n\right|_{R_2}=\gamma_2
		\right\} \right)\geq \mathbb{P}\left( \mathcal{P}^{\gamma_1,\gamma_2} \right)\geq
		\beta,
	\end{equation}
thereby completing the proof.
\end{proof}

\subsubsection{Forcing the geodesics to coalesce on the favourable events}
We shall now show that the positive probability event constructed above
implies that $\Gamma_{a_1,b_1}$ and $\Gamma_{a_2,b_2}$ coincide between
$\overline{R}_1$ and $\underline{R}_2$.

\begin{proposition}
	\label{one5}
	Take any $(\gamma_1,\gamma_2)\in \mathscr{I}$. On the event
	$\mathcal{P}^{\gamma_1,\gamma_2}\cap \mathcal{S}'\cap \mathbf{Restr} \cap
	\mathcal{W}$, we have that for $M$ large enough, and $n$ large
	depending on $M$:
	$\Gamma_{a_1,b_1},\Gamma_{a_2,b_2}\subseteq U_{M^2}$, $\Gamma_n\subseteq
	U_{2M}$ and both
	$\Gamma_{a_1,b_1}$ and $\Gamma_{a_2,b_2}$ intersect both $\gamma_1$ and
	$\gamma_2$.
\end{proposition}
\begin{proof}
	Let $u_1,v_1$ (resp.\ $u_2,v_2$) be the starting and ending points of $\gamma_1$ (resp.\ $\gamma_2$). Let 
	$\chi_1$ be the
	concatenation of the paths $\Gamma_{a_1,u_1},\gamma_1,\Gamma_{v_1,u_2},\gamma_2$ and
	$\Gamma_{v_2,b_1}$. By using that the paths $\gamma_1$ and
	$\gamma_2$ are $M^{3/4}$-typical along with the definition of the events $\mathbf{Restr}$
	and $\mathcal{W}$, we have that for some absolute constant $c$,
	\begin{equation}
		\label{eqn:28}
		\ell(\chi_1)-4n\geq -c M^2 n^{1/3}.
	\end{equation}
	Here, we have used that
	$\ell(\chi_1)=\ell(\Gamma_{a_1,u_1})+\underline{\ell}(\gamma_1)+
	\underline{\ell}(\Gamma_{v_1,u_2})+\underline{\ell}(\gamma_2)+\underline{\ell}(\Gamma_{v_2,b_1})$. The
	recentering in \eqref{eqn:28} with respect to $4n$ is done by
	using \eqref{eqn:14.17.0} for the endpoints of each of the five
	segments making up $\chi_1$, and using that all the endpoints lie
	inside $U_{2M}$. Indeed, this recentering leads to the $M^2$ term in
	\eqref{eqn:28}. Now, using that $\xi M^4/2> cM^2$ for large $M$, we have
	that $\chi_1\subseteq U_{M^2}$ and consequently
	$\Gamma_{a_1,b_1}\subseteq U_{M^2}$. By a symmetric argument, one also
	obtains that $\Gamma_{a_2,b_2}\subseteq U_{M^2}$.
	
	Now,
	let $\chi_2$ be the concatenation of the paths
	$\Gamma_{\mathbf{0},u_1},\gamma_1,\Gamma_{v_1,u_2},\gamma_2$ and
	$\Gamma_{v_2,\mathbf{n}}$. By using that $\gamma_1$ and $\gamma_2$ are
$M^{3/4}$-typical along with the definition of the events $\mathbf{Restr}$
	and $\mathcal{W}$, we have that for some absolute constant $c_1$, 
	\begin{equation}
		\label{eqn:28.1}
		\ell(\chi_2)-4n \geq -c_1 M^{3/2} n^{1/3}.
	\end{equation}
	Again, we have used \eqref{eqn:14.17.0} for the endpoints of the
	five segments of $\chi_2$ along with
	$\ell(\chi_2)=\ell(\Gamma_{\mathbf{0},u_1})+\underline{\ell}(\gamma_1)+
	\underline{\ell}(\Gamma_{v_1,u_2})+\underline{\ell}(\gamma_2)+\underline{\ell}(\Gamma_{v_2,\mathbf{n}})$. Here, we obtained $M^{3/2}$ since
	$\mathbf{0},u_1,v_1,u_2,v_2,\mathbf{n}\subseteq U_{M^{3/4}}$ and this is
	because
	$\gamma_1$ and $\gamma_2$ are $M^{3/4}$-typical as in the definition of
	$\mathscr{I}$. Using that $\xi(2M)^2/2> c_1 M^{3/2}$ for large $M$, and the definition of the event
	$\mathbf{Restr}_{\mathbf{0},2M}\supseteq \mathbf{Restr}$, we
	obtain that	
	$\Gamma_n\subseteq U_{2M}$.
	
	Now, consider any path
	$\gamma\subseteq U_{M^2}$ from $a_1$ to $b_1$ which does not intersect
	{at least one of $\gamma_1$ and $\gamma_2$}. By the definition of the event
	$\mathcal{P}^{\gamma_1,\gamma_2}$, $\mathcal{W}$ and $\mathcal{S}'$, we have that 
	\begin{equation}
		\label{eqn:29}
		\ell(\gamma)-4n\leq 4M n^{1/3}  -M^{4}n^{1/3} +c_2n^{1/3}.
	\end{equation}
	Here, one divides $\gamma$ into five segments which we call
	$\sigma_1,\sigma_2,\dots,\sigma_5$ and uses that
	$\ell(\gamma)=\ell(\sigma_1)+\underline{\ell}(\sigma_2)+
	\underline{\ell}(\sigma_3)+\underline{\ell}(\sigma_4)+\underline{\ell}(\sigma_5)$.
	The recentering in \eqref{eqn:29} is done by using the right
	hand side part of \eqref{eqn:14.17.0} and this leads to $c_2$, a
	positive constant. Indeed, one uses
	{\eqref{eqn:14.17.0}} once for each segment of $\gamma$, and $c_2$ is the
	cumulative effect of these five applications. By noting that
	$M^4-4M-c_2>c M^2$ for large $M$ and by using
	\eqref{eqn:28} along with \eqref{eqn:29}, we have that $\Gamma_{a_1,b_1}$
	must intersect $\gamma_1$ and $\gamma_2$ both. By
	entirely analogous arguments, one observes that $\Gamma_{a_2,b_2}$ must
	intersect $\gamma_1$ and $\gamma_2$ both; see Figure \ref{fig:coalesce} for an illustration of this argument.
\end{proof}

\begin{figure}[t]
\begin{center}
	\includegraphics[width=0.5\linewidth]{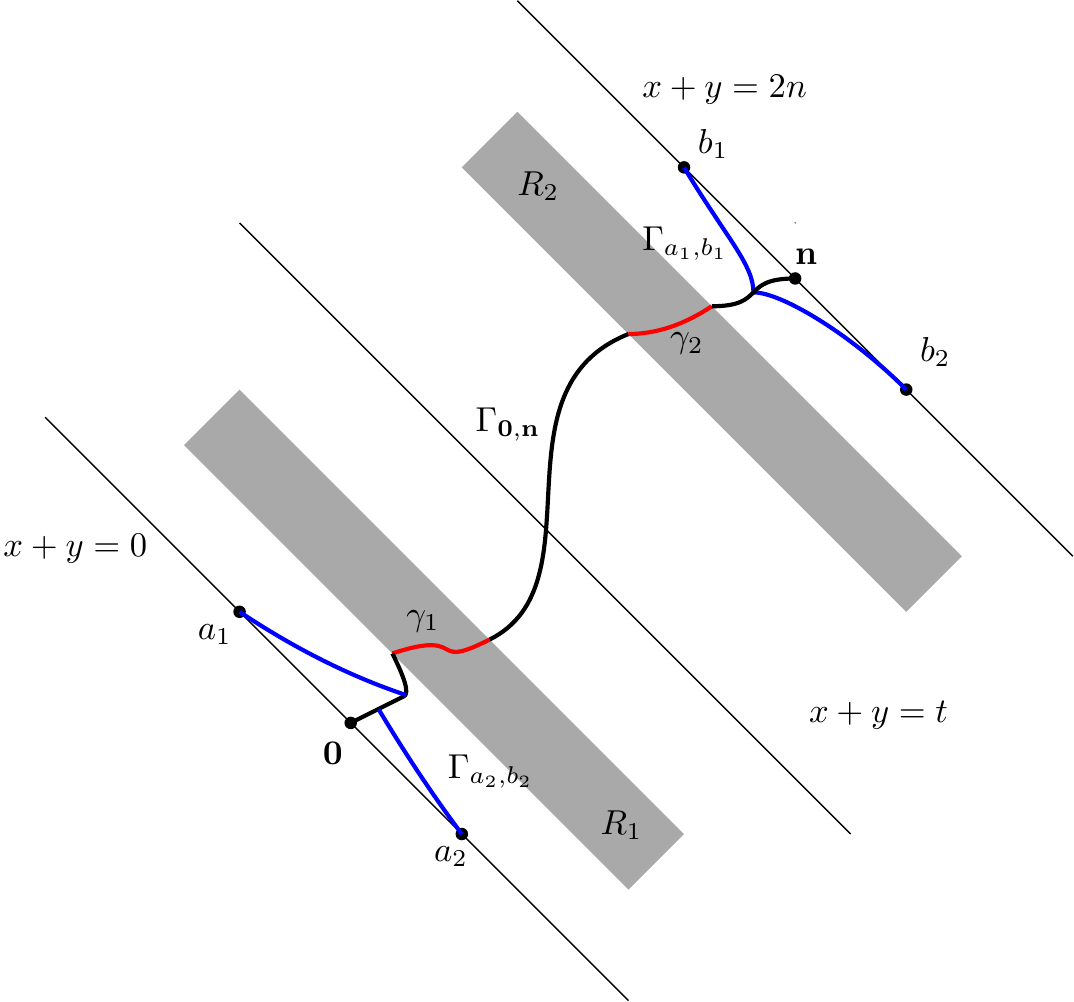}
	\end{center}
	\caption{Coalescence of geodesics $\Gamma_{a_1,b_1}, \Gamma_{a_2,b_2}$  on the favourable events: we show that for a pair of paths $\gamma_1$ and $\gamma_2$ across $R_1$ and $R_2$ respective satisfying the required regularity conditions, on the intersection of event $\mathcal{P}^{\gamma_1,\gamma_2}$ and the typical event $\mathcal{S}'\cap \mathbf{Restr} \cap
	\mathcal{W}$, $\Gamma_{a_1,b_1}$ and $\Gamma_{a_2,b_2}$ each intersect
	both $\gamma_1$ and $\gamma_2$. This is done by using the definition of
	$\mathcal{P}^{\gamma_1,\gamma_2}$ which ensures that any path crossing
	$R_1$ or $R_2$ and disjoint with $\gamma_1$ or $\gamma_2$ respectively
	incurs a heavy penalty in weight. Finally, the event
	$\{\left.\Gamma_{n}\right|_{R_1}=\gamma_1\} \cap \{ \left.\Gamma_{n}\right|_{R_2}=\gamma_2\}$ ensures the required coalescence.}
	\label{fig:coalesce}
\end{figure}

We are finally ready to prove Proposition \ref{one0}.

\begin{proof}[Proof of Proposition \ref{one0}]
	We first show that we can fix $M$ large enough such that if we take any $(\gamma_1,\gamma_2)\in \mathscr{I}$, we have 
		\begin{equation}
			\label{eqn:29.1}
			\mathcal{P}^{\gamma_1,\gamma_2}\cap \mathcal{S}'\cap \mathbf{Restr} \cap \mathcal{W} \cap \left\{
		\left.\Gamma_n\right|_{R_1}=\gamma_1 \right\}\cap \left\{
		\left.\Gamma_n\right|_{R_2}=\gamma_2
		\right\} \subseteq 
		 \left\{\Gamma_{a_1,b_1}(t)=
		\Gamma_{a_2,b_2}(t)\right\} \cap \left\{ |\Gamma_{a_1,b_1}(t)|\leq
		2Mn^{2/3}\right\} 	
		\end{equation}
		for all $n$ large enough. 
		As mentioned in Section \ref{ss:constants1}, fix $M$ large so as to satisfy the
		conclusions of Lemma \ref{one0.2}, Lemma \ref{one1} and
Proposition \ref{one5} for all large $n$. By using Proposition \ref{one5}, we have that $\Gamma_{a_1,b_1}$ intersects
	$\gamma_1$ and $\gamma_2$. Since $\Gamma_n$ is a geodesic, this implies
	that $\Gamma_{a_1,b_1}=\Gamma_n$ in the region between $\overline{R}_1$ and
	$\underline{R}_2$. Using an analogous argument, we have that
	$\Gamma_{a_1,b_1}=\Gamma_{a_2,b_2}=\Gamma_n$ in the region between  $\overline{R}_1$ and
	$\underline{R}_2$, and note that the line  $\mathbb{L}_t$ lies in this region. Since
	$\Gamma_n\subseteq U_{2M}$ by Proposition \ref{one5}, it is clear that
	$ |\Gamma_n(t)|\leq
	2Mn^{2/3}$. This proves \eqref{eqn:29.1}.
	
	We now use Lemma \ref{one3.1} to obtain $\beta=\beta(M)$ for the $M$ which was
	just fixed. Now, by using \eqref{eqn:29.1}, we have that for any
	$(\gamma_1,\gamma_2)\in \mathscr{I}$,
	\begin{align}
		\label{eqn:30}
		&\mathbb{P}\left(  \left\{\Gamma_{a_1,b_1}(t)=
		\Gamma_{a_2,b_2}(t)\right\} \cap \left\{ |\Gamma_{a_1,b_1}(t)|\leq
		2Mn^{2/3}\right\}| \left\{\left.\Gamma_n\right|_{R_1}=\gamma_1 \right\}\cap \left\{
		\left.\Gamma_n\right|_{R_2}=\gamma_2
		\right\} \right)\nonumber \\&\geq \mathbb{P}\left( 	\mathcal{P}^{\gamma_1,\gamma_2}\cap
		\mathcal{S}'\cap \mathbf{Restr} \cap \mathcal{W}|   \left\{\left.\Gamma_n\right|_{R_1}=\gamma_1 \right\}\cap \left\{
		\left.\Gamma_n\right|_{R_2}=\gamma_2
		\right\} \right)\nonumber\\
		&\geq 	\mathbb{P}\left( \mathcal{P}^{\gamma_1,\gamma_2}| \mathcal{S}'\cap \mathbf{Restr}\cap \mathcal{W}\cap  \left\{
		\left.\Gamma_n\right|_{R_1}=\gamma_1 \right\}\cap \left\{
		\left.\Gamma_n\right|_{R_2}=\gamma_2
		\right\} \right) \mathbb{P}\left( \mathcal{S}'\cap \mathbf{Restr}\cap \mathcal{W}|
\left.\Gamma_n\right|_{R_1}=\gamma_1, \left.\Gamma_n\right|_{R_2}=\gamma_2
\right) \nonumber\\ & \geq \beta/2. 
	\end{align}
	Note that we used Proposition \ref{one4} and the definition of $\mathscr{I}$ (as
	in \eqref{eqn:14.26}) to obtain
the last inequality. Now, observe that
\begin{align}
	\label{eqn:31}
	&\mathbb{P}\left( \left\{\Gamma_{a_1,b_1}(t)=
		\Gamma_{a_2,b_2}(t)\right\} \cap\left\{ |\Gamma_{a_1,b_1}(t)|\leq
		2Mn^{2/3}\right\}  \right)\nonumber\\
		&\geq \mathbb{P}\left( \left\{\Gamma_{a_1,b_1}(t)=
		\Gamma_{a_2,b_2}(t)\right\} \cap\left\{ |\Gamma_{a_1,b_1}(t)|\leq
		2Mn^{2/3}\right\} \cap  \left\{(\left.\Gamma_n\right|_{R_1},
		\left.\Gamma_n\right|_{R_2})\in \mathscr{I}
		\right\}  \right)\nonumber\\
		&=\sum_{(\gamma_1,\gamma_2)\in \mathscr{I}} \bigg[ \mathbb{P}\left(  \left\{\Gamma_{a_1,b_1}(t)=
		\Gamma_{a_2,b_2}(t)\right\} \cap \left\{ |\Gamma_{a_1,b_1}(t)|\leq
		2Mn^{2/3}\right\}| \left\{\left.\Gamma_n\right|_{R_1}=\gamma_1 \right\}\cap \left\{
		\left.\Gamma_n\right|_{R_2}=\gamma_2
		\right\} \right)\nonumber\\
		&\qquad \qquad \qquad \mathbb{P}\left( (\left.\Gamma_n\right|_{R_1},
		\left.\Gamma_n\right|_{R_2})=(\gamma_1,\gamma_2)
		\right)\bigg]\nonumber\\
		&\geq \frac{\beta}{2}\mathbb{P}\left( (\left.\Gamma_n\right|_{R_1},
		\left.\Gamma_n\right|_{R_2})\in \mathscr{I}  \right)\nonumber\\
		&\geq \frac{0.9\beta}{2}.
\end{align}
Note that we have used \eqref{eqn:30} and Lemma \ref{one3} to obtain the last
two inequalities. This completes the proof of the proposition.
\end{proof}


\section{Concluding remarks and possible extensions}
\label{s:dis}
As we have remarked throughout, our objective in this paper was to focus on
one of the simplest settings to maintain maximum clarity of exposition. However, we
expect that the methods illustrated here have broader applicability and we conclude with a discussion of possible extensions of the results presented in this paper. We shall not attempt to make this discussion precise; working out these details will be taken up elsewhere. 

There are primarily two directions of possible generalizations we will discuss: the first one will focus on applicability of our results beyond the exponential LPP model. The other will focus on the set up of exponential LPP itself but will look at the geometry of parts of the geodesics at macroscopic or mesoscopic scales.

\subsection{Beyond exponential LPP}
As alluded to before, we recall the reader's attention to the fact that even
though we worked with the specific model of exponential LPP, our arguments
depended only on the one-point estimates (Proposition \ref{mod1}) and its
consequences about passage times across the parallelogram (primarily Proposition
\ref{mod2}, Proposition \ref{mod3}) together with some basic
tools of percolation like the FKG inequality. An attempt to formalize this
axiomatic study of last passage percolation on $\mathbb{Z}^2$ was made in
\cite{BGHH20}, and one expects that the results in this paper will continue to hold
under the same set of assumptions (see Section 1.1, page 5 and Appendix A in
\cite{BGHH20}). In particular, the analogues of Proposition \ref{mod2} and
Proposition \ref{mod3} for Geometric LPP were verified  in \cite{BGHH20} based
on the one point convergence and moderate deviation estimates from \cite{Jo00,
BDMMZ02} (see Section B.1 and Section B.3 in
\cite{BGHH20}). Hence, one expects straightforward modifications of Theorem \ref{t:sb} to
hold for Geometric LPP, after one modifies the statement to deal with the
non-uniqueness of the geodesic (by e.g.\ fixing the left most geodesic). Note
that our proof of Theorem \ref{t:onepoint} used the uniqueness of the geodesic
in a slightly more crucial way, and hence a direct adaptation of the same to
the geometric LPP setting would not work. Even though we believe that an
appropriate variant of Theorem \ref{t:onepoint} would hold for geometric LPP, we shall not comment on this here. 


One point convergence and moderate deviation estimates are known for two non-lattice
exactly solvable models of planar last passage percolation as well. The first of these is Poissonian
LPP on $\R^2$, where the underlying randomness is a homogeneous rate one
Poisson point process on $\R^2$ and the last passage time between two ordered
points is the maximum number of Poisson points that can be collected in an
up/right journey from the ``smaller" to the ``larger" point. The analogue of Proposition \ref{mod1} for Poissonian LPP was established in \cite{LM01, LMS02} and using these, the parallelogram estimates (analogues of Proposition \ref{mod2}, Proposition \ref{mod3}) were established in \cite{BSS14}. Using these, we expect that our arguments could be used to extend Theorem \ref{t:sb} to Poissonian LPP as well (dealing with the non-uniqueness issue as before).

The second model, the semi-discrete Brownian LPP is of particular interest to us, as this is the only model of
planar LPP for which convergence to the Directed Landscape has been rigorously
established so far. Let us define this model precisely. Let
$\{B_i(\cdot)\}_{i\in \Z}$ denote a sequence of two sided standard Brownian
motions on $\R$. For a non-decreasing function $\phi: [0,n]\to \{0,1,2,\ldots,
n\}$ with $\phi(0)=0$ and $\phi(n)=n$, let us define 
$E(\phi):=\sum_{i=0}^{n} B_{i}(\phi_{i+1})-B_i(\phi_{i})$ where $[\phi_i,\phi_{i+1}]$ is the smallest closed interval containing the set $\{x:\phi(x)=i\}$. The last passage time from $(0,0)$ to $(n,n)$ is defined to the maximum of $E(\phi)$ over all $\phi$. Let $\Pi_{n}$ denote the (almost surely unique) function $\phi$ which attains the last passage time (geodesic).

Using the correspondence between passage times in Brownian
LPP and the largest eigenvalue of Gaussian Unitary Ensemble (GUE)
\cite{Bar01}, the one point estimates for passage times in
Brownian LPP can be obtained from \cite{LR10} (see also \cite{Aub04}), and the
convergence to the GUE Tracy-Widom distribution is proved in \cite{GTW00}. An inspection of the arguments in
\cite{BSS14,BGZ19} shows that the arguments are sufficiently robust and uses
only the curvature of limit shape and one point moderate deviation estimates
and hence one expects to establish analogues of Proposition \ref{mod2} and
Proposition \ref{mod3} for Brownian last passage percolation as well. Although
the parallelogram estimates for Brownian LPP has not been worked out anywhere
in the literature as far as we are aware, some similar estimates have appeared
in the works \cite{Ham16, Ham17a, Ham17b, DOV18, DV18, DSV20, SS20,GH20} using the
Brownian Gibbs property of a line ensemble associated to Brownian motion. Even
though we will not attempt to provide a proof of any of these results, we
shall state the following  precise analogue of Theorem \ref{t:sb} as a
conjecture that we believe can be proved by adapting the arguments presented in this paper. 


\begin{conjecture}
\label{con:blpp}
There exists $\delta>0$ and positive constants $C_1,c_1,C_2,c_2$ such that for all
$\delta<\delta_0$, we have that for all $n$ large depending on $\delta$,
$$C_2e^{-c_2\delta^{-3/2}}\leq \P\left( \sup_{s\in (0,1)}n^{-2/3}|\Pi_{n}(ns)-ns| \leq \delta \right) \leq C_1e^{-c_1\delta^{-3/2}}.$$
\end{conjecture}
As geodesics in Brownian LPP are almost surely unique, one expects the arguments proving Theorem \ref{t:onepoint} to go through in this case as well. 

\subsubsection{Geodesics in the directed landscape}

It was shown in \cite[Theorem 1.1]{DOV18} that $n^{-2/3}|\Pi_{n}(ns)-ns|$
converges almost surely to a continuous random function $\Pi$ on $[0,1]$. The
directed landscape $\mathcal{L}(x,s;y,t)$ is a four parameter random field
defined on  on the half space $s<t$ of $\R^4$ such that $\mathcal{L}(x,s;y,t)$
is the scaling limit (as $n\to \infty$) of the centered and scaled passage
times from $(x,s)$ to $(y,t)$ under the affine spatial scaling that keeps the
origin fixed and takes the point $(n+yn^{2/3},n)$ to the point $(y,1)$. We do
not need the precise definition of the directed landscape, but we remark that
it was shown in \cite{DOV18} that $\Pi$ above is the geodesic in the directed
landscape from $(x,s)=(0,0)$ to $(y,t)=(0,1)$. Geometric properties of $\Pi$,
working directly with the directed landscape, have been studied in
\cite{DOV18,DSV20} where an analogue of \eqref{e:tfub} was proved and an expression for the $3/2$ variation of $\Pi$ was obtained. 

One can also attempt to study the geometry of $\Pi$ by studying finite
geodesics in Brownian LPP. For example, the proof of Corollary \ref{c:as}
together with Conjecture \ref{con:blpp} would show that Corollary \ref{c:as}
(i) remains valid with $\pi$ replaced by $\Pi$ as above, and one expects a
similar reasoning to yield a variant of  Corollary \ref{c:as} (ii), from the
appropriate variant of Theorem \ref{t:onepoint} proved for Brownian LPP. We
also remark the the directed landscape is expected to be universal and it is
believed that one could construct the same object by taking a suitable space time scaling of the exponential LPP. If such a result is established, Corollary \ref{c:as} would directly apply to the geodesic $\Pi$ in the directed landscape. 


\subsection{Geometry of geodesics at a finer Scale}
\label{s:disf}
Here we shall bring our focus back to the exponential LPP model and discuss the applicability of our results to (i) parts of the geodesic $\Gamma_{n}$ (both macroscopic and mesoscopic), (ii) the scenario when $\delta$ is allowed to go to $0$ with $n$.

\subsubsection{Macroscopic segments of the geodesic}
Notice that in Theorem \ref{t:sb}, we only considered the small ball
probability for the whole geodesic $\Gamma_{n}$. However, by following our
arguments, one can easily also derive the same result for any macroscopic segment of the
geodesic. More precisely, for $\epsilon>0$ fixed, let $\llbracket t_1,t_2 \rrbracket$ denote a sub-interval of $\llbracket 0,2n \rrbracket$ such that $t_2-t_1\geq \epsilon n$. We have the following analogue of Theorem \ref{t:sb}:
\begin{equation}
\label{e:inter}
C_2e^{-c_2\delta^{-3/2}}\leq \P\left(\sup_{t\in \llbracket t_1,t_2
\rrbracket}|\Gamma_n(t)|\leq \delta n^{2/3}\right)\leq C_1e^{-c_1\delta^{-3/2}}.
\end{equation}
Notice that the lower bound here is immediate from Theorem \ref{t:sb} whereas
for the upper bound, one needs to redo the argument restricted to the interval $\llbracket t_1, t_2\rrbracket$.

\subsubsection{Mesoscopic Segments of the Geodesic at either end}
A more interesting question is to consider the segment of geodesic
$\Gamma_{n}$ restricted to the interval $\llbracket 0,r \rrbracket$ or
$\llbracket 2n-r,2n \rrbracket$ for some $1\ll r\ll n$. One of the advantages
of working with a pre-limiting model, rather than a limiting model such as the
directed landscape, is that these mesoscopic statistics cannot be read off from
the limiting model.  It is known that the transversal fluctuation of
$\Gamma_n$ at scale $r$ is $O(r^{2/3})$ (see \cite[Theorem 3]{BSS17B}), hence
the natural question is to ask for the probability that $\sup_{t\in [0,r]}
|\Gamma_n(t)|\leq \delta r^{2/3}$. We believe that the argument in this paper
together with \cite[Theorem 3]{BSS17B} and \cite[Theorem 3]{BG18} can be used
to show that for $r$ sufficiently large and $\delta$ small, we have 
\begin{equation}
\label{e:meso}
C_2e^{-c_2\delta^{-3/2}}\leq \mathbb{P}\left(\sup_{t\in [0,r]} |\Gamma_n(t)|\leq \delta
r^{2/3}\right)\leq C_1e^{-c_1\delta^{-3/2}}.
\end{equation}
One also expects a similar estimate to hold for the semi-infinite geodesic from $\mathbf{0}$ in the direction $(1,1)$. 

The analogue of Theorem \ref{t:onepoint} is expected to hold at the scale $r\ll n$ as well. In particular, one expects that appropriate modifications of our estimates will yield that 
\begin{equation}
\label{e:mesoonept}
c\delta \leq  \P(|\Gamma_n(r)|\leq \delta r^{2/3})\leq C\delta.
\end{equation}
The translation invariance in the proof of the upper bound in
Theorem \ref{t:onepoint} will directly give the upper bound in
\eqref{e:mesoonept}, but the lower bound requires significant modifications
in the argument for it to work and would be taken up elsewhere. 

\subsubsection{The case of vanishing $\delta$ and small deviations away from the diagonal}
Finally, we want to point out that even though for the sake of notational
convenience we have stated our results for a fixed but small $\delta$ while
letting $n$ become arbitrarily large, our arguments are sufficiently robust to
handle the case when $\delta$ is allowed to go to $0$ with $n$ sufficiently
slowly. This is rather transparent for Theorem \ref{t:onepoint}, where the key
estimates Proposition \ref{one-2} and Proposition \ref{one0} did not depend on
$\delta$ at all and the role of $\delta$ was merely in setting up the
appropriate translations. A moment's thought should convince the reader that
Theorem \ref{t:onepoint} holds for all $\delta$ such that $\delta n^{2/3} \geq
1$ which ensures that the translations can be made sense of in the lattice. 

{As already mentioned in the introduction, Theorem \ref{t:onepoint} can
also be strengthened by considering $\P(\Gamma_{n}(\cdot)\in I)$ for any
compact interval $I$ of length $\delta n^{2/3}\geq 2$ (the lower bound is imposed to make sure that $\{\Gamma_{n}(\cdot)\in I\}$ is not vacuously empty). Indeed, for any $L>0$ and $f\in
\llbracket-Ln^{2/3},Ln^{2/3}\rrbracket$, it is easy to check that for $t$ even (an analogous statement holds for $t$ odd)
$\P(\Gamma_{n}(t)=2f)=\P(\Gamma_{u_i,u_{i}+\mathbf{n}}(t)-\psi(u_i)=2f)$ for each
$i\in \Z$ where $u_i=(i,-i)$. One can show that $\sum_{|i|\leq L' n^{2/3}}
\P(\Gamma_{u_i,u_{i}+\mathbf{n}}(t)-\psi(u_i)=2f)$ is bounded away from $0$ and
$\infty$ uniformly in $n$ for every fixed $L'$ sufficiently larger compared to
$L$. Indeed, one observes that the argument in Lemma \ref{one-2.1} remains
valid if one changes the definition of $\mathcal{L}_{t}$ to include all points
in $\L_{t}\cap U_{L'}$ and Proposition \ref{one0} is true for arbitrarily
large choices of $M$. This estimate, together with repeating the proofs of the
upper and lower bounds in the proof of Theorem \ref{t:onepoint}, gives the following corollary.}

{
\begin{corollary}
\label{c:lbgen}
For each $L\geq 0$ and for all $\epsilon\in (0,1)$, there exist positive constants $C_3,c_3$ depending
on $\epsilon$ and $L$ such that for
all $n\geq n_0(\epsilon,L)$ and $t\in \llbracket \epsilon n,
(2-\epsilon)n \rrbracket$, and for all intervals $I\subseteq
[-Ln^{2/3},Ln^{2/3}]$ with $|I|=\delta n^{2/3}\geq 2$, we have
\begin{displaymath}
	c_3\delta \leq \P\left (\Gamma_n(t)\in I\right) \leq C_3\delta.
\end{displaymath}
\end{corollary}}

Working through the steps of the proof of Theorem \ref{t:sb} in the case of $\delta\to 0$ requires a little
more work, but observe that whenever we have applied estimates like Proposition \ref{mod2} or Proposition \ref{mod3} to a rectangle or parallelogram whose dimensions involved $\delta$, it was applied to a parallelogram of size $O(\delta^{3/2}n) \times \delta n^{2/3}$. Application of these parallelogram estimates only require that the dimensions of the parameters be sufficiently large and hence it is expected that the proofs will all go through as long as $\delta n^{2/3}\to \infty$. 


\bibliography{delocalization}
\bibliographystyle{plain}

\end{document}